%% file: operator_final.tex
\documentclass[11pt,oneside,english,reqno]{amsart}
\usepackage[utf8]{inputenc}
\usepackage[T1]{fontenc}
\usepackage{color}
\usepackage{babel}
\usepackage{amstext}
\usepackage{float}
\usepackage{amsthm}
\usepackage{amssymb}
\usepackage{graphicx}
\usepackage{amsmath}
\usepackage[unicode=true,pdfusetitle,
 bookmarks=true,bookmarksnumbered=false,bookmarksopen=false,
 breaklinks=false,pdfborder={0 0 0},pdfborderstyle={},backref=false,colorlinks=true]
 {hyperref}
\hypersetup{
 citecolor=blue}

\usepackage{geometry}
\makeatletter
\numberwithin{equation}{section}
\numberwithin{figure}{section}
\theoremstyle{definition}
\newtheorem{thm}{\protect\theoremname}[section]
  \theoremstyle{definition}
  
  \theoremstyle{definition}
  \newtheorem{defn}[thm]{\protect\definitionname}
  \theoremstyle{definition}
  \newtheorem{lem}[thm]{\protect\lemmaname}
  \theoremstyle{remark}
  \newtheorem{rem}[thm]{\protect\remarkname}
  \theoremstyle{definition}
  \newtheorem{cor}[thm]{\protect\corollaryname}
  \theoremstyle{definition}
  \newtheorem*{thm*}{\protect\theoremname}
  \theoremstyle{definition}
  \newtheorem{prop}[thm]{\protect\propositionname}

\input{general-preamble}

\makeatother
  \providecommand{\corollaryname}{Corollary}
  \providecommand{\definitionname}{Definition}
  \providecommand{\factname}{Fact}
  \providecommand{\lemmaname}{Lemma}
  \providecommand{\propositionname}{Proposition}
  \providecommand{\remarkname}{Remark}
  \providecommand{\theoremname}{Theorem}

\usepackage{float}

\address{Department of Mathematics, The University of Chicago, Chicago, IL 60637, USA} 
\email{kihopark@math.uchicago.edu}

\address{Department of Mathematics, Northwestern university, Evanston, IL 60208, USA} 
\email{mark.piraino@northwestern.edu}

\begin{document}

\sloppy
\title{Transfer operators and limit laws for typical cocycles}
\author{Kiho Park, Mark Piraino}
\date{\today}
\begin{abstract}
We show that typical cocycles (in the sense of Bonatti and Viana) over irreducible subshifts of finite type obey several limit laws with respect to the unique equilibrium states for \hol potentials. These include the central limit theorem and the large deviation principle. We also establish the analytic dependence of the top Lyapunov exponent on the underlying equilibrium state.
The transfer operator and its spectral properties play key roles in establishing these limit laws. 
\end{abstract}
\maketitle
%\tableofcontents
%\addtocontents{toc}{\vskip-40pt}

\section{Introduction}

There is a vast literature on the asymptotic properties for products of independent random matrices, including the strong laws of large numbers \cite{MR121828}, the simplicity of the top Lyapunov subspace \cite{MR727020, MR2087783}, various limit laws \cite{MR669072}, and the regularity of the Lyapunov exponent with respect to the data \cite{page1989regularite, duarte2016lyapunov, MR3667994}. The goal of this article is to extend some of these results to the case in which the matrices are not independent but instead driven by a sufficiently nice dynamical system. 

Let $T$ be an irreducible adjacency matrix and $\Sig$ be the two-sided subshift of finite type defined by $T$. For a given \hol potential $\h\psi \colon \Sig \to \R$ and its unique equilibrium state $\hmu=\hmu_{\h\psi}$, it is well known that $\hmu$ satisfies various ergodic and statistical properties; see \cite{bowen1975ergodic, MR1793194}. 
These systems generalize the i.i.d. setting in the sense that they have conditional probabilities which may depend on the entire past. The assumption that $\h\psi$ is \hol is equivalent to the assumption that the conditional probabilities vary \hol continuously. 
%Probabilistically, we can view these systems as those whose conditional probabilities, which may depend on the entire past, vary H\"older continuously.

Given a continuous cocycle $\h\A \colon \Sig \to \glr$, there are already known results about some of the listed properties above for the long term behavior of $\displaystyle \frac{1}{n}\log \|\h\A^n(\hx)\|$ for $\hmu$-generic points $\hx$. The strong law of large numbers for example is known to hold in great generality \cite{MR356192}; see also \cite{oseledets1968multiplicative}. Sufficient conditions which imply the simplicity of the top Lyapunov subspace have received much attention \cite{bonatti2004lyapunov,avila2007simplicity}.
There has also been a significant amount of work concerning the continuity of the Lyapunov exponents \cite{bochi2005lyapunov,duarte2016lyapunov,MR3667994,viana2019continuity}.
However, other than few exceptions such as \cite{bougerol1988theoremes}, less attention has been paid to the problem of establishing the limit laws. 
%\cite{bougerol1988theoremes} \textcolor{red}{check}
\let\thefootnote\relax\footnotetext{M.P. was supported in part by the National Science Foundation grant "RTG: Analysis on manifolds" at Northwestern University.}

 In the scalar case, it is well known that for a \hol potential $\h\vp \colon \Sig \to \R$ whose $n$-th Birkhoff sum is denoted by $S_n\h\vp$, the distribution of 
$$\frac{S_n\h\vp(\hx) - n\hmu(\h\psi)}{\sqrt{n}}$$
with respect to $\hmu$ satisfies various limit laws such as the central limit theorem, the large deviations principle, and a collection of other properties; see \cite{ratner1973central,young1990large} for instance. 
%\textcolor{red}{add duarte-klein somewhere}
%\textcolor{red}{I just realized after writing the outline of the intro that I was considering $T$ as a primitive matrix; but I am guessing all these results (including what appears in next paragraph) also hold when $T$ is irreducible?}
%\mpcomment{Probably they hold when the matrix is irreducible finding a reference might be difficult. I think it is ok to be a bit vague about the exact assumptions in the introduction.}
%\textcolor{red}{Connect to random walks $\{Y_n\}_{n\in \N}$ taking values in $\glr$ according to some law .... (fill in). Denoting by $\lambda\in \R$ the Lyapunov exponent of the random walk, the distributions of 
%$$\frac{\log \|Y_n\ldots Y_1\|-n\lambda}{\sqrt{n}}$$ 
%and
%$$\frac{\log \|Y_n\ldots Y_1u\|-n\lambda}{\sqrt{n}} \text{ for any }u\in \mathbb{S}^{d-1}$$
%attracted much attention in recent past. Despite difficulties arising from noncommutativity of the group $\glr$, such distributions are shown to obey the standard limit laws listed above under suitable assumptions such as strong irreducibility and proximality; add reference. Maybe relate to a locally constant cocycle for smooth transition to next paragraph.
%}

The main goal of this article is to extend some of these limit theorems to a class of  $\glr$-cocycles. More specifically, we consider 1-typical cocycles among fiber-bunched cocycles.
For $\h\A \colon \Sig\to \glr$, we say $\h\A$ is \textit{fiber-bunched} if it is nearly conformal, and here ``nearly'' depends on the metric defined on $\Sig$ and the \hol exponent of $\h\A$. 
The 1-typicality assumption was first introduced by Bonatti and Viana \cite{bonatti2004lyapunov}. They showed that 1-typical cocycles have simple top Lyapunov exponent with respect to any ergodic measures with continuous local product structure. The set of all such measures includes unique equilibrium states for \hol potentials such as $\hmu$. In some sense, the 1-typicality assumption is a suitable analogue for the proximality and the strong irreducibility that commonly appear in the study of the products of i.i.d. random matrices. Moreover, the set of 1-typical cocycles forms an open and dense subset of fiber-bunched cocycles. See Section \ref{sec: prelim} for its precise definition.
We denote the Lyapunov exponents of $\h\A$ with respect to $\hmu$ by $\displaystyle \lambda_1(\h\A,\hmu) \geq \ldots \geq \lambda_d(\h\A,\hmu)$.
For the next two theorems, fix any unit vector $u\in \R^d$ and define the variance as follows:
\begin{equation}\label{eq: sigma}
\var:=\var(\h\A,\hmu)=\lim_{n \to \infty}\frac{1}{n}\int \left(\log \norm{\hat{\A}^{n}(\hx)u}-n\lambda_{1}(\hat{\A},\hmu)\right)^{2}d\hmu.
\end{equation}
% \mpcomment{We should probably us a different letter for the variance so that it is not confused with the shift}
We will see that under the 1-typicality assumption on $\h\A$, the value of $\var$ is independent of the choice of $u$; see Proposition \ref{prop: deriv of rho}.

The first main theorem is the following version of the central limit theorem.

\begin{thmx}[Central Limit theorem]\label{thm:CLT}
Let $\S_{T}$ be a shift of finite type defined by an irreducible adjacency matrix $T$, and let $\hat{\A} \colon \Sig \to \glr$ be a $1$-typical cocycle. Let $\h\psi \colon \Sig \to \R$ be a \hol potential and $\hmu \in \M(\s)$ its unique equilibrium state. If $\var(\h\A,\hmu)>0$, then
 for every unit vector $u \in \R^d$,
	\[ \frac{\log \norm{\hat{\A}^{n}(\hx)u} - n\lambda_{1}(\hat{\A},\hmu)}{\sqrt{n}}\xrightarrow[n \to \infty]{\text{dist}}\n(0,\var). \]
	If $\var(\h\A,\hmu)=0$, then
	\[ \frac{\log \norm{\hat{\A}^{n}(\hx)u} - n\lambda_{1}(\hat{\A},\hmu)}{\sqrt{n}}\xrightarrow[n \to \infty]{\text{dist}}0. \]
\end{thmx}	

The next theorem establishes the large deviation principle in the same setting.

\begin{thmx}[Large Deviation Principle]\label{thm:LDP}
In the setting of Theorem \ref{thm:CLT}, suppose that $\var(\h\A,\hmu)>0$.
%	Suppose $\hat{\A}:\S_{T}\to \glr$ is $1$-typical, $\varphi:\S_{T}\to \R$ is a H\"older continuous, $\mu_{\varphi}$ its Gibbs state and that 
Then there exist $\eta>0$ and a strictly convex function $\Lambda^{\ast} \colon [0,\eta] \to \R$ such that $\Lambda^\ast$ is positive on $(0,\eta)$ and that for any $\e\in (0,\eta)$ and any unit vector $u \in \R^d$,
	\[ \lim_{n \to \infty}\frac{1}{n}\log \hmu\set{\hx \in \Sig \colon \Big| \log \norm{\hat{\A}^{n}(\hx)u} - n\lambda_{1}(\hat{\A},\hmu) \Big|>n\e}=-\Lambda^{\ast}(\e). \]
\end{thmx} 

As a corollary, we also obtain the following version of the large deviation principle.

\begin{cor}\label{cor: LDP}
In the setting of Theorem \ref{thm:CLT}, suppose that $\var(\h\A,\hmu)>0$. Then for all sufficiently small $\ep>0$, there exists $C>0$ such that for all $n\geq 0$,
$$\hmu\set{\hx \in \Sig \colon \Big| \log \norm{\hat{\A}^{n}(\hx)} - n\lambda_{1}(\hat{\A},\hmu) \Big|>n\e} \leq Ce^{-C^{-1}n}.$$
\end{cor}
We note that this version of the large deviation principle has previously been obtained by Gou\"ezel and Stoyanov \cite{gouezel2019} in the same setting. Furthermore, during the writing process of this paper, we were made aware that Duarte, Klein, and Poletti \cite{duartekleinpoletti} also establish such large deviation principle in a related setting.
Both works then use this result to show interesting applications.

We also remark that for the products of i.i.d. random matrices, both the central limit theorem and the large deviation principle have been previously established under various assumptions; see \cite{MR669072, MR886674, benoist2016random,sert2019}. Under a suitable moment condition together with an irreducibility assumption (i.e., Zariski density of the semi-group generated by the support of the distribution), the latest versions of the central limit theorem and the large deviation are due to Benoist and Quint \cite{benoist2016random} and Sert \cite{sert2019}, respectively. 	

Our next result establishes the analytic dependence of the Lyapunov exponent $\lambda_1(\h\A,\hmu)$ on the measure $\hmu$.

%\begin{thmx}\label{thm: analytic}
%In the setting of Theorem \ref{thm:CLT}, suppose further that $t \mapsto \h\psi_{t}$ is a function defined on some interval $(-\e,\e)$ such that $\h\psi_{t}$ is a \hol continuous function on $\Sig$ for each $t$. If $t \mapsto \h\psi_{t}$ is analytic, then $t \mapsto \lambda_{1}(\h\A ,\hmu_{\h{\psi}_t})$ is real analytic in a neighborhood of $0$.
%\end{thmx}

\begin{thmx}[Analyticity of the exponent]\label{thm:Analyticitygeneral}
In the setting of Theorem \ref{thm:CLT}, let $\{\h\vp_n\}_{n\in \N_0}$ be a sequence of positive real-valued \hol continuous functions on $\Sig$, and let  
	\[ \hat{\psi}_{t} :=\sum_{n=0}^{\infty}t^{n}\hat{\varphi}_{n}.\]
Suppose that $\displaystyle \sum_{n=0}^{\infty}\e^{n}\norm{\hat{\varphi}_{n}}_{C^{\alpha}(\S_{T})}<\infty$, for some $\e,\a>0$. Then $t \mapsto \lambda_{1}(\hat{\A},\mu_{\hat{\psi}_{t}})$ is real analytic in a neighborhood of $0$.
%\kp{added in $\a$ next to $\a$ as $\a$ shows up in $\|\cdot \|_\a$}
\end{thmx}
% \mpcomment{I was thinking maybe we state Thm C like this in the introduction. That way it is more accessible to people less familiar with complex analysis in Banach spaces}

These theorems are established via introducing and studying the spectral properties of the operator $\L$, which we describe below.
From the given \hol continuous function $\h\psi \colon \Sig \to \R$, we construct a \hol continuous function $\psi$ in the cohomology class of $\h\psi$ that depends only on the future coordinates and realize $\psi$ as a function on the one-sided subshift $\Sig^+$. Associated to $\psi$, the normalized Ruelle-Perron-Frobenius operator $L_\psi$ on $C(\Sig^+)$ is defined by 
$$
L_\psi f(x) :=\frac{1}{\lambda} \sum\limits_{y \colon \s y= x}e^{\psi(y)}f(y)
$$
where $\lambda = e^{P(\psi)}$.

In this setting, the operator $L_\psi$ acting on a suitable space of \hol functions is quasi-compact and has a unique eigenfunction $h$ and a unique eigenmeasure $\eta$ corresponding to its spectral radius. Such eigendata give rise to the unique equilibrium state $\mu:=\mu_\psi$ for $\psi$ defined by $d\mu =h\cdot d\eta$; in fact, the equilibrium state $\hmu$ for $\h\psi$ can be retrieved as the natural extension of $\mu$. Historically, exploiting such spectral property of $L_\psi$ has been a popular method for establishing statistical properties for $\hmu$ mentioned above, see \cite{MR1793194} and the references contained therein. 
%\textcolor{blue}{is this paragraph still true for irreducible subshifts? might have to change/remove 	``uniqueness'' statement about the eigendata?}
%Given a fiber-bunched cocycle $\h{\A} \colon \Sig \to \glr$, we may use its holonomies to conjugate it to another cocycle $\A$ depending only on the future and realize it as a cocycle over $\Sig^+$. 

By extending $L_\psi$, we define an operator 
%$\L_0:=\L_{\A,\psi}$ on $ C(\Sig^+\times \mathbb{P}^{d-1})$ given by 
%$$
%\L_0 f(x,\ol{u}) :=  \sum\limits_{y \colon \s y = x}e^{\psi(y)} f(y, \ol{\A(y)^{-1}u}). 
%$$
%Since the weights $e^{\psi(y)}$ depends only on the base, $\L_0$ is an extension of $L$ in the sense that its action on functions that only depend on the base agrees with that of $L_\psi$.
%In particular, the spectral radii of $\L_0$ and $L_\psi$ coincide, and are equal to $\lambda:=e^{P(\psi)}$ where $P(\psi)$ is the topological pressure of $\psi$. For simplicity, we will work with the normalized operator 
$\L \colon C(\Sig^+\times \P^{d-1}) \to C(\Sig^+\times  \P^{d-1})$ defined by 
$$
\L f(x,\ol{u}) := \frac{1}{\lambda}\sum\limits_{y \colon \s y = x}e^{\psi(y)} f(y, \ol{\A(y)^{*}u}).
$$

In Theorem \ref{thm: main} we establish a Lasota-Yorke inequality for $\L$, and we use it to deduce further spectral properties of $\L$ in Theorem \ref{thm:spectralstructure}. The proofs of the main theorems stated above are based on such spectral properties of $\L$. In proving these spectral properties of $\L$, we exploit the gap in the top Lyapunov exponents $\lambda_1(\h\A,\hmu)>\lambda_2(\h\A, \hmu)$ ensured by the 1-typicality assumption on $\h\A$. One difficulty in exploiting such a gap comes from the non-uniform nature of Lyapunov exponents. Such difficulty is overcome with the aid of the 1-typicality assumption: we show that non-uniform convergence of some expressions that limit to $\lambda_1(\h\A,\hmu)$ becomes uniform when suitably averaged. See Section \ref{sec: quasi-compactness} for details

We note that in the case where $\Sig$ is a full shift and $\A$ is locally constant, the value of $\L f(x,\ol{u})$ only depends on the second coordinate $\ol{u}$, and $\L$ descends to an operator on $C(\P^{d-1})$. Similar operators have appeared in the literature to study various objects such as the subadditive equilibrium states \cite{piraino2018weak} and the products of random matrices and its associated stationary measures; see \cite{guivarc2015spectral} and references therein.

The paper is organized as follows. In Section \ref{sec: prelim}, we elaborate on the setting for the main theorems and survey relevant preliminary results. Then in Section \ref{sec: quasi-compactness}, we prove quasi-compactness of the operator $\L$ under the 1-typicality assumption on $\A$. Using the quasi-compactness of $\L$, we establish spectral properties of $\L$ in Section \ref{sec: spectral results}. In Section \ref{sec: limit theorems}, we prove the main results of this paper. Section \ref{sec: appendix} is the Appendix containing the results useful in establishing the main theorems.
\\
%\textcolor{red}{(not entirely sure where might be a suitable place to mention this, and what more we should say on this)} \textcolor{red}{I think end of the intro is fine for this-MP}

%\textcolor{red}{We should put an analyticity theorem here}
%\\
\noindent \textbf{Acknowledgments:} The authors would like to thank Amie Wilkinson and Aaron Brown for helpful discussions. The authors would also like to thank Jairo Bochi and Silvius Klein for helpful comments.

\section{Preliminaries}\label{sec: prelim}
\subsection{Set up} We elaborate on the setting for the main theorems and set up the relevant notations here.
\begin{definition}
We say that a non-negative matrix $T$ is \emph{irreducible} if for any $i,j$ there exists an $M$ such that $(T^{M})_{ij}>0$. We say that a matrix is \emph{primitive} if there exists an $M\in \N$ such that $T^{M}>0$.
\end{definition}
% \mpcomment{typo}
Throughout the paper, $T$ is an irreducible $q\times q$-square matrix with entries in $\{0,1\}$. 
We denote by $\Sig$ and $\Sig^+$ the corresponding two-sided and one-sided subshift of finite type defined by $T$, respectively. We will use $\s$ to denote the left shift operator for both $\Sig$ and $\Sig^+$. For points and objects relevant to the two-sided subshift $\Sig$, we will often indicate it with the hat notation, such as $\hx$. 
%primitive, that is, there exists $N\in \N$ such that $T^N$ is a positive matrix. This is equivalent to the topological mixing of the corresponding subshift.

Let $P \colon \Sig \to \Sig^+$ be the standard projection mapping $ (x_i)_{i \in \Z}$ to $(x_i)_{i\in \N_0}$.
For any $\hx\in \Sig$, we define \textit{cylinder around $\hx$ of length $n$} by  
$$[\hx]_n:=\{(y_i)_{i \in \Z} \in \Sig \colon x_i = y_i \text{ for all }0 \leq i \leq n-1\}.$$
For any $x \in \Sig^+$, we likewise define the cylinder $[x]_n$ as $P([\hx]_n)$ for any $\hx \in P^{-1}(x)$.

We equip $\Sig$ with the following metric $\rho_0$: for $\hx = (x_i)_{i\in \Z}$ and $\hy = (y_i)_{i\in \Z}$,
$$\rho_0(\hx,\hy) := 2^{-k}$$
where $k$ is the largest integer such that $x_i = y_i$ for all $|i|<k$.
Similarly, we equip the one-sided subshift $\Sig^+$ with the same metric.
With such metric $\rho_0$, the subshift $(\Sig,\s)$ becomes a hyperbolic homemorphism. Its local stable sets are defined as
$$\Wloc^s(\hx):=\{\hy \in \Sig \colon x_i = y_i \text{ for all }i \in \N_0\},$$
and the local unstable sets $\Wloc^u(\hx)$ are likewise defined as the set of $\hy$ where $x_i = y_i$ for all $i \leq 0$. In the standard manner, we then extend the definition to global stable and unstable sets $\w^{s/u}(\hx)$.
By an abuse of notations, for any $x\in \Sig^+$ we define the local stable set of $x$ as $$\Wloc^{s}(x):=P^{-1}(x)=\Wloc^{s}(\hx) \text{ for any }\hx \in P^{-1}(x).$$

For any $\hx,\hy\in \Sig$ with $x_0 = y_0$, define the \textit{bracket} of $\hx$ and $\hy$ by
$$[\hx,\hy] := \Wloc^u(\hx) \cap \Wloc^s(\hy).$$
For any $\hx = (x_i)_{i \in \Z} \in \Sig$ and $y = (y_i)_{i\in \N_0} \in \Sig^+$ with $x_0 = y_0$, we define $[\hx,y]:=[\hx,\hy]$ by choosing any $\hy \in \Wloc^s(y)$. Note $[\hx,y]$ is well-defined independent of the choice of $\hy \in \Wloc^s(y)$, and we may often denote $[\hx,y]$ by $\hx_y$.

For any continuous cocycle $\A \colon \Sig \to \glr$ and $n\in \N$, we set
$$\A^n(\hx):= \A(\s^{n-1}\hx) \ldots \A(\hx).$$
From the definition, it satisfies the \textit{cocycle equation}: for any $m,n\in\N$ and $\hx \in \Sig$, $$\A^{m+n}(\hx) = \A^m(\s^n \hx)\A^n(\hx).$$

We can associate $\A$ to two related cocycles:
the adjoint cocycle and the inverse cocycle. Both are cocycles over $(\Sig,\s^{-1})$ and they are defined by
\begin{equation}\label{eq: inverse cocycle}
\A_*(\hx) := \A(\s^{-1}\hx)^* \text{ and }\A^{-1}(\hx) :=\A(\s^{-1}\hx)^{-1}.
\end{equation}
In particular, their iterations are given by $\A_*^n(\hx) = [\A^{n}(\s^{-n}\hx)]^*$ and $\A^{-n}(\hx)=\A^n(\s^{-n}\hx)^{-1}$. We denote the adjoint of the inverse cocycle by $\A^{-1}_*$; this defines a cocycle over $(\Sig,\s)$ by $\A^{-1}_*(\hx) = [\A(\hx)^{-1}]^*$.

Another system related to a given cocycle $\A \colon \Sig \to \glr$ is the skew product $\h{F}_\A \colon \Sig \times \P^{d-1} \to \Sig \times \P^{d-1}$ defined by
$$\h{F}_\A(\hx,v) = (\s\hx ,\ol{\A(\hx)v}).$$ It is clear that the action of $\A^n(x)$ on the projective space $\P^{d-1}$ is encoded in the second coordinate of the iterations of the skew product $\h{F}_\A$. Similarly, we denote the skew product on $\Sig^+ \times \P^{d-1}$ by $F_\A$. 
%\mpcomment{Is this the notation for the projective action? Should the whole thing be under a $\ol{\cdot\cdot \cdot}$?}

A cocycle $\A \colon \Sig \to \glr$ is $\theta$-\hol if there exists $C>0$ such that for all $\hx,\hy\in \Sig$, 
$$\|\A(\hx)- \A(\hy)\| \leq C \cdot \rho_0(\hx,\hy)^\theta.$$
We denote the set of all $\theta$-\hol cocycles by $C^\theta(\Sig,\glr)$. Throughout the paper, we will work with $\theta$-\hol cocycles $\A\colon \Sig \to \glr$ for some $\theta \in (0,1]$ satisfying an extra condition called the fiber-bunching.
%We now introduce the fiber-bunching assumption on 1-\hol cocycles 
%$\A \colon \Sig \to \glr$. 
\begin{defn}
A $\theta$-\hol cocycle $\A$ is \textit{fiber-bunched} if 
$$\|\A(x)\|\cdot \|\A(x)^{-1}\| < 2^\theta$$
for every $x\in \Sig$.  
\end{defn}
Without loss of generality, we may assume that $\theta = 1$. Indeed, for the general case when $\theta \neq 1$ we may rescale the metric on $\Sig$ to 
$$\rho(x,y) := \rho_0(x,y)^\theta$$
so that $\A$ belongs to $\Lip(\Sig,\glr):=C^1(\Sig,\glr)$. 

%\subsection{Fiber-bunched cocycles, the canonical holonomies, and typical cocycles}

Clearly, conformal cocycles are fiber-bunched. In fact, small perturbations of conformal cocycles are also fiber-bunched. An alternate way to think of fiber-bunched cocycles is that their skew products $ \h{F}_\A \colon \Sig\times \P^{d-1} \to \Sig\times \P^{d-1}$ defined above are partially hyperbolic systems.
The set of fiber-bunched cocycles 
$$C_b(\Sig,\glr):=\{\A\in \Lip(\Sig,\glr) \colon \A \text{ is fiber-bunched}\}$$ forms an open subset of $\Lip(\Sig,\glr)$. 
%\mpcomment{We should define $C^{\alpha}(\S_{T}^{+}\times \P^{d-1})$ somewhere. When I first saw this I was confused by I was thinking $C^{1}$ in the differentiable sense}

One of the most important property of fiber-bunched cocycles is the convergence of the \textit{canonical stable/unstable holonomy} $H^{s/u}_{\hx,\hy}$: for any $\hy \in \Wloc^{s/u}(\hx)$, 
\begin{equation}\label{eq: canonical holonomies}
H^s_{\hx,\hy} :=\lim\limits_{n \to \infty} \A^n(\hy)^{-1}\A^n(\hx) ~\text{ and }~H^u_{\hx,\hy}:= \lim\limits_{n \to -\infty} \A^n(\hy)^{-1}\A^n(\hx).
\end{equation}
Moreover, the canonical holonomies vary \hol continuously in the basepoints $\hx,\hy\in \Sig$: there exists $C>0$ such that for any $\hy \in \Wloc^{s/u}(\hx)$, 
\begin{equation}\label{eq: holder holonomies}
\|H^{s/u}_{\hx,\hy} - I\| \leq C\cdot \rho(\hx,\hy).
\end{equation}
See \cite{kalinin2013cocycles} for further details.

It can be easily checked that the canonical stable holonomies $H^{s}_{\hx,\hy}$ satisfy the following properties:
\begin{enumerate}
	\item $H^s_{\hx,\hx} = I$ and $H^s_{\hy,\h{z}}\circ H^s_{\hx,\hy} = H^s_{\hx,\h{z}}$ for any $\hy,\h{z} \in \Wloc^s(\hx)$,
	\item $\A(\hx) = H^s_{\s \hy,\s \hx} \circ \A(\hy) \circ H^s_{\hx,\hy}$,
	\item $H^s\colon (\hx,\hy) \mapsto H^s_{\hx,\hy}$ is continuous as $\hx$ and $\hy$ vary continuously while satisfying the relation $\hy \in \Wloc^s(\hx)$.
\end{enumerate}
Likewise, the canonical unstable holonomies $H^u_{\hx,\hy}$ satisfy the analogous properties. 

The canonical holonomies guarantee that $\A$ has the following bounded distortion property: there exists $C>0$ such that for any $n \in \N$ and $\hx,\hy\in \Sig$ with $\hy \in [\hx]_n$, we have
\begin{equation}\label{eq: unif comparison}
C^{-1} \leq \frac{\|\A^n(\hx)\|}{\|\A^n(\hy)\|} \leq C.
\end{equation}

%When the context is clear, we will continue to use $\h{F}_\A$ to also denote its projectivization $\P\h{F}_\A \colon \Sig \times \P^{d-1}\to \Sig \times \P^{d-1}$, and likewise for $F_\A$.

For any vector $v\in \R^d$, we denote its corresponding projection onto $\P^{d-1}$ by $\ol{v}$. 
%Conversely, for any $\ol{v} \in \P^{d-1}$, \textcolor{red}{we identify it a unit vector $v \in \R^d$ in the direction of $\ol{v}$ when there is no ambiguity.} 
We put a metric $d $ on $\P^{d-1}$ defined by 
%\mpcomment{We should rephrase this because there is always ambiguity of $-v$ or $v$.}
\begin{equation}\label{eq: d}
d(\ol{u},\ol{v}) := \frac{\|u \wedge v\|}{\|u\| \cdot \|v\|}.
\end{equation}
It can be easily verified that $d(\ol{u},\ol{v})$ is equal to the sine of the angle between $\ol{u}$ and $\ol{v}$. 

For such metric $d$, there exists $C>0$ (twice the constant from \eqref{eq: holder holonomies}, for instance) such that for any $\ol{u}\in \P^{d-1}$ and $\hy \in \Wloc^{s/u}(\hx)$, 
\begin{equation}\label{eq: holder holonomies 2}
d\left(\ol{H^{s/u}_{\hx,\hy}u},\ol{u}\right) \leq C\cdot \rho(\hx,\hy).
\end{equation}

% \mpcomment{$\left(\right)$}

By an abuse of notation, we also denote by $d$ the metric on the product space $\Sig \times \P^{d-1}$ and $\Sig^+ \times \P^{d-1}$ defined as the maximum of the distance in each coordinate:
$$d((\hx,\ol{v}),(\hy,\ol{u})) := \max\{\rho(\hx,\hy),d(\ol{u},\ol{v})\}.$$ 

%There exists $C>0$ such that for any $\hy\in \Wloc^{s/u}(\hx)$, 

Lastly, we describe how to conjugate and realize a given fiber-bunched cocycle $\h{\A}\colon \Sig \to \glr$ as a $\glr$-cocycle $\A$ over the one-sided subshift $(\Sig^+,\s)$.
Assuming that $\Sig$ has $q$ alphabets, for each $i \in \{1,2,\ldots,q\}$, fix $\heta^i \in \Sig$ with $(\heta^i)_0 = i$.
%By an abuse of notation, we identify $\Sig^-$ with $W^s_{loc}(x)$ for $x \in \Sig^+$. 
For each $\hx = (x_i)_{i \in \Z} \in \Sig$, we set
$$\hx_\eta:=[\heta^{x_0},\hx].$$ 
%\mpcomment{We might want be explicit about what conjugate means here because, it is conjugate in the sense of co-cycles but not matrices I think. Is that correct?}
%\kp{yes, indeed. They are conjugated in the sense of cocycles. I think referencing BBB here should be sufficient.}
We then define a new cocycle $\A$ on $\Sig$ given by
$$\A(\hx):= \hat{\A}(\hx_\eta).$$
%Then the iterations of $\A$ takes the following form:
%$$\A^n(\hx):=\h{H}^s_{\s^n (\hx_\eta),(\s^n \hx)_\eta }\hat{\A}^n(\hx_\eta).$$
Denoting the local holonomies of $\h{\A}$ by $\h{H}^{s/u}$, two cocycles $\A$ and $\h{\A}$ are conjugated to one another by the conjugacy $\CC(\hx):=\h{H}^s_{\hx,\hx_\eta}$; see \cite{backes2018continuity} for instance.
From its definition, we have 
\begin{equation}\label{constant along stable}
\A(\hx) = \A(\h{y}) \text{ for }\h{y} \in \Wloc^s(\hx).
\end{equation}
Denoting its local holonomies by $H^{s/u}$, we have $H^s \equiv I$, and hence, we may realize $\A$ as a cocycle over the one-sided subshift $(\Sig^+,\s)$. 

We will always assume that all fiber-bunched cocycles from here on have already been conjugated as above, and use $\A$ to denote the cocycle over $(\Sig^+,\s)$ as well as $(\Sig,\s)$ when the context is clear.

%The class of fiber-bunched cocycles denoted by $C^\alpha_b(\Sig,\glr)$ forms an open subset of $C^\alpha(\Sig,\glr)$. Within the set of fiber-bunched cocycles, we define a class of \textit{typical} cocycles.

\subsection{Typicality assumption on $\A$}
In order to introduce the 1-typicality assumption on $\A$, we need to introduce the notion of holonomy loops. For any periodic point $p \in \Sig$, the \textit{homoclinic points} of $p$ are points that belong to the intersection between $\w^s(p)$ and $\w^u(p)$. For uniformly hyperbolic systems such as $(\Sig,\s)$, the set of $p$-homoclinic points are dense for any periodic point $p$. To each $p$-homoclinic point $z$, we can associate a \textit{homoclinic loop} defined by
$$\psi_z:= H^s_{z,p}\circ H^u_{p,z}.$$

\begin{defn}
We say $\A \in C_b(\Sig,\glr)$ is \textit{1-typical} if there exist a periodic point $p$ of some period $n$ and a $p$-homoclinic point $z$ such that
\begin{enumerate}
\item $P:=\A^{n}(p)$ has simple eigenvalues of distinct norms, and
\item for any $I,J \subseteq \{1,\ldots,d\}$ with $|I|+|J| \leq d$, 
$$\{\psi_z(v_i) \colon i\in \I \} \cup \{v_j \colon j\in J\}$$
are linearly independent, where $\{v_i\}_{i=1}^d$ are the eigenvectors of $P$.
\end{enumerate}
\end{defn}

%\textcolor{red}{Maybe included a comment about who defined $1$-typical and why}
%Such conditions are called \textit{pinching} and \textit{twisting}. 

Bonatti and Viana \cite{bonatti2004lyapunov} introduced the 1-typicality assumption as a sufficient condition for the simplicity of the top and bottom Lyapunov exponents with respect to any ergodic measures with continuous local product structure. Considering the pinching and twisting assumptions as suitable analogues for the proximality and the strong irreducibility, their work can be seen as generalizations to non-i.i.d. setting of Furstenberg's positivity of top Lyapunov exponent \cite{furstenberg1963noncommuting}. 

If the pinching and twisting conditions hold for all exterior product cocycles $\A^{\wedge t}$, $t\in \{1,\ldots, \lfloor d/2\rfloor \}$, such cocycles $\A$ are called \textit{typical}. Using standard tricks, the simplicity of the top and bottom Lyapunov exponents for 1-typicality can be modified to show the simplicity of all Lyapunov exponents for such cocycles. Moreover, typical cocycles form an open and dense subset of $C_b(\Sig,\glr)$.

Typicality, in slightly altered forms, have recently been employed to study subadditive thermodynamic formalism; see \cite{park2020quasi}.  

%\begin{rem} The 1-typicality assumption on $\A$ is equivalent to the 1-typicality assumption on $\A_*$, $\A^{-1}$, and $\A^{-1}_*$.
%\end{rem}

\subsection{Invariant measures}\label{subsec: invariant measures}
Let $\h\psi$ be a \hol potential on $\Sig$ and $\hmu\in \M(\s)$ be its unique equilibrium state; the uniqueness of $\hmu$ is well-known when $T$ is primitive \cite{bowen1975ergodic}, but the result still holds when $T$ is irreducible \cite{MR1793194}.
We fix $\h\psi$ and $\hmu$ once and for all throughout the paper.

Let $\psi$ be a \hol potential cohomologous to $\h\psi$ defined as in \cite[Lemma 1.6]{bowen1975ergodic}, and we realize $\psi$ as a potential on $\Sig^+$. Then the projection of $\hmu$ onto $\Sig^+$ denoted by $\mu$ is the unique equilibrium state for $\psi$.
%
%denote its equilibrium state on $\Sig^+$ by $\mu:=\mu_\psi \in \M(\s)$. The unique $\s$-invariant probability measure on $\Sig$ that projects to $\mu$ is denoted by $\hmu$; that is $\hmu$ is the lift of $\mu$.

Instead of working directly with $\psi$, we will work with $g$-functions, which helps simplify the presentations of the proofs for main theorems. We say $g \colon \Sig^+ \to \R$ is a \textit{$g$-function} if $g$ is a positive function satisfying
$$\sum\limits_{y\colon \s y = x} g(y) = 1$$
for all $x \in \Sig^+$. This is equivalent to the condition that $L_{\log g}1 = 1$. 

Recalling that $h$ is the unique eigenfunction corresponding to the spectral radius of $L_\psi$, a \hol function $\log g$ defined by 
\begin{equation}\label{eq: g}
 g(x):= \frac{e^{\psi(x)}}{\lambda} \cdot \frac{h(x)}{h(\s x)}
 \end{equation}
is cohomologous to $\psi$. Hence, we may work with $\log g$ instead of the given $\psi$, and its Ruelle-Perron-Frobenius operator is then defined by 
$$L_{\log g} f(x) := \sum\limits_{y \colon \s y = x}g(y)f(y).$$
Then its unique eigenmeasure corresponding to the spectral radius coincides with $\mu$, the unique equilibrium state for $\psi$. Replacing $\psi$ with $\log g$, we can see that $\L$ takes the form 
%\mpcomment{A slightly different phrasing here might be better}
$$\L f(x,\ol{u}) := \sum\limits_{y \colon \s y = x}g(y) f(y,\ol{\A(y)^{*}u}).$$
%\textcolor{blue}{still true about uniqueness of $h$?
%Likewise, we will work with $L_{\log g} f(x) = \sum\limits_{y \colon \s y = x}g(y)f(y)$ as the usual Ruelle-Perron-Frobenius operator on the base $\Sig^+$.

It is clear that the spectral radius $\rho(\L)$ is equal to 1, and the constant function 1 is an eigenfunction of $\L$.
Moreover, setting $g^{(n)}(x):=g(\s^{n-1} x)\ldots g(x)$, the Gibbs property of $\mu$ gives
$$\mu(\I) \asymp g^{(n)}(x)$$
for any cylinder $\I$ of length $n$ and any $x \in \I$.

%For this subsection, let $\h\psi$ be a \hol potential on $(\Sig,\s)$. Denoting the set of all $\s$-invariant probability measures by $\M(\s)$, let $\hmu := \hmu_{\h\psi} \in \M(\s)$ be its unique equilibrium state on $\Sig$.  Let $\h\psi$ be a \hol potential on $(\Sig,\s)$ be a \hol potential cohomologous to $\h\psi$ defined as in \cite[Lemma 1.6]{bowen1975ergodic}. We realize $\psi$ as a potential on $\Sig^+$ and denote its unique equilibrium state on $\Sig^+$ by $\mu:=\mu_\psi \in \M(\s)$. 

%Moreover, $\hmu$ and $\mu$ satisfy the
%\textit{Gibbs property}: there exists $C\geq 1$ such that for every $\hx\in \Sig$ and $n\in \N$,
%$$C^{-1} \leq \frac{\hmu([\hx]_n)}{e^{-nP(\psi)+S_n\psi(\hx)}} \leq C.$$

By considering $\Sig^+$ as the parameter space for the local stable sets of $\Sig$, we may consider a disintegration $\{\hmu^s_x\}_{x\in \Sig^+}$ of $\hmu$ where each $\hmu^s_x$ is a probability measure on $\Wloc^s(x)$ for $\mu$-a.e. $x$. For any such disintegration, we have
$$\hmu=\int_{\Sig^+} \hmu_x^s \, d\mu(x).$$

%Define $P \colon \Sig \to \Sig^+$, $\h\pi \colon \Sig \times\mathbb{P}^{d-1} \to \Sig$ and $\pi \colon \Sig^+ \times\mathbb{P}^{d-1} \to \Sig^+$; coherent notations for these.
It is well-known that equilibrium states of \hol potentials, such as $\hmu$, have continuous local product structure; see \cite{leplaideur2000local}. From such property, there exists a disintegration $\{\hmu^s_x\}_{x\in \Sig^+}$ of $\hmu$ whose  unstable holonomies between local stable sets defined by
\begin{align*}
\begin{split}
h_{x,y} \colon (\Wloc^s(x), \hmu^s_x) &\to(\Wloc^s(y), \hmu^s_y)\\
\hx &\mapsto [\hx,y]
\end{split}
\end{align*}
are absolutely continuous for \textit{every} $x,y \in \Sig$ with $x_0=y_0$. Its Jacobian $J_{x,y}$ depends continuously in $x,y \in \Sig^+$, and we have  
\begin{equation}\label{eq: abs cty}
\hmu^s_y = J_{x,y}(h_{x,y})_*\hmu^s_x.
\end{equation}

We now survey relevant results from Bonatti and Viana \cite{bonatti2004lyapunov}.
Let $\A \colon \Sig \to \glr$ be a fiber-bunched cocycle, and $\hm$ be a $ \h{F}_{\A}$-invariant measure on $\Sig \times \P^{d-1}$ that projects to $\hmu$ under the canonical projection $\h{\pi} \colon \Sig\times\P^{d-1} \to \Sig$. Setting $m:=(P\times \id)_*\hm$, a Martingale argument shows that $\hm$ can be retrieved from $m$ in the following sense: letting $\hx_n:=P(\s^{-n}\hx)$, for $\hmu$-a.e. $\hx$, 
\begin{equation}\label{eq: retrieve hm from m}
\lim_{n\to\infty} \A^n(\hx_n)_*m_{\hx_n} = \hm_{\hx}.
\end{equation}

\begin{defn} A probability measure $\hm$ on $\Sig\times \mathbb{P}^{d-1}$ is \textit{$H^u$-invariant} if it projects to $\hmu$ and there exists a disintegration $\{\hm_{\hx}\}_{\hx \in \Sig}$ along the fibers such that 
$$(H^u_{\hx,\h{y}})_*\hm_{\hx} = \hm_{\h{y}}$$
for every $\hx$ and $\h{y}$ in the same local unstable set. 
We say $\hm$ is \textit{$(\A,H^u)$-invariant} if, in addition, $\hm$ is $ \h{F}_{\A}$-invariant. Such measure is also known as a $u$-$state$.

We say a probability measure $m$ on $\Sig^+\times \mathbb{P}^{d-1}$ is \textit{$(\A,H^u)$-invariant} if there exists a $(\A,H^u)$-invariant probability measure $\hm$ on $\Sig^+\times \mathbb{P}^{d-1}$ with $(P\times \id)_* \hm = m$.
\end{defn} 
We similarly define the $(\A,H^s)$-invariance. While the existence of $(\A,H^u)$-invariant measures is not a priori obvious, it is shown in  \cite{bonatti2004lyapunov} that the set of $(\A,H^u)$-invariant measures is necessarily non-empty. 

The main result of Bonatti and Viana \cite{bonatti2004lyapunov} is that if $\A$ is 1-typical, then the top and bottom Lyapunov exponents of $\A$ with respect to $\hmu$ are simple. Let $\xi(\hx) \in\mathbb{P}^{d-1}$ be the projectivization of the top Oseledets subspace at $\hx$ with respect to $\A$ and $\hmu$; when there is no confusion, we will denote a unit vector in its direction also by $\xi(\hx)$. Then consider a probability measure $\hm$ on $\Sig\times \mathbb{P}^{d-1}$ which projects to $\hmu$ and whose conditional measures are defined by
$$\hm_{\hx} :=\delta_{\xi(\hx)}.$$
Since $\xi(\hx)$ is both $\h{F}_\A$ and $H^u$-invariant, it follows that $\hm$ is a $(\A,H^u)$-invariant measure.

Bonatti and Viana then proceed to show that $\hm$ is the \textit{unique} $(\A,H^u)$-invariant measure over $\hmu$ when $\A$ is 1-typical. Setting $m:=(P\times \id)_* \hm$, we summarize the properties of $\hm$ and $m$ in the following proposition:
\begin{prop}\label{prop: nu}\cite{bonatti2004lyapunov} Suppose $\A$ is 1-typical. Then $m$ constructed as above satisfies the following properties:
\begin{enumerate}
\item $m$ admits a disintegration $\{m_x\}_{x\in \Sig^+}$ such that 
$x \mapsto m_x$ is continuous for \textit{every} $x \in \Sig^+$.
\item For $\mu$-a.e. $x \in \Sig^+$, we have
$$\displaystyle\sum\limits_{\s y = x} \frac{1}{J_{\mu}\s(y)}\A(y)_*m_y = m_x,$$
where $J_{\mu}\s \colon \Sig^+ \to (0,\infty)$ is the Jacobian for $\mu$.
\item For any $x \in \Sig^+$ and proper subspace $V \subset \R^d$, we have $m_x(V) = 0$.
\item For every $x \in \Sig^+$, we have
\begin{equation}\label{eq: m and hm}
m_x = \int \hm_{\hx}~d\hmu_{x}^s(\h{x})  = \int \delta_{\xi(\hx)}~d\hmu_{x}^s(\h{x}).
\end{equation}
\end{enumerate}
\end{prop}

We end this subsection by commenting on the construction of the top Oseledets subspace $\xi(\hx)$. The following remark will be useful in the proof of Theorem \ref{thm: main}.
%Recall that $\xi(\hx)$ is the top Lyapunov subspace with respect to $\hmu$.
\begin{rem}\label{rem: inv measures}
Under 1-typicality, Bonatti and Viana obtain two measurable splittings of $\R^d$: for $\hmu$-a.e. $\hx$,
$$ \xi(\hx)\oplus V_2(\hx)=\R^d = \omega(\hx) \oplus W_2(\hx).$$
The 1-dimensional top Oseledets subspace equals to the span of $\xi(\hx)$ and its complement $V_2(\hx)$ consists of all vector $v \in \R^d$ whose asymptotic growth rate $\displaystyle \lim\limits_{n\to \infty} \frac{1}{n}\log \|\A^n(\hx)v\|$ is strictly less than $\lambda_1(\A,\hmu)$. Similarly, $\omega(\hx)$ spans the 1-dimensional bottom Oseledets subspace of $\A$ with respect to $\hmu$ and $W_2(\hx)$ is its complementary subspace. Note $\xi(\hx)$ is preserved under $H^u$ whereas $\omega(\hx)$ is preserved under $H^s$. These subspaces can be constructed as follow:

Letting $\hx_n:=P(\s^{-n}\hx)$, for $\hmu$-a.e. $\hx$, $\displaystyle \frac{\A^n(\hx_n)}{\|\A^n(\hx_n)\|}$ converges to a rank 1 quasi-projective map whose image coincides with $\xi(\hx) \in \mathbb{P}^{d-1}$. Denoting the $KAK$-decomposition, also known as the singular value decomposition, of $\A^n(\hx_n)$ by $K_nA_nU_n$, let $K_\infty \in O(d)$ be a subsequent limit of $K_n$. Then, we have $\xi(\hx)=K_\infty e_1$ for $\hmu$-a.e. $\hx$.
If we apply this construction of $\xi(\hx)$ to $\A_*$, then the corresponding direction $\xi_*(\hx)$ is used to define the slower subspace of the Lyapunov splitting $V_2(\hx):= \xi_*(\hx)^{\perp}$, necessarily transverse to $\xi(\hx)$. 
%\mpcomment{The KAK-decomposition is also known as the singluar value decomposition? I think they the same. We should probably say "also know as the singular value decomposition}

%\mpcomment{\textcolor{red}{Is this remark true if we just assume top and bottoms spaces are $1$-dimensional?} }
%
%\kp{\textcolor{blue}{yeah, I believe so.}}

Applying this construction to the inverse cocycle $\A^{-1}$ defines $\omega(\hx)$. Moreover, applying this construction to the adjoint of the inverse cocycle $\A^{-1}_*$, we obtain $\omega_*(\hx)$ whose orthogonal complement defines the complementary subspace $W_2(\hx)$. We notice that $\omega_*(\hx)$ coincides with $K_\infty e_d$, where $K_\infty$ is defined as in the previous paragraph. This is because keeping in mind the $KAK$-decomposition of $\A^{n}(\s^{-n}\hx)$ from the previous paragraph, the image of the rank 1 quasi-projective limit of $(\A^{-1}_*)^n(\s^{-n}\hx) = ((\A^n(\s^{-n}\hx))^{-1})^*$ is equal to $\ol{K_\infty e_d} \in \mathbb{P}^{d-1}$.
\end{rem}

\section{Quasi-compactness of $\L$}\label{sec: quasi-compactness}
For $f \in C^\a (\Sig^+\times \P^{d-1})$, we set 
$$\displaystyle |f|_{\a}:= \sup\limits_{(x,\ol{u})\neq (y,\ol{v})} \frac{ |f(x,\ol{u})-f(y,\ol{v})| }{d((\hx,\ol{u}),(\hy,\ol{v}))^\a} .$$ 
With this notation, $\norm{f}_\a = |f|_\a + \norm{f}_\infty$.

The main goal of this section is to prove the following theorem establishing a Lasota-Yorke inequality for $\L$.
%\kp{brought the theorem here as suggested}
\begin{thm}\label{thm: main} Suppose $\S_{T}^{+}$ is a shift of finite type defined by an irreducible $T$, and $\hat{\A}$ is $1$-typical. Then for all $\alpha>0$ sufficiently small, there exist $C>0$ and $\b\in (0,1)$ such that for any $f \in C^\alpha(\Sig^+\times \mathbb{P}^{d-1})$, we have
\begin{equation}\label{eq: lasota-yorke}
\|\L^n f\|_\alpha \leq \b^n\|f\|_\alpha+C\|f\|_\infty.
\end{equation}
for all $n \in \N$.
\end{thm}

%We are ready to prove Theorem \ref{thm: main}. 

As done in Section \ref{sec: prelim}, we will work with the cocycle $\A \colon \Sig \to \glr$ that is constant along the local stable sets obtained from $\h\A$ according to \eqref{constant along stable}. The canonical holonomies of $\A$ are denoted by $H^{s/u}$ where $H^s \equiv I$. 

For $x\in \Sig^+$ and $n\in \N$, we define
	\[ \A^{[n]}(x):= \A(x)^*\A(\sigma x)^*\cdots \A(\sigma^{n-1}x)^* = [\A^n(x)]^* .\]
Then the iterates of $\L$ takes the following form
$$\L^nf(x,\ol{u}) :=  \sum\limits_{y \colon \s^n y = x}g^{(n)}(y)\cdot f(y, \ol{\A^{[n]}(y)u}). 
$$
%In order to distinguish between the pushforward of a measure and the adjoint action of a linear map, we will use superscript $t$ for the adjoint action. 
%Setting $$\A^n_t(y):=\A(y)^t\A(\s y)^t\ldots \A(\s^{n-1}y)^t,$$
\subsection{Controlling the \hol norm of $\L^n f$} 
We will make a series of deductions to establish sufficient conditions for \eqref{eq: lasota-yorke}. First, we set 
$$
t_{n,\alpha}(x):=\sup\limits_{\ol{u} \neq \ol{v}} \sum\limits_{\s^n y = x} g^{(n)}(y) \cdot \Big(\frac{d(\ol{\A^{[n]}(y)u},\ol{\A^{[n]}(y)v})}{d(\ol{u},\ol{v})}\Big)^\alpha.
$$
and
$$ w_{n,\alpha}:=\max\limits_{x\in \Sig^+} t_{n,\alpha}(x).$$

For any $x_1$ and $x_2$ with the same 0-th symbol, consider any $y_1\in \s^{-n}x_1$. Then there exists $y_2 \in \s^{-n}x_2$ such that $y_1$ and $y_{2}$ belong to the same $n$-cylinder.
%\kp{added in an extra line here in attempt to clarify the definition for $\tau$. However, if you have other suggestions, I'd be happy to change the notation.}
We set
$$
\tau_{n,\alpha}  := \sup\limits_{\substack{x_1 \neq x_2 \\ \ol{v}}} \sum\limits_{\substack{\s^ny_i=x_i\\ i\in \{1,2\}}} g^{(n)}(y) \cdot \Big(\frac{d(\ol{\A^{[n]}(y_1)v},\ol{\A^{[n]}(y_2)v)}}{\rho(x_1,x_2)}  \Big)^\alpha.
$$
The supremum is taken over all distinct $x_1,x_2$ with the common $0$-th symbol and all $\ol{v} \in \P^{d-1}$, and the summation is over all $y_1 \in \s^{-n}x_1$ paired with the corresponding $y_2 \in \s^{-n}x_2$ in the same $n$-cylinder. 
%\mpcomment{I don't understand the set that the sum is taken over. That is I think we should find a different way of expressing it.}

The following proposition formulates a sufficient condition to establish Theorem \ref{thm: main}.
\begin{prop}\label{prop: w_n,alpha}
For all $\alpha>0$ sufficiently small, there exists $C>0$ such that for any $(x_1,\ol{u}),(x_2,\ol{v}) \in \Sig^+\times \P^{d-1}$, $f \in C^\alpha(\Sig^+ \times \mathbb{P}^{d-1})$, and $n\in \N$, we have
\begin{align*}
|\L^nf(x_1,\ol{u}) - \L^nf(x_2,\ol{v})| \leq |f|_\alpha \cdot \Big[w_{n,\alpha} \cdot  &d(\ol{u},\ol{v})^\alpha +(\tau_{n,\alpha}+2^{-n\alpha})\cdot \rho(x_1,x_2)^\alpha\Big]\\ &+C\cdot \|f\|_\infty \cdot \rho(x_1,x_2)^\alpha.
\end{align*}
\end{prop}

\begin{proof}
If $x_1$ and $x_2$ do not have the same $0$-th symbol, then $\rho(x_1,x_2) =1$ and 
$$
|\L^nf(x_1,\ol{u}) - \L^nf(x_2,\ol{v})| \leq \sum\limits_{\s^n y_1 = x_1}g^{(n)}(y_1)\|f\|_\infty +\sum\limits_{\s^n y_2 = x_2}g^{(n)}(y_2)\|f\|_\infty = 2\|f\|_{\infty}
$$
where the last equality follows from the fact that $\sum\limits_{\s^n y = x} g^{(n)}(y) = (L_{\log g}^n 1)(x) = 1$ for any $x \in \Sig^+$ and $n\in\N$.

So we may assume that the 0-th symbols of $x_1$ and $x_2$ agree.
In order to control the \hol norm of $\L^n f$, we consider the difference
$$\L^nf(x_1,\ol{u}) - \L^nf(x_2,\ol{v}) =\sum\limits_{\substack{\s^ny_i=x_i\\ i\in \{1,2\}}}\Big[g^{(n)}(y_1)f(y_1, \ol{\A^{[n]}(y_1)u})-g^{(n)}(y_2)f(y_2, \ol{\A^{[n]}(y_2)v})\Big].$$
The difference within the rectangular bracket may be estimated by the sum of four differences: 
\begin{equation}\label{diff 1}
g^{(n)}(y_1) f(y_1, \ol{\A^{[n]}(y_1)u})-g^{(n)}(y_1) f(y_1, \ol{\A^{[n]}(y_1)v})
\end{equation}
and
\begin{equation}\label{diff 15}
g^{(n)}(y_1) f(y_1, \ol{\A^{[n]}(y_1)v})-g^{(n)}(y_1) f(y_1, \ol{\A^{[n]}(y_2)v})
\end{equation}
and
\begin{equation}\label{diff 2}
g^{(n)}(y_1) f(y_1, \ol{\A^{[n]}(y_2)v})-g^{(n)}(y_1) f(y_2, \ol{\A^{[n]}(y_2)v})
\end{equation}
and 
\begin{equation}\label{diff 3}
g^{(n)}(y_1) f(y_2, \ol{\A^{[n]}(y_2)v})-g^{(n)}(y_2)f(y_2, \ol{\A^{[n]}(y_2)v}).
\end{equation}

Since $f$ is $\alpha$-H\"older, the absolute value of \eqref{diff 1} is bounded above by 
$$g^{(n)}(y_1)\cdot |f|_\alpha  \cdot d(\ol{\A^{[n]}(y_1)u},\ol{\A^{[n]}(y_1)v})^\alpha.$$
The sum of these terms over all $y_1 \in \s^{-n}x_1$ is bounded above by $w_{n,\a} \cdot |f|_\a \cdot d(\ol{u},\ol{v})^\a$.
%The sum over all $y_1$ is then bounded above by $w_{n,\alpha} \cdot d(\ol{u},\ol{v})^\alpha$ 
%\mpcomment{I think adding in which term is bounded by $w_{n,\alpha}$, $\tau_{n,\alpha}$ would be helpful}

Similarly, the absolute value of \eqref{diff 15} is bounded above by
$$g^{(n)}(y_1)\cdot |f|_\alpha  \cdot d(\ol{\A^{[n]}(y_1)v},\ol{\A^{[n]}(y_2)v})^\alpha,$$
and the sum of these terms is bounded above by $\tau_{n,\a}\cdot |f|_\a\cdot \rho(x_1,x_2)^\a$.

Likewise, the absolute value of \eqref{diff 2} is bounded above by
$$g^{(n)}(y_1)\cdot |f|_\alpha \cdot \rho(y_1,y_2)^\alpha = g^{(n)}(y_1)\cdot |f|_\alpha \cdot (2^{-n} \rho(x_1,x_2))^\a .$$

Since $\log g$ is \hol continuous, for sufficiently small $\a >0$ (i.e., smaller than the \hol exponent of $\log g$) we have $|\log g^{(n)}(y_1) - \log g^{(n)}(y_2)|\leq K \rho (x_1,x_2)^\alpha$ for some $K$. Then, the absolute value of \eqref{diff 3} is bounded above by
$$\|f\|_\infty \cdot g^{(n)}(y_1)\cdot | e^{K\rho(x_1,x_2)^\alpha}-1|\leq \|f\|_\infty \cdot g^{(n)}(y_1)\cdot K_1 \rho(x_1,x_2)^\alpha$$
for some $K_1>0$. 

Putting all together, this gives
\begin{align*}
|\L^nf(x_1,\ol{u}) - \L^nf(x_2,\ol{v})| \leq |f|_\alpha \cdot \Big[w_{n,
\alpha}&\cdot  d(\ol{u},\ol{v})^\alpha +\Big(\tau_{n,\alpha}+2^{-n\alpha}\sum\limits_{\s^n y_1 = x_1} g^{(n)}(y_1)\Big)\cdot \rho(x_1,x_2)^\alpha\Big]\\
&+\|f\|_\infty\cdot K_1 \sum\limits_{\s^n y_1 = x_1} g^{(n)}(y_1) \cdot \rho(x_1,x_2)^\alpha.
\end{align*}
%\begin{align*}
%\frac{|\L^nf(x_1,\ol{u}) - \L^nf(x_2,\ol{v})|}{d((x_1,\ol{u}),(x_2,\ol{v}))^\alpha} &\leq 
% \sum\limits_{\s^n y_1 = x_1} \frac{e^{S_n\psi(y_1)}}{\lambda^n} \Big[\|f\|_\alpha \Big(\frac{d(\ol{\A^n(y_1)^{-1}u},\ol{\A^n(y_2)^{-1}v})^\alpha}{d(\ol{u},\ol{v})^\alpha}+2^{-n\alpha}\Big)+K_0\|f\|_\infty \Big]\\
%&\leq  \sum\limits_{\s^n y_1 = x_1} \frac{e^{S_n\psi(y_1)}}{\lambda^n}\Big[(w_{n,\alpha}+\tau_{n,\alpha}+2^{-n\alpha})\|f\|_\alpha +K_0\|f\|_\infty\Big].
%\end{align*}
%$$\frac{|\L^nf(x_1,\ol{u}) - \L^nf(x_2,\ol{v})|}{d((x_1,\ol{u}),(x_2,\ol{v}))^\alpha} \leq  \|f\|_\alpha \cdot \Big(w_{n,\alpha}+\tau_{n,\alpha}+2^{-n\alpha} \sum\limits_{\s^n y_1 = x_1} \frac{e^{S_n\psi(y_1)}}{\lambda^n} \Big)+ \|f\|_\infty\cdot K_0 \sum\limits_{\s^n y_1 = x_1} \frac{e^{S_n\psi(y_1)}}{\lambda^n}.$$
Since $\displaystyle \sum\limits_{\s^n y_1 = x_1} g^{(n)}(y_1) = 1$, this proves the proposition.
%Note $\displaystyle \sum_{y \colon \s^n y = x} e^{S_n\psi(y)} = (\T^n1)(x)$ where $\T$ is defined in \eqref{eq: T} with $\rho(\T) = \lambda$. Since $(\T^n 1)/\lambda^n$ converges to a positive function exponentially fast [Bowen], $(\T^n 1)(x)/\lambda^n$ is uniformly bounded above for all $x\in \Sig^+$ and $n$ large enough. 
%Moreover, $d(\ol{u},\ol{v})$ is uniformly bounded above by the diameter of $\P^{d-1}$. This proves the claim.
%for any $\ep>0$ there exists $n(\ep)\in \N$ such that $\displaystyle \sum_{y \colon \s^n y = x} e^{S_n\psi(y)} \leq (\lambda+\ep)^n$ for all $n \geq n(\ep)$.
%In the last inequality, we've used $\displaystyle \sum_{y \colon \s^n y = x} e^{S_n\psi(y)} \leq \lambda^n$ for every $x \in \Sig^+$. \textcolor{red}{This follows because }
%Now assuming \eqref{eq: limsup w_n,alpha} and \eqref{eq: limsup tau_n,alpha}, choose $\beta>0$ such that $$\max\Big\{ \displaystyle\limsup\limits_{n\to\infty} \frac{1}{n}\log \tau_{n,\alpha} , \displaystyle\limsup\limits_{n\to\infty} \frac{1}{n}\log w_{n,\alpha} \Big\} \leq  -2\beta<0.$$ Then $\max(w_{n,\alpha},\tau_{n,\alpha}) \leq e^{-\beta n}$ for all $n$ large enough. 
%By choosing $\ep>0$ sufficiently small and $n\in \N$ sufficiently large, we obtain the Lasota-Yorke inequality \eqref{eq: lasota-yorke} with the constants $\g:=  2 e^{-\beta n}+2^{-n\alpha}(1+\ep/\lambda)^n<1$ and $C:=K_0(1+\ep/\lambda)^n$.
\end{proof} 

\begin{cor}\label{cor: LY}
Suppose for $\alpha>0$ sufficiently small that we have
\begin{equation}\label{eq: limsup w_n,alpha}
\displaystyle\limsup\limits_{n\to\infty} \frac{1}{n}\log w_{n,\alpha}<0  
\end{equation}
and 
\begin{equation}\label{eq: limsup tau_n,alpha}
\displaystyle\limsup\limits_{n\to\infty} \frac{1}{n}\log \tau_{n,\alpha}<0.
\end{equation}
Then Theorem \ref{thm: main} holds for such $\a>0$.
Moreover, 
there exist $\beta \in (0,1)$ and $C>0$ such that for any $(x,\ol{u}),(y,\ol{v}) \in\Sig^+\times \P^{d-1}$,
$$\abs{\L^{n}f(x,\ol{u})-\L^{n}f(y,\ol{v})}\leq \beta^{n}|f|_{\alpha} +C\rho(x,y)^{\alpha}\norm{f}_{\infty}.$$
%In particular, \eqref{eq: limsup w_n,alpha} and \eqref{eq: limsup tau_n,alpha} imply Theorem \ref{thm: main}. 
%the Lasota-Yorke inequality \eqref{eq: lasota-yorke} holds true for any $f \in C^\alpha(\Sig^+ \times \mathbb{P}^{d-1})$.
\end{cor}
\begin{proof}
The first statement is an immediate corollary of Proposition \ref{prop: w_n,alpha}.
The second claim also easily follows since the diameter of $\P^{d-1}$ is finite and $\rho(x,y)$ is bounded above by 1 for any $x,y\in \Sig^+$.
\end{proof}

%of $L$ for some $N \in \N$, it suffices to show the following: there exist $\alpha,\beta>0$ such that for all $x \in \Sig^+$,
%\begin{equation}\label{exp contraction}
%\limsup\limits_{n \to \infty} \frac{1}{n}\log \Big[  \sup\limits_{\ol{u} \neq \ol{v}} \sum\limits_{\s^n y = x} \frac{e^{S_n\psi(y)}}{\lambda^n} \cdot \Big(\frac{d(\ol{\A^n(y)^{-1}u},\ol{\A^n(y)^{-1}v})}{d(\ol{u},\ol{v})}\Big)^\alpha\Big] <-\beta.
%\end{equation}
%%where $\A^n_t(y):=\A(y)^*\A(\s y)^*\ldots \A(\s^{n-1} y)^* = (\A^n(y))^*$.
%
%
%In order to establish \eqref{exp contraction}, we will prove a slightly stronger statement that
%\begin{equation}
%\displaystyle\limsup\limits_{n\to\infty} \frac{1}{n}\log w_{n,\alpha}<0
%\end{equation}
%for some $\alpha>0$. We begin by showing that the limsup in \eqref{eq: limsup w_n,alpha} is indeed a limit:

We have reduced proving Theorem \ref{thm: main} to establishing  \eqref{eq: limsup w_n,alpha} and \eqref{eq: limsup tau_n,alpha} for all $\a>0$ sufficiently small.
We will first focus on establishing \eqref{eq: limsup w_n,alpha}, and deal with \eqref{eq: limsup tau_n,alpha} in Subsection \ref{subsec: tau}.
The following lemma shows that the $\limsup$ in \eqref{eq: limsup w_n,alpha} is in fact a limit. This fact allows us to manipulate the condition \eqref{eq: limsup w_n,alpha} more effectively.
\begin{lem}\label{lem: subadditive} The sequence $\{\log w_{n,\alpha}\}_{n\in \N}$ is subadditive, and hence
$$\lim_{n\to \infty} \frac{1}{n}\log w_{n,\alpha} = \inf_{n\to \infty} \frac{1}{n} \log w_{n,\alpha}.$$
In particular, \eqref{eq: limsup w_n,alpha} would follow if we can show that there exists $n \in \N$ such that
\begin{equation}\label{eq: lim w_n,alpha}
\log w_{n,\alpha} =\log \max\limits_{x \in \Sig^+} t_{n,\alpha}(x) <0.
\end{equation}
\end{lem}
\begin{proof}
%Note that $e^{S_n\psi(y)}/\lambda^n$ is already multiplicative, so we focus controlling the second term of $t_{n,\alpha}(x)$.
We demonstrate the idea for $w_n:=w_{n,1}$. The proof readily extends to $w_{n,\alpha}$ for any $\alpha>0$. 

Let $m,n\in \N$, and $x \in \Sig^+$ be the point that achieves the maximum in $w_{m+n}$. Then
\begin{align*}
w_{m+n} &= \sup\limits_{\ol{u} \neq \ol{v}} \sum\limits_{\s^{m+n} y = x}g^{(m+n)}(y)\cdot \frac{d(\ol{\A^{[m+n]}(y)u},\ol{\A^{[m+n]}(y)v})}{d(\ol{u},\ol{v})}\\
&=\sup\limits_{\ol{u} \neq \ol{v}} \sum\limits_{\s^{m+n} y = x} g^{(n)}(\s^m y) \cdot \frac{d(\ol{\A^{[n]}(\s^m y)u},\ol{\A^{[n]}(\s^m  y)v})}{d(\ol{u},\ol{v})} \cdot  g^{(m)}(y) \cdot \frac{d(\ol{\A^{[n+m]}(y)u},\ol{\A^{[n+m]}(y)v})}{d(\ol{\A^{[n]}(\s^m y)u},\ol{\A^{[n]}(\s^m  y)v})}\\
&=\sup\limits_{\ol{u} \neq \ol{v}} \sum\limits_{\s^{n} z = x}g^{(n)}(z)\cdot \frac{d(\ol{\A^{[n]}(z)u},\ol{\A^{[n]}(z)v})}{d(\ol{u},\ol{v})}\Big(\sum\limits_{\s^{m} y = z} g^{(m)}(y)\cdot \frac{d(\ol{\A^{[m]}(y)u_z},\ol{\A^{[m]}(y)v_z})}{d(\ol{u_z},\ol{v_z})}\Big)
\end{align*}
where $u_z:=\A^{[n]}(z)u$ and $v_z:=\A^{[n]}(z)v$. It then follows that
$$w_{m+n}\leq \sup\limits_{\ol{u} \neq \ol{v}} \sum\limits_{\s^{n} z = x} g^{(n)}(z) \cdot \frac{d(\ol{\A^{[n]}(z)u},\ol{\A^{[n]}(z)v})}{d(\ol{u},\ol{v})}\cdot  t_m(z) 
\leq w_n \cdot w_m.
$$

This shows that $\{\log w_{n}\}_{n \in \N}$ is subadditive. The second claim follows from Fekete's subadditive lemma. The last claim is a trivial consequence of the second claim.
\end{proof}
%The catch here is that by putting ``$\sup\limits_{u\neq v}$'' in front of the expression and without passing from the sum \eqref{exp contraction} to the integral yet, we obtain (almost) subadditivity of $w_{n,\alpha}$.
%In particular, $\displaystyle\frac{1}{n}\log w_{n,\alpha}$ converges to $\inf\limits_{n \to \infty} \displaystyle\frac{1}{n}\log w_{n,\alpha}$, and \eqref{eq: limsup w_n,alpha} would follow once show that 
%\begin{equation}
%\exists~N\in \N,~\alpha>0 \text{ such that } \log w_{N,\alpha}<0.
%\end{equation}
\subsection{Passing to the integral with respect to $\hmu^s_x$}

We will express $t_{n,\alpha}$ as an integral over $\hmu^s_x$. Recalling that $\hmu= \{\hmu^s_x\}_{x\in \Sig}$ is the disintegration of $\hmu$ along the local stable sets and $P\colon \Sig \to \Sig^+$ is the canonical projection, we begin by establishing the Gibbs-like property on $\hmu^s_x$. 
%For the following lemma, we write $A\asymp_C B$ if $C^{-1} \leq A/B \leq C$.
\begin{lem}\label{lem: unif comparison} 
For any $x=(x_i)_{i\in \N_0}\in \Sig^+$ and $\I=i_0\ldots i_{n-1}$ such that $y:=i_0\ldots i_{n-1}x_0x_1\ldots \in \Sig^+$ is admissible, we have
$$\hmu^s_x(\{\hy \in \Wloc^s(x) \colon P(\s^{-n}\hy) = y\}) \asymp g^{(n)}(y).$$
%with uniform comparability constant $C$ independent of $x,y$ and $n\in\N$.
\end{lem}
\begin{proof}
From $\s$-invariance of $\hmu$, $P_*\hmu = \mu$, and the Gibbs property of $\mu$, we have
$$\hmu(\{\h\omega \in \Sig \colon \omega_{i-n} = y_i \text{ for }0\leq i \leq n-1 \text{ and }\omega_0 = x_0 \})=\mu([\I x_0]) \asymp g^{(n+1)}(y).$$
%where $[\I] \subset \Sig^+$ is the $n$-cylinder containing $y$.
Since the term on the left is equal to $$ \int_{[x_0]} \hmu^s_x(\{\hy \in \Wloc^s(x) \colon P(\s^{-n}\hy) = y\})\,d\mu(x)$$ 
and the disintegrated measures $\{\hmu^s_x\}_{x \in \Sig^+}$ are absolutely continuous \eqref{eq: abs cty} with respect to each other, 
this integral is uniformly comparable to $$\hmu^s_x(\{\hy \in \Wloc^s(x) \colon P(\s^{-n}\hy) = y\}) \cdot \mu([x_0]) \asymp \hmu^s_x(\{\hy \in \Wloc^s(x) \colon P(\s^{-n}\hy) = y\}) \cdot g(x).$$
Noting that $\s^n y=x$, canceling off $g(x)$ from both sides proves the lemma. 
%the $\hmu^s_x$-measure of the set $\{\hy \in \Wloc^s(x) \colon P(\s^{-n}\hy) = y\}$ must also be uniformly comparable to $g^{(n)}(y)$ for every $x \in \Sig^+$.
%Moreover, the $\hmu^s_x$ measure of the $n$-cylinder $\{\hy \in \W^s_{loc}(x) \colon \pi(\hs^{-n}\hy) = y\}$ is uniformly comparable to $e^{S_n\psi(y)}/\lambda^n$ from the Gibbs property of $\mu$ (\textcolor{red}{check this statement; roughly, $\hmu(\{\hx \in \Sig \colon x_{i-n} = y_i, ~i \in \{0,\ldots,n-1\}\})$ is equal to $\mu([I]) \asymp e^{S_n\psi(y)}/\lambda^n$ where $[I] \subset \Sig^+$ is the $n$-cylinder containing $y$. Since $\{\hmu^s_x\}_{x \in \Sig^+}$ are absolutely continuous \eqref{abs cty}, the $\hmu^s_x$-measure of $\{\hx \in W^s_{loc}(x) \colon x_{i-n} = y_i, ~i \in \{0,\ldots,n-1\}\}$ must also be uniformly comparable to $e^{S_n\psi(y)}/\lambda^n$ for \textit{every} $x \in \Sig^+$}). }
\end{proof}

We will now switch to notation more suited to the two-sided subshift $\Sig$ by using the adjoint cocycle $\A_*$ over $(\Sig,\s^{-1})$. 

Consider any $x,y \in \Sig^+$ with $\s^n y = x$ and any $\s^{-n}\hy_1,\s^{-n}\hy_2 \in \Wloc^s(y)$. Since $\A$ is constant along the local stable sets \eqref{constant along stable}, we have $\A^{n}( \hy_1)^*=\A^{n}(\hy_2)^*$. This observation together with the previous lemma give
$$t_{n,\alpha}(x) \asymp \sup\limits_{\ol{u} \neq \ol{v}} \int \Big[\frac{d(\ol{\A_*^{n}(\hx)u},\ol{\A_*^{n}(\hx)v})}{d(\ol{u},\ol{v})}\Big]^\alpha \, d\hmu^s_x(\hx).$$

Similar to the arguments from \cite[Ch. 5, Lemma 2.3]{MR886674}, in order to establish \eqref{eq: lim w_n,alpha} for all $\a>0$ sufficiently small, it suffices to show that there exists $n\in \N$ such that 
\begin{equation}\label{eq: wts}
\sup\limits_{\substack{\ol{u} \neq \ol{v} \\ x\in \Sig^+}} \int\log \Big(\frac{d(\ol{\A_*^{n}(\hx)u},\ol{\A_*^{n}(\hx)v})}{d(\ol{u},\ol{v})}\Big)\, d\hmu^s_x(\hx)<0.
\end{equation}
Indeed assuming \eqref{eq: wts}, we establish \eqref{eq: lim w_n,alpha} as follows: let $\alpha_0 = \alpha_0(n)>0$ be a sufficiently small constant such that 
$\displaystyle\alpha\log \Big(\frac{d(\ol{\A_*^{n}(\hx)u},\ol{\A_*^{n}(\hx)v})}{d(\ol{u},\ol{v})}\Big)$ is sufficiently close to 0 for all $\a \in (0,\a_0)$, $\hx \in \Sig$, and $(\ol{u}, \ol{v}) \in \mathbb{P}^{d-1}\times \mathbb{P}^{d-1}\setminus \Delta$.
Applying the identity
$$e^{r} \leq 1+r +\frac{r^2}{2}e^{|r|} \text{ for } |r| \approx 0$$
to $\displaystyle r = \alpha\log \Big(\frac{d(\ol{\A_*^{n}(\hx)u},\ol{\A_*^{n}(\hx)v})}{d(\ol{u},\ol{v})}\Big)$ and integrating with respect to $\hmu^s_x$,
we have
$$\int\Big(\frac{d(\ol{\A_*^{n}(\hx)u},\ol{\A_*^{n}(\hx)v})}{d(\ol{u},\ol{v})}\Big)^\alpha \, d\hmu_x^s(\hx)
\leq 1+\alpha\int  \log \Big(\frac{d(\ol{\A_*^{n}(\hx)u},\ol{\A_*^{n}(\hx)v})}{d(\ol{u},\ol{v})}\Big)\,d\hmu^s_x(\hx).$$
Taking the supremum over all $x\in \Sig^+$ and $(\ol{u}, \ol{v})\in \mathbb{P}^{d-1} \times \mathbb{P}^{d-1} \setminus \Delta$ followed by the logarithm, \eqref{eq: wts} then translates to the required inequality $\log w_{n,\alpha}<0$ of \eqref{eq: lim w_n,alpha}. 

So far, we have reduced proving Theorem \ref{thm: main} to establishing \eqref{eq: limsup tau_n,alpha} and \eqref{eq: wts}. Note that we did not yet make use of the 1-typicality assumption on $\A$ in deductions thus far.

\subsection{Making use of the gap in Lyapunov exponents}
Denoting the $i$-th singular value of $A$ by $\s_i(A)$, the angular metric $d$ defined as in \eqref{eq: d} satisfy 
\begin{equation}\label{eq: property of d}
\frac{d(\ol{\A_*^{n}(\hx)u},\ol{\A_*^{n}(\hx)v})}{d(\ol{u},\ol{v})} \leq \frac{\s_1(\A_*^{n}(\hx))\cdot \s_2(\A_*^{n}(\hx))}{\|\A_*^{n}(\hx)u\|\cdot \|\A_*^{n}(\hx)v\|}.
\end{equation}
Using this inequality together with the 1-typicality assumption on $\A$, we will establish \eqref{eq: wts} in this subsection.

We begin with a general result which follows from Kingman's subadditive ergodic theorem. In particular, it does not require any assumption on $\A$.
\begin{prop}\label{prop: growth Kingman} The Lyapunov exponents $\{\lambda_i\}_{1\leq i \leq d}$ of $\A$ with respect to $\hmu$ listed in a decreasing order satisfy the following:
\begin{enumerate}
\item
For $\hmu$-a.e. $\hx$,
$$\lim_{n\to\infty}\frac{1}{n}\log \|\A_*^{n}(\hx)^{\wedge k}\| =\sum\limits_{i=1}^{k} \lambda_{i}.$$
\item
For every $x \in \Sig^+$ and $\hmu^s_x$-a.e. $\hx$,
$$\lim_{n\to\infty}\frac{1}{n}\log \|\A_*^{n}(\hx)^{\wedge k}\| =\sum\limits_{i=1}^{k} \lambda_{i}.$$
\item 
The convergence in 
$$\lim_{n\to\infty}\frac{1}{n}\int \log \|\A_*^{n}(\hx)^{\wedge k}\| d\hmu^s_x(\hx)= \sum\limits_{i=1}^{k} \lambda_{i}.$$
is uniform in $x \in \Sig^+$.
\end{enumerate}
\end{prop}
\begin{proof} (1) is due to Kingman's subadditive ergodic theorem applied to the adjoint cocycle $\A^{*}$.

For (2), the full $\hmu$-measure subset in which the statement of (1) holds is unstable-saturated; that is, if $\hx$ belongs to such a subset, then its entire local unstable set $\Wloc^u(\hx)$ also belongs to the subset. Hence, we can promote ``$\hmu$-a.e. $\hx$'' from (1) into ``every $x \in \Sig^+$ and $\hmu^s_x$-a.e. $\hx$.''

For (3), it is clear from (2) that the integral converges for every $x\in\Sig^+$. The uniform convergence follows from Lemma \ref{lem: unif comparison} and \eqref{eq: unif comparison}.
%For any $x,y \in \Sig^+$ with $x_0 = y_0$ and $\I = i_0\ldots i_{n-1} \in \L(n)$, the $\hmu$
% $\Wloc^s(x)$ in the same $1$-cylinder, consider the
%
%Approximating the integral over cylinders and using (uniform continuity of the) unstable holonomies as well as the absolute continuity \eqref{abs cty} of $\hmu^s_x$, the uniform convergence in $x \in \Sig^+$ also follows.
\end{proof}

The following proposition relies on the construction and the properties of the measures considered in Subsection \ref{subsec: invariant measures}.

\begin{prop}\label{prop: growth for all u} If $\A$ is 1-typical, then for every $x \in \Sig^+$, $\hmu^s_x$-a.e. $\hx$, and any unit vector $u \in \R^d\setminus \{0\}$,
\begin{equation}\label{typicality usage}
\lim_{n \to \infty}\frac{1}{n}\int \log \|\A_*^{n}(\hx)u\|\,d\hmu^s_x(\hx) =\lambda_1.
\end{equation} 
Moreover, the convergence is uniform in $x$ and $u$.
\end{prop}

%We first prove the following lemma which governs the growth rate of $\|\A^{-n}(\hx)u\|$ relative to $\|\A^{-n}(\hx)\|$: \textcolor{red}{(need revision from here below; not everything may be correct)}

In order to prove this proposition, we need the following preliminary lemma where we use the same notations from Remark \ref{rem: inv measures}. In particular, $\xi(\hx)$ is the one-dimensional top Lyapunov subspace of $\A$ with respect to $\hmu$, and we denote by $H^{s,*}$ the local stable holonomies of the adjoint cocycle $\A_*$ over $(\Sig,\s^{-1})$.
%$\omega_*(\hx)$ is obtained by applying the construction for $\xi(\hx)$ to the adjoint of the inverse cocycle $\A_*^{-1}$.

\begin{lem}\label{lem: growth rate of u} Let $\A$ be 1-typical. Then for any $x \in \Sig^+$, $\hmu^s_x$-a.e. $\hx$, and any sequence of unit vectors $\{u_n\}_{n \in\N}$ converging to some $u \in \R^d \setminus \{0\}$,
we have
$$\lim\limits_{n \to \infty} \frac{\|\A_*^{n}(\hx)u_n\|}{\|\A_*^{n}(\hx)\|} = |\langle u,\xi(\hx)\rangle|.$$
Moreover, for any $y \in \Sig^+$ with the same 0-symbol as $\hx$, we have
$$
\lim_{n \to \infty}\frac{\|\A_*^{n}(\hx)H^{s,*}_{\hx_y,\hx}v\|}{\|\A_*^{n}(\hx)\|} =|\langle v, \xi(\hx_y) \rangle|
$$
where $\h{x}_y:=[\hx,y]$.
\end{lem}
\begin{proof}
%Let $\A^n_t(\hx) = K_nA_nU_n$ be the Cartan decomposition. 
%Denoting the $KAK$-decomposition of $\A^n(\hs^{-n}\hx)$ by $K_nA_nU_n$, Bonatti-Viana showed that for $\hmu$-a.e. $\hx$, $\displaystyle \frac{\A^n(\hs^{-n}\hx)}{\|\A^n(\hs^{-n}\hx)\|}$ converges to a rank 1 quasi-projective map $Q$ whose image coincides with $\ol{\xi(\hx)} \in \mathbb{P}^{d-1}$. In particular, $K_n$ converges to some $K_\infty \in O(d)$ with $K_\infty e_1 = \xi(\hx)$.
Let $B:=\{\hx\in \Sig \colon \xi(\hx) \text{ is well-defined}\}$. Then $\hmu(B)=1$ and $B$ is unstable-saturated as $\xi(\hx)$ is $H^u$-invariant. 

For $\hx \in B$, we have seen in Remark \ref{rem: inv measures} that $\xi(\hx)$ is equal to $ K_\infty e_1$, where $K_n$ is from the $KAK$-decomposition $\A^{n}(P(\s^{-n}\hx)) = K_nA_nU_n$ and $K_\infty$ is the subsequent limit of $K_n$. Then $\A_*^{n}(\hx)$ is equal to $(K_nA_nU_n)^*=U_n^*A_nK^*_n$.
%Keeping up with the notations from above, for $\hx \in B$ and $n \in \N$, we write $K_nA_n^{-1}U_n=(\A^{-1}_*)^n(\hs^{-n}\hx)$ so that $K_n \to K_\infty$ and $\omega_*(\hx) = K_\infty e_d$.
Since $U_n \in O(d)$, we have
$$\frac{\|\A_*^{n}(\hx)u_n\|^2}{\|\A_*^{n}(\hx)\|^2}= \frac{\|(U_n^*A_nK^*_n)u_n\|^2}{\|\A_*^{n}(\hx)\|^2}=\frac{\|A_nK_n^*u_n\|^2}{\|A_n\|^2} \xrightarrow{n\to\infty} \langle K_\infty^* u, e_1\rangle^2=\langle u, \xi(\hx)\rangle^2.$$
Since $B$ is unstable-saturated, this holds true for \textit{every} $x$ and $\hmu^s_x$-a.e. $\hx$
%Similar to the proof of (2) of Proposition \ref{prop: growth Kingman}, the subset $\{\omega_*(\hx) \colon \xi(\hx) \text{ is well-defined}\}$ of $\Sig$ is unstable-saturated, and hence, the observation from the previous paragraph is true for \textit{every} $x$ and $\hmu^s_x$-a.e. $\hx$. For such $\hx$, we have
%$$\frac{\|\A^{-n}(\hx)u_n\|^2}{\|\A^{-n}(\hx)\|^2} =\frac{\|A_n^{-1}K_n^tu_n\|^2}{\|A_n^{-1}\|^2} \to \textcolor{red}{\langle K_\infty^t u, e_d\rangle^2=\langle u, \omega(\hx)\rangle^2,}$$
%\textcolor{red}{where $\omega(\hx):=K_\infty e_d$. In relation to the Oseledet's theorem, $\omega(\hx)^\perp$ defines $W_2$, the slower subspace consisting of vectors that do not grow as $-\lambda_d$ in the past.}
%See Bougerol-Lacroix for details.  

For the second statement, using the identity $H^{s,*}_{\hx,\hy} =(H^u_{\hy,\hx})^*$ we have
$$
\lim_{n \to \infty}\frac{\|\A_*^{n}(\hx)H^{s,*}_{\hx_y,\hx}v\|}{\|\A_*^{n}(\hx)\|} = |\langle H^{s,*}_{\hx_y,\hx}v, \xi(\hx) \rangle| =|\langle v,H^{u}_{\hx,\hx_y} \xi(\hx) \rangle|.
$$ 
The claim then follows by noticing $\xi(\hx_y) = H^{u}_{\hx,\hx_y} \xi(\hx)$ since the top Lyapunov subspace $\xi(\hx)$ is preserved under $H^u$.
%$ |\langle H^{s,*}_{\hx_y,\hx}v, \omega_*(\hx) \rangle|$
\end{proof}

%Since $\A^{-1}_*$ is 1-typical if and only if $\A$ is 1-typical, we apply the construction for $m$ and $\hm$ of $\A$ to $\A^{-1}_*$, and obtain $\eta$ and $\h\eta$ satisfying analogous properties from Proposition \ref{prop: nu}: the conditional measures $\h\eta_{\hx}$ of $\h\eta=\{\h\eta_{\hx}\}$ are dirac masses sitting at $\omega_*(\hx)$ whereas the conditional measures $\eta_x$ of $\eta = \{\eta_x\}_{x \in \Sig^+}$ are proper (i.e., gives zero mass to any proper subspace) and vary continuously in $x\in \Sig^+$. In particular, $\eta_x$ satisfies the analogue of \eqref{eq: m and hm}: $\eta_x$ is equal to the integral of $\h\eta_{\hx} = \delta_{\omega_*(\hx)}$ over $\Wloc^s(x)$ with respect to $\hmu^s_x$. 
%We use these properties of $\rho$ in the proof of Proposition \ref{prop: growth for all u} below.

\begin{proof}[Proof of Proposition \ref{prop: growth for all u}]
For any $x \in \Sig^+$ and unit $u \in \R^d \setminus \{0\}$, we define $$B_{x,u}:=\{\hx \in \Wloc^s(\hx) \colon u \perp \xi(\hx)\} = \{\hx \in \Wloc^s(\hx) \colon \xi(\hx) \in u^\perp\}.$$
Recalling the $F_\A$-invariant probability measure $m$ from Proposition \ref{prop: nu}, we have $$\hmu_{x}^s(B_{x,u}) = \hmu_{x}^s(\{\hx \in \Wloc^s(x) \colon \xi(\hx ) \in u^\perp \})= \int \delta_{\xi(\hx)}(u^\perp)d\hmu_{x}^s(\hx)=m_x(u^\perp) =0.$$
The last equality is due to the fact that $m_x$ is proper for \textit{every} $x \in\Sig^+$. Then Proposition \ref{prop: growth Kingman} (2) and Lemma \ref{lem: growth rate of u}
give
$$
\lim\limits_{n \to \infty} \frac{1}{n}\log\|\A_*^{n}(\hx)u\| = \lambda_1$$
for every $x \in \Sig^+$ and $\hmu^s_x$-a.e. $\hx$, which then yields \eqref{typicality usage}.
%Recalling that $\displaystyle \hmu = \int \hmu_{x}^s d\mu(x)$, integrating with respect to $\mu$ gives
%$$\hmu(B_u) = \int \hmu_{x}^s(B_u \cap W^s_{loc}(x)) d\mu(x) = 0,$$ 
%completing the proof of the claim. 
%$$\lim\limits_{n \to \infty}\int \frac{1}{n}\log \|\A^n_t(\hx) u\|d\mu(\hx) = \lambda_1(\mu).$$

%(\textcolor{red}{is the following true? probably, but might need some uniform convergence? Not too confident with what I've written below. If we can make sure this is true, then I think we are in good shape}). 
Since $\Sig^+ \times \mathbb{S}^{d-1}$ is compact, the uniform convergence of \eqref{typicality usage} in $x \in \Sig^+$ and $u \in \mathbb{S}^{d-1}$ would follow if for any $x_n \to x \in \Sig^+$ and $u_n \to u \in \mathbb{S}^{d-1}$, we have
\begin{equation}\label{uniform conv}
\lim\limits_{n\to\infty}\frac{1}{n}\int \log \|\A_*^{n}(\hx)u_n\|\,d\hmu^s_{x_n}(\hx) = \lambda_1.
\end{equation}
In order to establish \eqref{uniform conv}, we consider a full $\hmu^s_x$-measure subset $U_x$ of $\Wloc^s(x)$ such that for any $\hx \in U_x$, we have
\begin{enumerate}
\item $\xi(\hx)$ exists, and
\item $\hx$ does not belong to $B_{x,u}$, and 
\item $\displaystyle \lim\limits_{n\to\infty}\frac{1}{n}\log\|\A_*^{n}(\hx)\|=\lambda_1$.
\end{enumerate} 

For any $\hx \in U_x$ and $n\in \N$, let $\hx_n:=[\hx,x_n]$ be the unique point in the intersection between $\Wloc^s(x_n)$ and $\Wloc^u(\hx)$. 
%Denoting by $H^{s,*}$ the local stable holonomies of the adjoint cocycle $\A_*$ over $(\Sig,\s^{-1})$,
Then we have 
$$\A_*^{n}(\hx_n)u_n = H^{s,*}_{\s^{-n}\hx,\s^{-n}\hx_n}\A_*^{n}(\hx) H^{s,*}_{\hx_n,\hx} u_n.$$
Denoting $H^{s,*}_{\hx_n,\hx} u_n$ by $\wt{u}_n$, we have $\wt{u}_n \to u$ since $\hx_n \to \hx$ and $u_n \to u$. Moreover, $H^{s,*}_{\s^{-n}\hx,\s^{-n}\hx_n} \to I$ as $n\to \infty$. The first statement of Lemma \ref{lem: growth rate of u} then gives
$$\frac{\|\A_*^{n}(\hx_n)u_n\|}{\|\A_*^{n}(\hx)\|}=\frac{ \|H^{s,*}_{\s^{-n}\hx,\s^{-n}\hx_n}\A_*^{n}(\hx)\wt{u}_n\|}{\|\A_*^{n}(\hx)\|} \xrightarrow{n\to \infty} |\langle u,\xi(\hx)\rangle|.$$
%In particular, $\displaystyle\frac{1}{n}\log \|\A^n_t(\hx_n)u_n\|$ limits to $-\lambda_d$ .
Since $x \mapsto \hmu^s_x$ is continuous from absolute continuity  \eqref{eq: abs cty} of $\{\hmu^s_x\}_{x\in \Sig^+}$ and $\displaystyle \frac{1}{n}\log\|\A_*^{n}(\hx)\|$ limits to $\lambda_1$ from the choice of $\hx \in U_x$, together we obtain \eqref{uniform conv} as required.
%\textcolor{red}{I was trying to prove the uniform convergence directly, but didn't quite succeed. So this was the only reasonable way I could make it work, but not too confident with the proof.}
\end{proof}

Then \eqref{eq: wts} easily follows:
%Recall that it suffices to show \eqref{eq: wts}.
from \eqref{eq: property of d} and Proposition \ref{prop: growth Kingman} and \ref{prop: growth for all u}, there exists $n \in \N$ such that
\begin{align*}
\frac{1}{n}\int\log \Big(&\frac{d(\ol{\A_*^{n}(\hx)u},\ol{\A_*^{n}(\hx)v})}{d(\ol{u},\ol{v})}\Big)\,d\hmu^s_x(\hx) \\
&=\frac{1}{n}\int \Big(\log \|\A_*^{n}(\hx)^{\wedge 2}\| - \log\|\A_*^{n}(\hx)u\|-\log\|\A_*^{n}(\hx)v\|\Big) \,d\hmu^s_x(\hx)\\
&< \frac{1}{2}(\lambda_{2}-\lambda_{1})\\
&<0
\end{align*}
for every $x\in \Sig^+$ and $(\ol{u},\ol{v})\in \mathbb{P}^{d-1}\times \mathbb{P}^{d-1} \setminus \Delta$. In the last inequality, we have used the result of Bonatti and Viana \cite{bonatti2004lyapunov} that the top Lyapunov exponent $\lambda_1$ is simple when $\A$ is 1-typical. This establishes \eqref{eq: wts} as required.

\subsection{Controlling $\tau_{n,\alpha}$}\label{subsec: tau}
We will complete the proof of Theorem \ref{thm: main} by showing \eqref{eq: limsup tau_n,alpha} which states that
$$\displaystyle\limsup\limits_{n\to\infty} \frac{1}{n}\log \tau_{n,\alpha}<0.$$
  
%need to show that the following term is exponentially decreasing for all $x_1,x_2 \in \Sig$:
%$$\frac{1}{d(x_1,x_2)^\alpha} \sum\limits_{\s^n y_i =x_i}\frac{e^{S_n\psi(y_1)}}{\lambda^n}\cdot \|f\|_\alpha  \cdot d(\ol{\A^n(y_1)^{-1}v},\ol{\A^n(y_2)^{-1}v})^\alpha.$$
Similar to how we passed from $t_{n,\alpha}$ to the integral, we have 
\begin{equation}\label{eq: tau integral}
\tau_{n,\alpha} \asymp \sup\limits_{\substack{x,y \in\Sig^+\\ \ol{v} \in \mathbb{P}^{d-1}}} \int \Big(\frac{ d(\ol{\A_*^{n}(\hx)v},\ol{\A_*^{n}(\hx_y)v})}{\rho(x,y)}\Big)^\alpha d\hmu^s_x(\hx)
\end{equation}
where the supremum is taken over all $x,y\in \Sig^+$ with the common 0-symbol and all $\ol{v}\in \P^{d-1}$.
% and $\hx_y =[\hx,y]$.
This is because for any $x,y\in \Sig^+$ with the common 0-symbol and $z_1 \in \s^{-n}x$ and $z_2 \in \s^{-n}y$ in the same $n$-cylinder, we have $$d(\A^{[n]}(z_1),\A^{[n]}(z_2)) = d(\A_*^n(\hx),\A_*^n(\hx_y))$$
for any $\hx \in \s^n(\Wloc^s(z_1))$. Lemma \ref{lem: unif comparison} then gives \eqref{eq: tau integral}.
%\textcolor{red}{I think it would be nice to have some explanation of why the $\hat{x}_{y}$ is here. }
%Applying \eqref{eq: abs cty} gives
%$$\int J_{y,x}(\hy) d(\A^{-n}(\hy_x)v,\A^{-n}(\hy)v)d\hmu^s_y(\hy)$$
%which doesn't really help. So I think we got to separate it as follows:

Writing $\A_*^{n}(\hx_y) = H^1_{\hx_y}\A_*^{n}(\hx)H^2_{\hx_y} $ where $H^1_{\hx_y} = H^{s,*}_{\s^{-n}\hx,\s^{-n}\hx_y}$ and $H^2_{\hx_y} = H^{s,*}_{\hx_y,\hx}$, 
the integral appearing in \eqref{eq: tau integral} is bounded above by
$$
\int  \Big(\frac{d(\ol{H^1_{\hx_y}\A_*^{n}(\hx)H^2_{\hx_y}v},\ol{\A_*^{n}(\hx)H^2_{\hx_y}v})}{\rho(x,y)}\Big)^\alpha \,d\hmu^s_x(\hx)+
\int \Big( \frac{ d(\ol{\A_*^{n}(\hx)H^2_{\hx_y} v},\ol{\A_*^{n}(\hx)v})}{\rho(x,y)}\Big)^\alpha\, d\hmu^s_x(\hx),
$$ 
and it suffices to show that both terms are decreasing exponentially fast to 0. 
%\mpcomment{Is the $-$ in $H^{s,-}_{\s^{-n}\hx,\s^{-n}\hx_y}$ an inverse? It might be useful to clarify that notation.}
%\kp{defined above that $H^{s,*}$ is the stable holonomies for $\A_*$}

For the first term, we apply \eqref{eq: holder holonomies 2} with $u = \A_*^{n}(\hx)H^2_{\hx_y}v$:
$$d(\ol{H^1_{\hx_y}u},\ol{u}) \leq C\cdot \rho(\s^{-n}\hx,\s^{-n}\hx_y)=C \cdot 2^{-n} \cdot \rho(x,y).$$
Hence, the first term is bounded above by $C^\a \cdot 2^{-n\alpha}$.
%$$
%d(\ol{H^1u},\ol{u})\lesssim\|H^1u-u\| \leq C\cdot d(\s^{-n}\hx,\s^{-n}\hx_y)=C \cdot 2^{-n} \cdot d(x,y).
%$$
%Hence, the second term is bounded above by $\|f\|_\alpha \cdot C\cdot 2^{-n\alpha}$; i.e., exponentially decreasing.

For the second term, apply \eqref{eq: property of d}:
\begin{align*}
d(\ol{\A_*^{n}(\hx)H^2_{\hx_y}v},\ol{\A_*^{n}(\hx)v}) &\leq d(\ol{H^2_{\hx_y} v},\ol{v}) \cdot  \frac{\s_1(\A_*^{n}(\hx))\cdot \s_2(\A_*^{n}(\hx))}{\|\A_*^{n}(\hx)H^2_{\hx_y} v\|\cdot \|\A_*^{n}(\hx)v\|}\\
&\leq C\cdot \rho(x,y) \cdot  \frac{\s_1(\A_*^{n}(\hx))\cdot \s_2(\A_*^{n}(\hx))}{\|\A_*^{n}(\hx)H^2_{\hx_y} v\|\cdot \|\A_*^{n}(\hx)v\|}.
\end{align*} 
%\mpcomment{There a bunch of places here where $v$ appears in () is this a typo or is that suppose to have meaning?}
All terms in the product except $\|\A_*^{n}(\hx)H^2_{\hx_y} (v)\|$ can be dealt with by proceeding similar to the above subsection. So it suffices to show the following analogue of \eqref{typicality usage}: for any $x,y\in \Sig^+$ with $x_0 = y_0$ and $v\in \mathbb{S}^{d-1}$, we have
\begin{equation}\label{eq: growth rate 2}
\lim\limits_{n\to \infty}\frac{1}{n}\int \log \|\A_*^{n}(\hx)H^{2}_{\hx_y}v\|\, d\hmu^s_x(\hx)= \lambda_1
\end{equation}
and the convergence is uniform in $x,y\in \Sig^+$ and $v \in \mathbb{S}^{d-1}$. 

%In view of Lemma \ref{lem: growth rate of u}, we have
%$$
%\lim_{n \to \infty}\frac{\|\A_*^{n}(\hx)H^{s,*}_{\hx_y,\hx}v\|}{\|\A_*^{n}(\hx)\|} = |\langle H^{s,*}_{\hx_y,\hx}v, \omega_*(\hx) \rangle| =|\langle v, \omega_*(\hx_y) \rangle|.
%$$ 
%\mpcomment{This could probably also be included as a part of Lemma 3.7}
%The last equality is due to the identity $(H^{s,-}_{\hx_y,\hx})^* = H^{u,-*}_{\hx,\hx_y}$ and the invariance of $\omega_*(\hx)$ under the unstable holonomies of $\A^{-1}_*$.

Recalling that $\hx_y = [\hx,y]$, we have
\begin{align*}
\hmu^s_{x}(\{\hx \in \Wloc^s(x) \colon v \perp \xi(\hx_y)\}) &=  \int \delta_{\xi([\hx,y])}(v^\perp)\,d\hmu^s_x(\hx)\\
&= \int J_{y,x}(\hy) \delta_{\xi([h_{y,x}(\hy),y])}(v^\perp)\,d\hmu^s_y(\hy)
\end{align*}
where the second equality is due to the absolute continuity \eqref{eq: abs cty} of $\{\hmu^s_x\}_{x\in \Sig^+}$. 
%$$\hmu^s_{x}(\{\hx \in \Wloc^s(x) \colon v \perp \omega_*(\hx_y)\}) = \int J_{y,x}(\hy) f(h_{y,x}(\hy))\,d\hmu^s_y(\hy) =0 .$$ 
Since
$[h_{y,x}(\hy),y]=[[\hy,x],y]= \hy$, this gives 
$$\hmu^s_{x}(\{\hx \in \Wloc^s(x) \colon v \perp \xi(\hx_y)\}) = \int J_{y,x}(\hy)\delta_{\xi(\hy)}(v^\perp)\, d\hmu^s_y(\hy)= 0$$ using the properness of $m_y$ as in the proof for Proposition \ref{prop: growth for all u}. This together with Proposition \ref{prop: growth Kingman} and the second statement of Lemma \ref{lem: growth rate of u} establishes \eqref{eq: growth rate 2}. The uniform convergence in \eqref{eq: growth rate 2} can also be shown as in the proof of Proposition \ref{prop: growth for all u} using the absolute continuity of $\{\hmu^s_x\}_{x\in \Sig^+}$.

\subsection{Quasi-compactness and the eigenmeasure of $\L$}
Using the Lasota-Yorke inequality established in Theorem \ref{thm: main}, we show that $\L$ is quasi-compact by applying the following result of Hennion to $X =C(\Sig^+ \times \P^{d-1})$, $T=\L$, $\|\cdot\|_1 = \|\cdot \|_\a$, $\|\cdot\|_2 = \|\cdot \|_\infty$. 
For details, see \cite{MR1129880} and references therein.

\begin{prop}\cite{MR1129880}\label{prop: hennion}
Let $(X,\|\cdot \|_1)$ be a Banach space and $T  \colon (X,\|\cdot\|_1) \to (X,\|\cdot\|_1)$ a bounded linear operator with spectral radius $\rho(T)$. Suppose there exists a norm $\|\cdot \|_2$ such that $T \colon (X,\|\cdot\|_1)\to (X,\|\cdot\|_2)$ is compact and there exist $\{R_n\}_{n\in \N},\{r_n\}_{n\in \N} \subset \R$ such that $r:=\liminf\limits_{n\to \infty} (r_n)^{1/n}<\rho(T)$ and for every $f \in X$
$$\|T^nf\|_1 \leq r_n \|f\|_1 + R_n\|f\|_2.$$
Then $T$ is quasi-compact and the essential spectrum of $T$ is less than or equal to $r$.
\end{prop}

%\textcolor{red}{I think the standard citation for this is:}
%\mpcomment{Citation}

%Under 1-typicality assumption on $\A$, the slowest Lyapunov subspace with respect to $\hmu$ is 1-dimensional, and it is spanned by $\omega(\hx)$ and $H^s$-invariant. Moreover, the unique $(\A,H^s)$-invariant measure on $\Sig\times \P^{d-1}$ is characterized by its conditional measures sitting on $\omega(\hx)$.
%
% Since the local stable holonomy $H^s$ of $\A$ is identically equal to $I$, it makes sense to denote such direction by $\omega(x)$ where $x = P(\hx)$. Indeed, $\omega(x)\in \P^{d-1}$ is the slowest Lyapunov subspace of $\A$ with respect to $\mu$ at $x\in \Sig^+$. 

We now establish properties of the eigenmeasure for $\L$. 
For any probability measure $\zeta$ on $\Sig^+ \times \mathbb{P}^{d-1}$ sitting over $\mu$ and continuous function $f \in C(\Sig^+ \times \mathbb{P}^{d-1})$, consider a function $\wt{f}$ on $\Sig^+$ defined by $\displaystyle \wt{f}(x):=\int f(x,\ol{\A(x)^{*}u})\,d\zeta_{\s x}(u).$
Recalling that $\mu$ is the eigenfunction of $L_{\log g}$, we have
\begin{align*}
\int \wt{f} \,d\mu =\int \wt{f}\,d(L_{\log g}^*\mu)= \int  L_{\log g}\wt{f}\, d\mu &=\int \sum\limits_{\s y = x}g(y)\wt{f}(y)\,d\mu(x)\\
&= \int\sum\limits_{\s y=x}g(y)\int f(y,\ol{\A(y)^{*}u})\, d\zeta_x(u)d\mu(x)\\
&= \int \L f \, d\zeta.
\end{align*}
Comparing the conditional measures of both sides (with the definition of $\wt{f}$ in mind), this is equivalent to 
\begin{equation}\label{eq: L^*}
(\L^*\zeta)_x =(\A(x)^*)_*\zeta_{\s x}.
\end{equation}
This may also be observed from the identity
$$(\L f_1)\cdot f_2 =\L(f_1 \cdot f_2 \circ F_{\A_*^{-1}})$$
where $\A_*^{-1}$ is the adjoint of the inverse cocycle for $\A$.
\begin{prop}\label{prop: eigendata}
There exists a unique probability measure $\nu$ on $\Sig^+ \times \P^{d-1}$ such that $\L^*\nu = \nu$. Moreover, $\nu$ is $F_{\A^{-1}_*}$-invariant and $\nu_x$ is equal to $\delta_{\xi_*(x)}$ for $\mu$-a.e. $x$.
\end{prop}

\begin{proof}
The existence of such $\nu$ is standard: it follows from the fact that $\L^*$ maps the space of probability measures to itself and that the set of probability measures is weak$^{\ast}$ compact.
%We first need to look into how $\L$ transforms $\{\nu_x\}_{x\in \Sig^+}$. For any probability measure $\zeta$ on $\Sig^+ \times \mathbb{P}^{d-1}$ and continuous function $f \in C(\Sig^+ \times \mathbb{P}^{d-1})$, consider a function $\wt{f}$ on $\Sig^+$ defined by $\displaystyle \wt{f}(x):=\int f(x,\ol{\A(x)^{-1}u})\,d\zeta_{\s x}(u).$
%Recalling that $\mu$ is the eigenfunction of $L_{\log g}$, we have
%\begin{align*}
%\int \wt{f} \,d\mu =\int \wt{f}\,d(L_{\log g}^*\mu)= \int  L_{\log g}\wt{f}\, d\mu &=\int \sum\limits_{\s y = x}g(y)\wt{f}(y)\,d\mu(x)\\
%&= \int\sum\limits_{\s y=x}g(y)\int f(y,\ol{\A(y)^{-1}u})\, d\zeta_x(u)d\mu(x)\\
%&= \int \L f \, d\zeta.
%\end{align*}
%Comparing the conditional measures of both sides (with the definition of $\wt{f}$ in mind), this is equivalent to 
%$$
%(\L^*\zeta)_x =\A(x)_*^{-1}\zeta_{\s x}.
%$$
From \eqref{eq: L^*} we have $\nu_{\s x} = (\A^{-1}_*(x))_*\nu_x$, which shows that $\nu$ is $F_{\A^{-1}_*}$-invariant.

For the remaining statements claimed in the proposition, it suffices to show that $\nu$ is $(\A_*^{-1},H^s)$-invariant.
This is because the $(\A^{-1}_*,H^s)$-invariant measure is unique when $\A$ (hence $\A^{-1}_*$) is 1-typical, and it is characterized by its conditional measures supported on the slowest Lyapunov subspace $\xi_*(x)$ with respect to $\A^{-1}_*$ and $\mu$.
%This will prove the uniqueness of $\nu$. 

Being an eigenfunction of $\L$, $\nu$ has to project to $\mu$ under $\pi\colon \Sig^{+}\times\P^{d-1} \to \Sig^+$, and \eqref{eq: retrieve hm from m} applied to $\A^{-1}_*$ gives
$$\hnu_{\hx}=\lim\limits_{n\to \infty}(\A_*^{-1})^n(P(\s^{-n}\hx)) \nu_{P(\s^{-n}\hx)} =\lim\limits_{n\to \infty} \nu_x = \nu_{x}.$$
%The last equality is due to $\nu_{\s x} =\A(x)_* \nu_x$ from \eqref{eq: L*}.
This implies that $\hnu_{\hx} = \hnu_{\hy}$ is for any $\hy \in \Wloc^s(\hx)$. Since the local stable holonomies for $\A^{-1}_*$ are identically equal to $I$ (due to the same reasoning why $H^s \equiv I$ for $\A$), it follows that $\hnu$ is $(\A^{-1}_*,H^s)$-invariant.
\end{proof}

\subsection{$\L$ defined with the inverse}
%\textcolor{red}{added in this subsection to show that the results still hold for $\L$ defined with the inverse. However, I would be okay not including this subsection as a whole. Pls let me know what you think.} \mpcomment{I think it is worth including this.}
In this subsection, we remark that all results obtained thus far equally holds if we had defined the operator with $\A(y)^{-1}$ in place of $\A(y)^{*}$:
$$\L_{\text{inv}} f(x,\ol{u}) := \sum\limits_{y \colon \s y = x}g(y) f(y,\ol{\A(y)^{-1}u}).$$
Theorem \ref{thm: main} holds for $\L_\inv$ with the statements of Proposition \ref{prop: growth for all u} and Lemma \ref{lem: growth rate of u} replaced by
$$\lim_{n \to \infty}\frac{1}{n}\int \log \|\A^{-n}(\hx)u\|\,d\hmu^s_x(\hx) =-\lambda_d$$ 
and for $u_n \to u$,
$$\lim\limits_{n \to \infty} \frac{\|\A^{-n}(\hx)u_n\|}{\|\A^{-n}(\hx)\|} = |\langle u,\omega_*(\hx)\rangle|.$$
Moreover, the proof would make use of the gap in the bottom exponents $\lambda_{d-1}>\lambda_d$ instead of the gap in the top exponents $\lambda_1 >\lambda_2$.

We have
$$(\L_{\text{inv}} f_1)\cdot f_2 =\L_{\text{inv}}(f_1 \cdot f_2 \circ F_\A),$$
and its eigenmeasure $\nu_{\inv}$ of $\L_{\text{inv}}$ obtained as in Proposition \ref{prop: eigendata} is unique and $F_\A$-invariant. 
%Indeed, $(\L_\inv^*\zeta)_x =\A(x)_*^{-1}\zeta_{\s x}.$
%and this translates to the $F_\A$-invariance of $\nu_{\inv}$. 
Moreover, $\nu_\inv$ is $(\A,H^s)$-invariant, so its conditional measure $\nu_{\inv,x}$ sits on the slowest Lyapunov subspace $\omega(x)$ with respect to $\A$ and $\mu$. This observation is compatible with the fact that $\L_\inv$ acts on the fiber $\P^{d-1}$ by the inverse cocycle $\A^{-1}$.

%Since the local stable holonomy $H^s$ of $\A$ is identically equal to $I$, it makes sense to denote such direction by $\omega(x)$ where $x = P(\hx)$. Indeed, $\omega(x)\in \P^{d-1}$ is the slowest Lyapunov subspace of $\A$ with respect to $\mu$ at $x\in \Sig^+$. Then the conditional measures of $\nu_\inv$ are dirac masses sitting on the one-dimensional slowest Lyapunov subspaces $\omega(x)$. In fact, \eqref{eq: retrieve hm from m} applies to give
%$$\hnu_{\hx}=\lim\limits_{n\to \infty} \A^n(P(\s^{-n}\hx)) \nu_{P(\s^{-n}\hx)} = \nu_{x}.$$
%This implies that $\hnu_{\hx} = \hnu_{\hy}$ is for any $\hy \in \Wloc^s(\hx)$ and it follows that $\hnu$ is $(\A,H^s)$-invariant since $H^s\equiv I$. Under 1-typicality assumption $\A$, there exists a unique $(\A,H^s)$-invariant measure characterized by its conditional measures given by $\delta_{\omega(\hx)}$.

%$
%(\L_\inv^*\zeta)_x =\A(x)_*^{-1}\zeta_{\s x}.
%$
%In particular, this translates to the $F_\A$-invariance of $\nu_{\inv}$. 
%$\nu_{\inv,\s x} = \A(x)_*\nu_{\inv,x}$.

\section{The Peripheral Spectrum of $\L$}\label{sec: spectral results}
In this section, we will prove the following theorem assuming the results and terminologies from the Appendix (Section \ref{sec: appendix}). 

\begin{thm}\label{thm:spectralstructure}
Suppose $\S_{T}^{+}$ is a shift of finite type defined by an irreducible $T$, and $\hat{\A}$ is $1$-typical. Then for all $\alpha>0$ sufficiently small, $\L$ is quasi-compact and the following are true: there exists $h \in \N$ such that
	\begin{enumerate}
		\item 
		There exists a unique probability measure $\nu$ with $\L^{\ast}\nu = \nu$.
%		 \mpcomment{Adding the word unique is fine.}
		
		\item 
		The only eigenvalues of $\L$ of modulus $1$ are $ \set{e^{\frac{2 \pi i}{h}k}}_{k=0}^{h-1}$ and each is a simple.
		
		\item 
		$\L$ can be written as
		\[ \L=\sum_{k=0}^{h-1}\lambda_{k}P_{k}+S \]
		where $\displaystyle \lambda_{k}=\exp\Big(\frac{2 k \pi i}{h}\Big)$. 
		
		Moreover, $P_{k}^{2}=P_{k}$, $P_{i}P_{j}=0$ for $i\neq j$, $P_{0}f =\inn{f,\nu}1$, and $S$ is a bounded linear operator with $\rho(S)<1$ and $P_{i}S=SP_{i}=0$.
	\end{enumerate}
\end{thm}

As $T$ is irreducible, so is $T^{\ast}$. Hence, up to reordering the alphabet, we may assume that $T^{\ast}$ can written as a block matrix
\[ T^{\ast}= \bmat{0& T_{12}^{\ast} & 0 & \cdots & 0 \\ \vdots &0 & T_{23}^{\ast}& \cdots &0\\ \vdots &  &   & \ddots & \vdots \\  0 & 0& \cdots & & T_{h-1h}^{\ast}\\ T_{h1}^{\ast} & 0& \cdots & & 0} \]
where the diagonal blocks are square. This naturally partitions the alphabet in classes $\set{(p)}_{p=1}^{h}$ based on the implied decomposition of $\R^{n}$. For convenience we will write
\[ [(p)] = \bigcup_{i \in (p)}[i]. \]

For any $f \in C^{\alpha}(\S_{T}^{+}\times \P^{d-1})$, we define the multiplication operator $M_{f}$ by
\begin{equation}\label{eq: multiplication operator}
 M_{f}h= fh .
\end{equation}
It is clear that $M_f$ is a bounded linear operator on the space $C^{\alpha}(\S_{T}^{+}\times \P^{d-1})$ for any $0<\alpha\leq 1$.
%\kp{brought this here as it appears in the proof of the following theorem.}

\begin{proposition}\label{prop: L decomposed}~ 
	\begin{enumerate}
		\item 
		The space $C^{\alpha}(\S_{T}^{+}\times \P^{d-1})$ is isomorphic to the space $\bigoplus\limits_{i}C^{\alpha}([(i)]\times \P^{d-1})$.
		
		\item 
		When thought of as acting on $\bigoplus\limits_{i}C^{\alpha}([(i)]\times \P^{d-1})$, $\L$ can be written as $\L=[L_{ij}]$ where 
		\[ L_{ij}f(x,\ol{u}) = \sum_{k \in (j)\colon kx\in \S_{T}^{+}} g(kx) f(kx, \ol{\A(kx)^*u})\chi_{[(i)]\times \P^{d-1}}(x,\ol{u}). \]
		Notice that $L_{ij}=0$ unless there exist $a\in (i)$ and $b\in (j)$ such that $T_{ab}^{\ast}=T_{ba}=1$; this is equivalent to $j \equiv i+1 \mod h$.
		\item 
		For each $1\leq p\leq h$, we have
		\[ \L \chi_{[(p)]\times \P^{d-1}} = \begin{cases}
		\chi_{[(h)]\times \P^{d-1}}& p=1\\
		\chi_{[(p-1)]\times \P^{d-1}}& p\geq 2
		\end{cases}. \] 
%				\kp{looks good!}
%		\kp{changed typo from $k$ to $p$; also removed $(x,\ol{u})$}
%\textcolor{red}{A quick question/comment: isn't $L_{ij} = 0$ unless $j \equiv i+1~ (\mod h)$? I guess I am slightly confused about the notation ''$T_{ab}=1$ for some $a \in (i)$ and $b\in (j)$'' because from above matrix form $T$, the index seems to range between $1$ and $h$ but $a,b$ ranges from all possible alphabets of $\Sig^+$ which there are more than $h$ in number (or maybe there is doubling up on the notation $T_{ab}$)? So the question is can't we just use ``$j \equiv i+1~ (\mod h)$'' instead of ``$a \in (i)$ and $b\in (j)$''? Or am I missing something in the notation? Actually another question question related to this; if $j \equiv i+1~ (\mod h)$, then doesn't there necessarily exist $a\in (i)$ and $b\in (j)$ such that $T_ab = 1$? that is, can't we just use $j \equiv i+1~ (\mod h)$?}		
%	 \mpcomment{Yes this is correct, I have made the change}
	\end{enumerate}
\end{proposition}
\begin{proof}~
	\begin{enumerate}
		\item 
		Define the map
		\[ \id:C^{\alpha}(\S_{T}^{+}\times \P^{d-1}) \to \bigoplus_{i=1}^{h}C^{\alpha}([(i)]\times \P^{d-1})\]
		by
		\[\id(f)= (f\cdot \chi_{[(1)]\times \P^{d-1}}, f\cdot\chi_{[(2)]\times \P^{d-1}}, \ldots, f\cdot \chi_{[(h)]\times \P^{d-1}}).  \]
		One can check that $\id$ is linear and that $\norm{\id(f)}\leq \norm{f}_{\alpha}$. Moreover the function
		\[ (f_{1},f_{2}, \cdots , f_{h})\mapsto \sum_{i=1}^{h}f_{i} \]
		defines an inverse for $\id$.
		
		\item 
		Notice that
		\begin{align*}
		L_{ij}f(x,\ol{u}) &= M_{\chi_{[(i)]\times \P^{d-1}}}\L M_{\chi_{[(j)]\times \P^{d-1}}}f(x,\ol{u})\\
		&=\chi_{[(i)]\times \P^{d-1}}(x,\ol{u})\sum_{k:kx\in \S_{T}^{+}}g(kx)M_{[(j)]\times \P^{d-1}}f(kx,\ol{\A(kx)^{\ast}u})\\
		&=\chi_{[(i)]\times \P^{d-1}}(x,\ol{u})\sum_{k:kx\in \S_{T}^{+}}g(kx)\chi_{[(j)]\times \P^{d-1}}(kx,\ol{\A(kx)^{\ast}u})f(kx,\ol{\A(kx)^{\ast}u})\\
		&=\sum_{k:k\in (j),kx\in \S_{T}^{+}}g(kx)f(kx,\ol{\A(kx)^{\ast}u})\chi_{[(i)]\times \P^{d-1}}(x,\ol{u}).
		\end{align*}
	
	\item 
	Notice that
	\begin{align*}
	\L \chi_{[(p)] \times \P^{d-1}}(x,\ol{u}) &= \sum_{k:kx \in \S_{T}^{+}}g(kx)\chi_{[(p)]\times \P^{d-1}}(kx,\ol{\A(kx)^*u})\\
	&=\sum_{k:k\in (p),kx \in \S_{T}^{+}}g(kx)\\
	&=\begin{cases}
	\chi_{[(h)]\times \P^{d-1}}(x,\ol{u})& p=1\\
	\chi_{[(p-1)]\times \P^{d-1}}(x,\ol{u})& p\geq 2
	\end{cases}.
	\end{align*} 
	\end{enumerate} 
This completes the proof.
\end{proof}

\begin{lemma}\label{lem:positivity}
Suppose that $T$ is primitive with $T^m>0$ for some $m\in \N$. Given $f \geq 0$, suppose further that there exists a point $(z,\ol{w})$ such that $\L^nf(z,\ol{w})=0$ for some $n\in \N$.
Then for any $0 \leq k \leq n$ and $x \in \S_{T}^{+}$, there exists $(y,\ol{v})$ such that $\L^kf(y,\ol{v}) = 0$ and $\rho(x,y) \leq 2^{-(n-k-m)}$. 
%	 for all $y$ such that $\sigma^{n-k}y=x$ there exists a $\ol{v}$ such that $\L^{k}f(y,\ol{v})=0$.
\end{lemma}
\begin{proof}
%\kp{added slight correction for these two lemmas}
	Notice that
	\[ 0=\L^{n}f(z,\ol{w})= \sum_{\sigma^{n-k}y=z}g^{(n-k)}(y)\L^{k}f(y,\ol{\A^{[n-k]}(y)w}). \]
	Thus $\L^{k}f(y,\ol{\A^{[n-k]}(y) w})=0$ for all $y$ such that $\sigma^{n-k}y=z$. Among all such $y$, there necessarily exists a $y$ such that $\rho(x,y) \leq 2^{-(n-k-m)}$.
\end{proof}

\begin{lemma}\label{lem:blockspectralgap}
For each $1 \leq p\leq h$, the restriction $\L_{(p)}:=\L^h|_{C^\alpha([p]\times \P^{d-1})}$ has a spectral gap. 
\end{lemma}
\begin{proof}
Since $\L_{(p)}$ is quasi-compact from Proposition \ref{prop: hennion},
in view of Theorem \ref{abstractPFThm-Primtive} it suffices to show that $\L_{(p)}$ is semi-positive with respect to the cone of non-negative functions. 

Notice that if $\L_{(p)}^{N}f>0$ for some $N\in \N$, then $\L_{(p)}^{n}f>0$ for all $n \geq N$. 
Therefore, in view of Definition \ref{defn: sufficient} and \ref{defn: semi-positive}, it suffices to show the following statement: if $f\geq 0$ and there exists $\{(z_{n},\ol{w}_{n})\}_{n\in \N}\subset [(p)]\times \P^{d-1}$ with $\L_{(p)}^{n}f(z_{n},\ol{w}_{n})=0$ for all $n\in \N$, then $\L_{(p)}^{n}f$ converges to $0$.

%	Suppose that $f \in C^{\alpha}([(p)]\times \P^{d-1})$, $f\geq 0$ is such that for all $n$ there exists $(z_{n},\ol{w}_{n})\in [(p)]\times \P^{d-1}$ with $(L_{p}^{h})^{n}f(z_{n},\ol{w}_{n})=0$. Then by the proof of Proposition --- we have that
%	\[ \abs{\L^{n}f(x,\ol{u})-\L^{n}f(y,\ol{w})}\leq \norm{f}_{\alpha} \beta^{n}+K_{0}\norm{f}_{\infty}d(x,y)^{\alpha} \]
%	for some $0<\beta<1$ for any $x,y \in \S_{T}^{+}$.

Let $x$ be any given point in $[(p)]$.		
By Theorem \ref{thm:irreduciblePF} there exists $M\in \N$ such that for any $i,j \in (p)$, there exists a word $\I$ of length $Mh$ such that the word $i\I j$ is admissible.		
%Since $\L_{(p)}^{2n+M}f(z_{2nh+Mh},\ol{w}_{2nh+Mh})=0$, with $(2n+M)h,$ 	
Thus for any $n\in \N$, 
we can apply Lemma \ref{lem:positivity} to $\L_{(p)}^{2n+M}f(z_{2n+M},\ol{w}_{2n+M})=0$ with $M$, $2n+M$, and $n$ playing the roles of $m$, $n$, and $k$ respectively. This gives $(y,\ol{v})$ such that $\L^{n}_{(p)}(y,\ol{v}) = 0$ and $\rho(x,y) \leq 2^{-nh}$.

%there exists $z$ such that $\sigma^{k}z=x$ and $d(y,z)\leq 2^{-k+M}$. As $\L^{2nh+Mh}f(z_{2nh+Mh},\ol{w}_{2nh+Mh})=0$ by Lemma \ref{lem:positivity} it follows that there exists $(y,\ol{w})$ such that $\L^{nh}f(y,\ol{w})=0$ and $d(x,y)\leq 2^{-nh}$. 
Noticing that $\L_{(p)}^{n}f$ agrees with $ \L^{nh}f$ for $f\in C([(p)]\times \P^{d-1})$, Corollary \ref{cor: LY} then gives for any $\ol{u} \in \P^{d-1}$
	\begin{align*}
	\abs{\L_{(p)}^{n}f(x,\ol{u})}&=\abs{\L^{nh}f(x,\ol{u})-\L^{nh}f(y,\ol{v})}\\
	&\leq \norm{f}_{\alpha} \beta^{nh}+C\norm{f}_{\infty}\rho(x,y)^{\alpha}\\
	& \leq \norm{f}_{\alpha} \beta^{nh}+C\norm{f}_{\infty}2^{-\a nh}.
	\end{align*}
In particular, this shows that $\L_{(p)}^{n}f(x,\ol{u})$ converges to $0$ for any $(x,\ol{u})$.

It is clear that point-mass measures $\{\d_{(x,\ol{u})}\}$ form a sufficient collection for the cone of non-negative functions, and hence we conclude that $\L_{(p)}$ is semi-positive.
\end{proof}

\begin{lemma}\label{lem:algmultiplicityof1}
	The algebraic multiplicity of the eigenvalue $1$ for $\L$ is one.
\end{lemma}
\begin{proof}
%\kp{slightly puzzled about the proof (second sentence); could you please expand a little on the proof?}
%\mpcomment{I replaced the first sentence with your suggestion. This proof was not quite correct in the last draft I have fixed it now. Let me know if it makes more sense now.}
%\kp{\textcolor{red}{looks good to me}}
Given $f = (f_1,\ldots,f_{h}) \in \ker(I-\L$), we have $\L_{(p)} f_p = f_p$ for each $1 \leq p \leq h$. From Lemma \ref{lem:blockspectralgap}, this implies that $f_p$ is a constant multiple of the function $\chi_{[(p)]\times \P^{d-1}}$. Hence $f \in \spn\set{\chi_{[(p)]\times \P^{d-1}}\colon 1 \leq p \leq h}$. 

Now consider the action of $\L$ on the subspace $\spn\set{\chi_{[(p)]\times \P^{d-1}}\colon 1 \leq p \leq h}$. Notice from Proposition \ref{prop: L decomposed} that 
\[ \L\chi_{[(p)]\times\P^{d-1}} = \chi_{[(p-1)]\times\P^{d-1}} \text{ for }p\geq 2 \text{ and }\L\chi_{[(1)]\times\P^{d-1}} =\chi_{[(h)]\times\P^{d-1}} .\]
%\textcolor{red}{just realized that there is slight inconsistency for whether $p$ ranges between 0 and $h-1$ or between 1 and $h$; not a big deal. I can fix this.} 
%\mpcomment{I fixed Prop 4.1. Does it look correct now?}
%\textcolor{red}{please correct me if I am wrong, but isn't $\L\chi_{[(p)]\times\P^{d-1}} = \chi_{[(p+1)]\times\P^{d-1}}$; that is, with $p+1$ instead of $p-1$? Because if we think of $x = x_0x_1\ldots$ with $x_0 \in( i),x_1\in (i+1)\ldots$ for some $i$, then any $y$ with $\s y=x$ has $y_0 \in (i-1)$. Then 
%$$\L \chi_{ [(p)] \times \P^{d-1}}(x,\ol{u}) = \sum_y g(y) \chi_{ [(p)] \times \P^{d-1}}(y,\ol{\A(y)^*u})$$
%and this is equal to 1 if and only if $y \in [(p)]$. Since $y \in [(i-1)]$, this happens only when $i-1=p$; that is when $i = p+1$. So shouldn't $\L \chi_{ [(p)] \times \P^{d-1}}$ be equal to $ \chi_{ [(p+1)] \times \P^{d-1}}$? However, it is quite likely that my hasty computation has obvious holes.}
 When written as a matrix with basis $\set{\chi_{[(p)]\times \P^{d-1}} \colon  1 \leq p \leq h}$, the operator $\L$ is exactly 
 \[ [\L]_{\set{\chi_{[(p)]\times \P^{d-1}}\colon 1 \leq p \leq h}}=\bmat{0& 1 & 0 & \cdots & 0 \\ \vdots &0 & 1& \cdots &0\\ \vdots &  &   & \ddots & \vdots \\  0 & 0& \cdots & & 1\\ 1 & 0& \cdots & & 0}. \]
This matrix is irreducible and its only eigenvector corresponding to the eigenvalue $1$ is the constant function $1$. Therefore $f =\alpha 1$ for some $\alpha \in \C$.
	
	Next we show that the order of the eigenvalue $1$ is one. We need to show that 
	\[ \ker (I-\L)^{2}=\ker(I-\L) .\]
The inclusion $\ker (I-\L)\subseteq \ker (I-\L)^{2}$ is clear, so we must show the other inclusion. If $f \in \ker (I-\L)^{2}$, then $f - \L f \in \ker (I-\L)$ and thus $f - \L f = \alpha 1$ for some $\alpha \in \C.$ Then
	\[ 0=\inn{f-\L f, \nu} = \alpha \inn{1,\nu} = \alpha, \]
	where $\nu$ is the eigenmeasure of $\L$ from Proposition \ref{prop: eigendata}.
	Therefore $f-\L f =0$ and $f \in \ker(I-\L)$.
\end{proof}

\begin{proof}[Proof of Theorem \ref{thm:spectralstructure}] The quasi-compactness of $\L$ is already shown in Proposition \ref{prop: hennion}.
	\begin{enumerate}
	\item This follows from Proposition \ref{prop: eigendata}.
		\item
		It is clear that $1$ is an eigenvalue of $\L$.  From Proposition \ref{prop:rotationsymmetry} the spectrum of $\L$ is invariant under multiplication by $e^{\frac{2 \pi i}{h}}$. Thus for each $0\leq k\leq h-1$, $e^{\frac{2 \pi i}{h}k}$ is an eigenvalue of $\L$. As the only eigenvalue of modulus $1$ for $\L^{h}$ is $1$, this implies that the only eigenvalues of modulus $1$ are $\set{e^{\frac{2 \pi i}{h}k}}_{k=0}^{h-1}$. By Lemma \ref{lem:algmultiplicityof1} the eigenvalue $1$ is simple, and thus by Proposition \ref{prop:rotationsymmetry} the eigenvalue $e^{\frac{2 \pi i}{h}k}$ is simple for each $k$.
		
		\item
		Using spectral projections, $\L$ can be written as
		\[ \L=\sum_{k=0}^{h-1}e^{\frac{2 \pi i}{h}k}P_{k}+S \]
		where $P_{k}$ is the spectral projection onto the eigenspace associated to $e^{\frac{2 \pi i}{h}k}$ and $\rho(S)<1$. 
		
		Notice that as the eigenvalues $\set{e^{\frac{2 \pi i}{h}k}}_{k=0}^{h-1}$ are simple, $P_{k}$ is a projection onto the one-dimensional subspace $\ker(e^{\frac{2 \pi i}{h}k}I-\L)$ and it can be written as
%		\mpcomment{removed the $z$ from $\L_{z}$ here}
		\[ P_{k}=\lim_{\xi \to e^{\frac{2 \pi i}{h}k}}(\xi - e^{\frac{2 \pi i}{h}k})R(\xi, \L) \]
		where $R(\xi, \L) = (\xi -\L)^{-1}$ is the resolvent of $\L$. As $\L$ commutes with $R(\xi, \L)$ we have that $P_{k}\L=\L P_{k} = e^{\frac{2 \pi i}{h}k}P_{k}$. Notice
		\[ e^{\frac{2 \pi i}{h}k}P_{k}P_{k'}= \L P_{k}P_{k'} = P_{k}\L P_{k'}=e^{\frac{2 \pi i}{h}k'}P_{k}P_{k'} \]
		for $k \neq k'$, so it must be that $P_{k}P_{k'}=0$.
		
		Finally notice that 
		\[ SP_{j} = \left(\L - \sum_{k=0}^{h-1}e^{\frac{2 \pi i}{h}k}P_{k}\right)P_{j} = \L P_{j}-e^{\frac{2 \pi i}{h}j}P_{j} =0 \]
		and similarly $P_{j}S=0$. 
	\end{enumerate}
\end{proof}

\begin{corollary}
		Suppose $T$ is primitive, $\S_{T}^{+}$ is the shift of finite type defined by $T$, 
%		$g:\S_{T}^{+}\to \R$ is a H\"older continuous $g$-function,
		 and ${\h\A}$ is $1$-typical. Then for all $\a>0$ sufficiently small the operator $\L:C^{\alpha}(\S_{T}^{+}\times \P^{d-1})\to C^{\alpha}(\S_{T}^{+}\times \P^{d-1})$ has a spectral gap.
\end{corollary}
\begin{proof}
	If $T$ is primitive, then $h=1$ and the result follows from Theorem \ref{thm:spectralstructure}.
\end{proof}

\section{Applications of Theorem \ref{thm:spectralstructure}}\label{sec: limit theorems}
%\textcolor{red}{I think we have to state the main theorems in the introduction for $\h\A$ and $\hmu$ over two-sided shift $\Sig$. However, this section was written for $\A^{[n]}(x)$ and $\mu$ over one-sided shift $\Sig^+$. Instead of reworking the entire section, I thought it'd be best to add the following paragraphs explaining why it suffices to prove the result for $\A^{[n]}(x)$ with respect to $\mu$. Could you make sure that these look okay? Especially, the sentence in blue?}
In this section, we will prove Theorem \ref{thm:CLT} and \ref{thm:LDP}.
%However, we will prove these theorems for the adjoint cocycle $\A_*$ instead because 1-typicality of $\A$ is equivalent to 1-typicality of $\A_*$. As always, we assume that the subshift $\Sig$ is defined by an irreducible adjacency matrix $T$.
However, instead of directly proving them for $\A$, we will prove them for the adjoint cocycle $\A_*$. Recall that $\A$ is 1-typical if and only if $\A_*$ is.

First, we make a few simplifications as done in Section \ref{sec: prelim}; that is, we assume that $g \colon \Sig^+ \to \R$ is a \hol continuous $g$-function, and let $\mu$ be its equilibrium state obtained as the eigenmeasure of $L_{\log g}$. We denote the $\s$-invariant measure on $\Sig$ that projects to $\mu$ by $\hmu$. We will use $\lambda_1$ to denote the common top Lyapunov exponent:
$$\lambda_1 :=\lambda_1(\A,\mu)=\lambda_1(\A,\hmu) = \lambda_1(\A_*,\hmu).$$
%As in the previous section, we will continue to use the operator $\L$ defined by the adjoint cocycle $\A_*$ as in \eqref{eq: L with adjoint}.

Recalling the notation $\A^{[n]}(x)=[\A^n(x)]^*$ from Section \ref{sec: quasi-compactness}, we have
$$\int \log \norm{\A^{[n]}(x)u} d\mu = \int \log \norm{\A(\hx)^*\ldots \A(\s^{n-1}\hx)^*u}d\hmu = \int \log \norm{\A_*^n(\hx)u}d\hmu,$$
where the first equality is because $\A$ is constant along the local stable sets and the second equality is due to $\s$-invariance of $\hmu$.
%Given a 1-typical cocycle $\h\A \colon \Sig \to \glr$, we conjugate it to a cocycle $\A$ that takes constant values along the local stable sets.
Hence, for the adjoint cocycle $\A_*$, $\var(\h\A,\hmu)$ defined in \eqref{eq: sigma} may alternatively be described by 
%$$\displaystyle \lim_{n \to \infty}\frac{1}{n}\int \left(\log \norm{\hat{\A}^{n}(\hx)u}-n\lambda_{1}\right)^{2}d\hmu$$ may alternatively be described by
$$
\var=\lim_{n \to \infty}\frac{1}{n}\int \left(\log \norm{\A^{[n]}(x)u} -n\lambda_{1}\right)^{2}d\mu.
$$

%\mpcomment{I just realized we are doubling up on $\sigma$ as both the shift and the variance in the CLT we should probably choose a different letter for the variance}

In the same vein, because the distribution of $\Big(\log \norm{\A_*^n(\hx)u} -n \lambda_1\Big)$ with respect to $\hmu$ agrees with that of $\Big(\log \norm{\A^{[n]}(x)u}-n\lambda_{1}\Big)$ with respect to $\mu$, it suffices to establish the claimed results for the later distribution; see Theorem \ref{thm: clt 2} and \ref{thm:LDPgfunction}.

%\textcolor{red}{I think that justification is fine.}

\subsection{Central limit theorem}
In this section we will prove the central limit theorem (Theorem \ref{thm:CLT}). The proof uses the spectral properties of the operator $\L$ as well as the standard functional analytic proof for central limit theorems.

\begin{theorem}\label{thm: clt 2}
	Suppose $\h\A:\S_{T}\to \glr$ is $1$-typical. If $\var>0$,
%	\[ \sigma^{2}=\lim_{n \to \infty}\frac{1}{n}\int \left(\log \norm{\A^{[n]}(x)u} -n\lambda_{1}\right)^{2}d\mu>0\]
	then 
	\[ \frac{\log \norm{\A^{[n]}(x)u} - n\lambda_{1}}{\sqrt{n}}\xrightarrow[n \to \infty]{\text{dist}}\n(0,\var). \]
	If $\var=0$, then
	\[ \frac{\log \norm{\A^{[n]}(x)u} - n\lambda_{1}}{\sqrt{n}}\xrightarrow[n \to \infty]{\text{dist}}0. \]
\end{theorem}

%We begin by defining a few notations.
%	For $n \geq 1$ and $x \in \Sig^+$, we set
%	\[ \A^{[n]}(x):= \A(x)^*\A(\sigma x)^*\cdots \A(\sigma^{n-1}x)^* = [\A^n(x)]^* \]
%	and
%	\[ \A^{(n)}(x)=\A(\sigma^{n-1}x)\cdots \A(\sigma x)\A(x). \]
We begin by defining a function $\psi_{n}:\S_{T}^{+}\times \P^{d-1}\to \R$ by
\begin{equation}\label{eq: psi_n} 
\psi_{n}(x,\ol{u}):=\log \norm{\A^{[n]}(x)\frac{u}{\norm{u}}}. 
\end{equation}
Moreover, for $n \geq 1$ and $z \in \C$, we set
	\[ S_{n,z}(x,\ol{u}):=\norm{\A^{[n]}(x)\frac{u}{\norm{u}}}^{z} .\]
%Notice that
%\[ \A^{[n]}(x)^{\ast} =\A_{\ast}^{(n)}(x).  \]
For $z \in \C$ we define the operator
\begin{equation}\label{eq:perturbedoperators}
\L_{z}f(x,\ol{u}):=\sum_{\sigma y =x}g(y)\norm{\A(y)^*\frac{u}{\norm{u}}}^{z}f(y,\ol{\A(y)^*u}). 
\end{equation}
Notice that when $z=0$, $\L_0$ agrees with the operator $\L$ considered in the previous sections.

In addition recall the definition of the transfer operator $L_{\log g} \colon C(\S_{T}^{+})\to C(\S_{T}^{+})$
%\[ L_{\log g}:C(\S_{T}^{+})\to C(\S_{T}^{+}) \]
defined by
\[ L_{\log g}f(x)=\sum_{\sigma y =x}g(y)f(y). \]
%That statement that $g$ is a $g$-function is then equivalent to saying $L_{\log g}1=1$.
Observe that $L_{\log g}$ could also be considered as acting on $C(\S_{T}^{+}\times \P^{d-1})$ by
\[ L_{\log g}f(x,\ol{u})=\sum_{\sigma y =x}g(y)f(y,\ol{u}). \]

\begin{lemma}\label{lemma:Ops}~
	\begin{enumerate}
		\item 
		For any $n \geq 1$.
		\[ \L_{z}^{n}f(x,\ol{u}) =\sum_{\sigma^{n}y=x}g^{(n)}(y)\norm{\A^{[n]}(y)\frac{u}{\norm{u}}}^{z}f(y, \ol{\A^{[n]}(y)u}) . \]
		
		\item 
		For any $n \geq 1$ and $u\in \R^{d}$ we have that
		\[ \L_{z}^{n}1(x,\ol{u})=L_{\log g}^{n}S_{n,z}(x,\ol{u}) \]
		\item For any unit vector $u \in \R^d$,
\begin{equation}
\label{eq:Characteristicfunction}
 \int \norm{\A^{[n]}(x)u}^{z}d\mu(x) =\int \L_{z}^{n}1(x,\ol{u})\,d\mu(x)
\end{equation}
	\end{enumerate}
\end{lemma}
\begin{proof}
	\begin{enumerate}
		\item 
		Notice that
		\begin{align*}
		\L^{2}_{z}f(x,\ol{u})&=\sum_{\sigma y_{1} =x}g(y)\norm{\A(y_{1})^*\frac{u}{\norm{u}}}^{z}\L_{z}f(y,\ol{\A(y_{1})^*u})\\
		&=\sum_{\sigma y_{1} =x}g(y_{1})\norm{\A(y_{1})^*\frac{u}{\norm{u}}}^{z}\sum_{\sigma y_{2} = y_{1}}g(y_{2})\norm{\A(y_{2})^*\frac{\A(y_{1})^*u}{\norm{\A(y_{1})^*u}}}^{z}f(y_{2},\ol{\A(y_{2})\A(y_{1})u})\\
		&=\sum_{\sigma^{2} y =x}g(y)g(\sigma y)\norm{\A(y)^*\A(\sigma y)^*\frac{u}{\norm{u}}}^{z}f(y , \ol{\A(y)^*\A(\sigma y)^*u})\\
		&= \sum_{\sigma^{2}y=x}g^{(2)}(y)\norm{\A^{[2]}(y)\frac{u}{\norm{u}}}^{z}f(y, \ol{\A^{[2]}(y)u}).
		\end{align*}
		The general case for $n \geq 1$ is similar.
		
		\item 
		Notice that 
$$
		L_{\log g}^{n}S_{n,z}(x,\ol{u}) = \sum_{\sigma^{n}y=x}g^{(n)}(y)S_{n,z}(y,\ol{u})
	=\sum_{\sigma^{n}y=x}g^{(n)}(y)\norm{\A^{[n]}(y)\frac{u}{\norm{u}}}^{z}=\L_{z}^{n}1(x,\ol{u}).$$
	\item Since $(L_{\log g})^* \mu = \mu$, from the definition of $S_{n, z}(x,\ol{u})$ we have
$$ \int \norm{\A^{[n]}(x)u}^{z}d\mu(x)=\int S_{n, z}(x,\ol{u}) \, d(L_{\log g}^{\ast})^{n}\mu(x) = \int L_{\log g}^{n}S_{n, z}(x,\ol{u}) \,d\mu(x).
$$
Now the claim follows from (2).
\end{enumerate}
\end{proof}

%Recall the following theorem will be the main tool for 

		For each $n\in\N$, define
		\[ C_{\A^{[n]}}f (x,\ol{u}):= f (x,\ol{\A^{[n]}(x)u}). \]
For the following lemma, recall also the definition for $M_f$ and $\psi_n$ from \eqref{eq: multiplication operator} and \eqref{eq: psi_n}. 
\begin{lemma}\label{lem: property L_z}~
% \textcolor{red}{coherent notation}
	\begin{enumerate}
		\item 
%		\kp{changed notation coherent for the adjoint cocycle}
	For any $n\in \N$,	$C_{\A^{[n]}}$ is a bounded linear operator on $C^{\alpha}(\S_{T}^{+}\times \P^{d-1})$.
		
		\item For any $z\in \C$,
		\[ \L_{z}^{n}=\sum_{k=0}^{\infty}\frac{z^{k}}{k!}L_{\log g}^{n} M_{\psi_{n}}^{k} C_{\A^{[n]}} \]
		where we interpret $M^{0}_{\psi_{n}}=I$. In particular, $z \mapsto \L_{z}$ is analytic. 
		
%		\item \textcolor{red}{maybe remove this because it doesn't seem to be used elsewhere?} \mpcomment{I guess it is used for n=1 in proposition 5.7 but it is clear we probably don't need to say it explicitly}
%		\[ (\L^{n})'_{0} = L_{\log g}^{n} M_{\psi_{n}}C_{\A^{[n]}} \text{ and }(\L^{n})_{0}'' =L_{\log g}^{n} M_{\psi_{n}}^{2}C_{\A^{[n]}}  \]
	\end{enumerate}
\end{lemma}
\begin{proof}~
	\begin{enumerate}
		\item 
%		\mpcomment{Some mixing of $\A/\B$ notation here. That is from my draft we should just decide if it should be a general $\B$ or have it by $\A^{\ast}$}
		We demonstrate the idea for $n=1$. The proof easily extends for other $n\in\N$.
		
		It is clear that $C_{\A^{[1]}}$ is linear. Notice that for $f \in C^{\alpha}(\S_{T}^{+}\times \P^{d-1})$ and $(x,\ol{u}), (y,\ol{w})\in \S_{T}^{+}\times \P^{d-1}$ we have
		\begin{align*}
		\abs{C_{\A^{[1]}}f(x,\ol{u})-C_{\A^{[1]}}f(y,\ol{w})}&=\abs{f(x,\ol{\A(x)^*u}) - f(y,\ol{\A(y)^*w})}\\
		&\leq |f|_{\a}\cdot d((x,\ol{\A(x)^*u}), (y,\ol{\A(y)^*w}))^\a\\
		& = |f|_{\a}\cdot \max\set{\rho(x,y), d(\ol{\A(x)^*u},\ol{\A(y)^*w})}^\a.
		\end{align*}
		We have that
		\begin{align*}
		d(\ol{\A(x)^*u},\ol{\A(y)^*w})&=\frac{\norm{\A(x)^*u \wedge \A(y)^*w}}{\norm{\A(x)^*u} \norm{\A(y)^*w}}\\
		&\leq \frac{\norm{\A(x)^*\wedge\A(y)^*}}{\norm{\A(x)^*u} \norm{\A(y)^*w}}\cdot \frac{\norm{u \wedge w}}{\norm{u}\norm{w}}\\
		&\leq \left(\sup_{x,y}\frac{\norm{\A(x)^*\wedge\A(y)^*}}{\sigma_{d}(\A(x)^*) \sigma_{d}(\A(y)^*)}\right)\cdot d(\ol{u},\ol{w}).
		\end{align*}
		Therefore
		\[ \abs{C_{\A^{[1]}}f}_{\alpha} \leq \abs{f}_{\alpha}\left(\sup_{x,y}\frac{\norm{\A(x)^*\wedge\A(y)^*}}{\sigma_{d}(\A(x)^*) \sigma_{d}(\A(y)^*)}\right)^\a. \]
		As $\norm{C_{\A^{[1]}}f}_{\infty}\leq \norm{f}_{\infty}$, we have that $C_{\A^{[1]}}$ is bounded.
		
%		\item
%		This follows from the proof that the product of two H\"older functions is H\"older.
		
		\item 
		Notice that 
		\[ \norm{L_{\log g}^{n} M_{\psi_{n}}^{k} C_{\A^{[n]}}}_{\op} \leq \norm{M_{\psi_{n}}}_{\op}^{k}\norm{L_{\log g}}_{\op}^{n}\norm{C_{\A^{[n]}}}_{\op} .\]
		Thus for any $z \in \C$
		\begin{align*}
		\sum_{k=0}^{\infty}\norm{\frac{z^{k}}{k!}L_{\log g}^{n} M_{\psi_{n}}^{k} C_{\A^{[n]}}}&\leq \norm{L_{\log g}}_{\op}^{n}\norm{C_{\A^{[n]}}}_{\op}\sum_{k=0}^{\infty}\frac{\left(\abs{z}\norm{M_{\psi_{n}}}_{\op}\right)^{k}}{k!}\\
		&=\norm{L_{\log g}}_{\op}^{n}\norm{C_{\A^{[n]}}}_{\op} e^{\abs{z}\norm{M_{\psi_n}}_{\op}}.
		\end{align*}
		Thus by Proposition \ref{prop:abelbanach}, $\displaystyle \sum_{k=0}^{\infty}\frac{z^{k}}{k!}L_{\log g}^{n} M_{\psi_{n}}^{k} C_{\A^{[n]}}$ converges for all $z\in \C$ and defines an analytic function from $\C$ to $\B(C^{\alpha}(\S_{T}^{+}\times \P^{d-1}))$. 
		
Then for $f \in C^{\alpha}(\S_{T}^{+}\times \P^{d-1})$, we have 
		\begin{align*}
		\left(\sum_{k=0}^{\infty}\frac{z^{k}}{k!}L_{\log g}^{n} M_{\psi_{n}}^{k} C_{\A^{[n]}}\right)f(x,\ol{u})
		&=\sum_{k=0}^{\infty}\frac{z^{k}}{k!}\sum_{\sigma^{n} y =x}g^{(n)}(y)\left(\log \norm{\A^{[n]}(y)\frac{u}{\norm{u}}}\right)^{k}f(y,\ol{\A^{[n]}(y)u})\\
		&=\sum_{\sigma^{n} y =x}g^{(n)}(y)\sum_{k=0}^{\infty}\frac{z^{k}}{k!}\left(\log \norm{\A^{[n]}(y)\frac{u}{\norm{u}}}\right)^{k}f(y,\ol{\A^{[n]}(y)u})\\
		&=\sum_{\sigma^{n} y =x}g^{(n)}(y)\exp\Big(z\log \norm{\A^{[n]}(y)\frac{u}{\norm{u}}}\Big)f(y,\ol{\A^{[n]}(y)u})\\
		&=\sum_{\sigma^{n} y =x}g^{(n)}(y)\norm{\A^{[n]}(y)\frac{u}{{\norm{u}}}}^{z}f(y,\ol{\A^{[n]}(y)u})\\
		&=\L_{z}^{n}f(x,\ol{u}).
		\end{align*}
		
%		\item 
%		This follows from part (3).
	\end{enumerate}
\end{proof}
%
%\begin{rem}
%	\textcolor{red}{Derivative notation....}
%\end{rem}

\begin{proposition}\label{prop:perturbation}
	Suppose that $\h\A$ is $1$-typical, and let $\L_{z}$ be the family of operators defined in \eqref{eq:perturbedoperators}. There exists an open set $U \subseteq \C$ containing $0$ such that for any $z \in U$ we can write
	\[ \L_{z}=\rho_{z}\left(\sum_{k=0}^{h-1}e^{\frac{2 \pi i}{h}k}P_{k,z}+S_{z} \right) \]
	where $P_{i,z}P_{j , z}=P_{j,z}P_{i,z} = 0$ for $i \neq j$, $\rho(S_{z})<1$, and $P_{k, z}S_{z}=S_{z}P_{k, z}=0$ for all $0\leq k\leq h-1$. The functions $z \mapsto \rho_{z},P_{k,z}, S_{z}$ are all analytic. Moreover, for any $p\geq 1$ there exist constants $C$ and $0<\beta<1$ such that
	\[ \norm{\frac{d^{p}}{dz^{p}}S_{z}^{n}}\leq C \cdot \beta^{n}. \]
\end{proposition}
\begin{proof}
	As the function $z \mapsto \L_{z}$ is analytic we have by analytic perturbation theory for linear operators (see \cite[Chapter 7 Section 3]{MR1335452}) that there exists some neighborhood $U$ containing $0$ on which for all $z \in U$, the only eigenvalues of $\L_{z}$ outside a disk of radius $0<r<\rho(L_{z})$ are $\set{\lambda_{k,z}}_{k=0}^{h-1}$. 
	
	Moreover the function $z \mapsto \lambda_{k,z}$ is analytic, $\lambda_{k,0}=e^{\frac{2 \pi i}{h}k}$, and $\lambda_{k,z}$ is simple for all $k$. It can be seen that when $\L_{z}:\bigoplus_{k}C^{\alpha}([(k)] \times \P^{d-1}) \to \bigoplus_{k}C^{\alpha}([(k)] \times \P^{d-1})$ is written in its matrix form that
	\[ \L_{z}= \bmat{0&L_{12,z} & 0 & \cdots & 0 \\ \vdots &0 & L_{23,z}& \cdots &0\\ \vdots &  &   & \ddots & \vdots \\  0 & 0& \cdots & & L_{h-1h,z}\\ L_{h1,z} & 0& \cdots & & 0} .\]
Thus by Proposition \ref{prop:rotationsymmetry} the spectrum of $\L_{z}$ is invariant under multiplication by $e^{\frac{2 \pi i}{h}k}$ for $0 \leq k \leq h-1$. Thus $\L_{z}$ has $h$ eigenvalues of modulus $\rho(\L_{z})$; in particular it must be that $\abs{\lambda_{k,z}}=\rho(\L_{z})$ and $\lambda_{k,z}:=e^{\frac{2 \pi i}{h}k}\lambda_{0,z}$.
	The remainder of the proof is similar to that of Theorem \ref{thm:spectralstructure} (3). 
%	\colorbox{red}{    }We can write
%	\[ \L_{z}=\rho_{z}\left(\sum_{k=0}^{h-1}e^{\frac{2 \pi i}{h}k}P_{k,z}+S_{z} \right) \]
%	where $P_{i,z}P_{j , z}=P_{j,z}P_{i,z}$ for $i \neq j$, $\rho(S_{z})<1$, and $P_{k, z}S_{z}=S_{z}P_{k, z}$ for all $k$. Here $\rho_{z}=\lambda_{0,z}$ and for each $k$, $P_{k,z}$ is the spectral projection associated to the isolated eigenvalue $\lambda_{k,z}$. Thus $z \mapsto \rho_{z}$ and $z\mapsto P_{k,z}$ are analytic.
%	
%Since the eigenvalues $\set{\lambda_{k,z}}_{k=0}^{h-1}$ are simple, $P_{k,z}$ is a projection onto the 1-dimensional subspace $\ker(\lambda_{k,z}I-\L_{z})$ and it can be written
%	\[ P_{k,z}=\lim_{\xi \to \lambda_{k,z}}(\xi - \lambda_{k,z})R(\xi, \L_{z}) \]
%	where $R(\xi, \L_{z}) = (\xi -\L_{z})^{-1}$ is the resolvent of $\L_{z}$. As $\L_{z}$ commutes with $R(\xi, \L_{z})$ we have that $P_{k,z}\L_{z}=\L_{z}P_{k,z} = \lambda_{k,z}P_{k,z}$. Notice
%	\[ \lambda_{i, z}P_{i,z}P_{j, z}= \L_{z}P_{i,z}P_{j,z} = P_{i,z}\L_{z}P_{j,z}=\lambda_{j,z}P_{i,z}P_{j,z} \]
%	so it must be that $P_{i,z}P_{j,z}=0$ for $i \neq j$.
%	
%	Notice that 
%	\[ S_{z}P_{j,z} = \left(\L - \sum_{k=0}^{h-1}\lambda_{k,z}P_{k,z}\right)P_{j,z} = \L P_{j,z}-\lambda_{j,z}P_{j,z} =0 \]
%	and similarly $P_{j,z}S_{z}=0$.\colorbox{red}{    }  \mpcomment{We could probably cut all of that}
	
	Finally to see that for any $p\geq 1$ there exist constants $C$ and $0<\beta<1$ such that
	\[ \norm{\frac{d^{p}}{dz^{p}}S_{z}^{n}}\leq C \beta^{n}, \]
we refer the reader to \cite[Chapter V Lemma 3.2(iv)]{MR886674}; while the assumptions in the lemma are different then ours, but the proof applies in our case.
\end{proof}

%For the following proposition, recall that $\mu$ is the unique equilibrium state for $\log g$ obtained as the eigenmeasure of $L_{\log g}$.

\begin{proposition} \label{prop: deriv of rho}
Denoting by $\rho'_z$ and $\rho''_z$ the derivative and the second derivative of $z \mapsto \rho_z$,
for any unit $u \in \R^{d}$ we have
	\begin{enumerate}
		\item  
$ \displaystyle		\rho_{0}' = \lim_{n \to \infty}\frac{1}{n}\int \log \norm{\A^{[n]}(x)u}d\mu=\lambda_{1}. $
		
		\item 
$\displaystyle \rho_{0}'' = \lim_{n \to \infty}\frac{1}{n}\int \left(\log \norm{\A^{[n]}(x)u}-n\lambda_{1}\right)^{2}d\mu = \var.$
	\end{enumerate}
\end{proposition}
\begin{proof}
	\begin{enumerate}
		\item 
%		For each $n\in \N$, we have $\displaystyle \int \log \Big\| \A^{[n]}(x)u\Big\| \,d\mu =  \int \log \Big\| \A_*^{n}(\hx)u\Big\| \,d\hmu $ from the $\s$-invariance of $\hmu$. Since $\displaystyle \lim_{n\to \infty} \frac{1}{n}\int \log \Big\| \A_*^{n}(\hx)u\Big\| \,d\hmu=\lambda_1$ for any unit $u \in \R^d$ as shown in Proposition ---, this establishes the second equality of (1).
		
%		$\displaystyle \int \norm{\A^{[n]}(x)u}^{z} d\mu(x)  = \int \norm{\A^{[n]}(\hx)u}^{z} d\hmu(\hx)$
		
%This proves the second equality of (1).

For the first equality, 
%\begin{align*}
%\int \norm{\A^{[n]}(x)u}^{z} d\mu(x)  = \int \norm{\A^{[n]}(\hx)u}^{z} d\hmu(\hx) &= \int  \|\A(\hx)^*\A(\sigma \hx)^*\cdots \A(\sigma^{n-1}\hx)^* u \|^z \,d\hmu(\hx)\\
%&=\int \|\A_*^n(\hx)u\|^z \,d\hmu(\hx).
%\end{align*}
%The last equality is due to the $\s$-invariance of $\hmu$. As noted in Remark ---, Proposition --- also applies to $\L$ in --- form. Hence, $\displaystyle \int \|\A_*^n(\hx)u\|^z \,d\hmu(\hx)$ is equal to the top Lyapunov exponent $\lambda_1(\A_*,\hmu)$ with respect to $\A_*$ and $\hmu$, which is then equal to $\lambda_1(\A,\hmu)=\lambda_1(\A,\mu)$.
recall from \eqref{eq:Characteristicfunction} that 
		\begin{align*}
		\int \norm{\A^{[n]}(x)u}^{z} \, d\mu(x) &= \int \L_{z}^{n}1(x,\ol{u})\,d\mu(x) \\
		&= \rho_{z}^{n}\int \sum_{k=0}^{h-1}e^{\frac{2 \pi i}{h}nk}P_{k,z}1(x,\ol{u})+S_{z}^{n}1(x,\ol{u})\,d\mu(x)
		\end{align*}
for any $z\in \C$.		
Taking the derivative of both sides give 
%\kp{typo? should be $S$ instead of $R$?} \mpcomment{Yes it is a typo}
\begin{align}\label{eq:firstderivative}	
		\int \norm{\A^{[n]}(x)u}^{z}\log \norm{\A^{[n]}(x)u}d\mu &= n\rho_{z}^{n-1}\rho_{z}'\int \sum_{k=0}^{h-1}e^{\frac{2 \pi i}{h}nk}P_{k,z}1(x,\ol{u})+S_{z}^{n}1(x,\ol{u})d\mu(x)\\
		& \;\;\;\;\;\;\; +\rho_{z}^{n}\int \sum_{k=0}^{h-1}e^{\frac{2 \pi i}{h}nk}P_{k,z}'1(x,\ol{u})+(S_{z}^{n})'1(x,\ol{u})d\mu(x)\nonumber
		\end{align}
Noticing that $\rho_0 = 1$ and $S_01 =(S_0P_0)1 = 0$, evaluating \eqref{eq:firstderivative} at $z=0$ gives
$$		\rho_{0}' = \frac{1}{n}\int \log \norm{\A^{[n]}(x)u}\,d\mu-\frac{1}{n}\int \sum_{k=0}^{h-1}e^{\frac{2 \pi i}{h}nk}P_{k,0}'1(x,\ol{u})+(S_{0}^{n})'1(x,\ol{u})\,d\mu(x).
$$
%Since $\displaystyle \sum_{k=0}^{h-1}e^{\frac{2 \pi i}{h}k}P_{k,0}'$ does not depend on $n$ and $\|(S^n_z)'\|$ decreases exponentially fast to 0, 
Now taking the limit as $n \to \infty$ gives the result.
		
		\item 
By dividing $\A(x)$ by $e^{\lambda_1}$ if necessary, we may assume that $\lambda_1=\rho_{0}'=0$. Differentiating \eqref{eq:firstderivative} again gives
		\begin{align*}
		\rho_{0}'' = \frac{1}{n}\int \left(\log \norm{\A^{[n]}(x)u}\right)^{2}d\mu-\frac{1}{n}\int \sum_{k=0}^{h-1}e^{\frac{2 \pi i}{h}nk}P_{k,0}''1(x,\ol{u})+(S_{0}^{n})''1(x,\ol{u})\,d\mu(x).
		\end{align*}
By the same reasoning as in the proof of (1), the second term tends to 0 as $n \to \infty$. 
%Hence, taking the limit as $n\to \infty$ gives the result.
	\end{enumerate}
\end{proof}

The last ingredient we need in order to prove Theorem \ref{thm: clt 2} is the following Levy continuity theorem:

\begin{theorem}[Levy continuity]\label{thm:LevyContinuity}
	A sequence of random variables $(X_{n})_{n=1}^{\infty}$ converges in distribution to a random variable $Y$ if and only if $\E(e^{itX_{n}})\xrightarrow{n \to \infty}\E(e^{itY})$ for all $t \in \R$.
\end{theorem}

\begin{proof}[Proof of Theorem \ref{thm: clt 2}]
As in the proof of Proposition \ref{prop: deriv of rho} (2), we may assume that $\lambda_1=0$.
	By Proposition \ref{prop:perturbation}, we can write
	\[ \L_{z}^{n}=\rho_{z}^{n}\left(\sum_{k=0}^{h-1}e^{\frac{2 \pi i}{h}nk}P_{k,z}+S_{z}^{n} \right) \]
for $\abs{z}$ sufficiently small.
	By \eqref{eq:Characteristicfunction} we have that for $n$ sufficiently large
	\begin{align*}
	\int \norm{\A^{[n]}(x)u}^{\frac{it}{\sqrt{n}}}d\mu(x)&= \int \L_{\frac{it}{\sqrt{n}}}^{n}1(x,\ol{u})d\mu(x) \\
	&= \int \rho_{\frac{it}{\sqrt{n}}}^{n}\left(\sum_{k=0}^{h-1}e^{\frac{2 \pi i}{h}nk}P_{k,\frac{it}{\sqrt{n}}}+S_{\frac{it}{\sqrt{n}}}^{n} \right)1(x,\ol{u})d\mu(x)  \\
	&= \rho_{\frac{it}{\sqrt{n}}}^{n}\left(\sum_{k=0}^{h-1}\int e^{\frac{2 \pi i}{h}nk} P_{k, \frac{it}{\sqrt{n}}}1(x,\ol{u})d\mu(x) + \int S^{n}_{\frac{it}{\sqrt{n}}}1(x,\ol{u})d\mu(x)\right)
	\end{align*}
	Notice that $z \mapsto P_{k, z}$ is continuous. Thus
	\[ \lim_{n \to \infty}P_{0,\frac{it}{\sqrt{n}}}1=P_{0,0}1= 1, \] 
	and for $k\neq 1$
	\[ \lim_{n \to \infty}P_{k ,\frac{it}{\sqrt{n}}}1 = P_{k,0}1=0 \]
	where the limit is taken in $C^{\a}(\S_{T}^{+}\times \P^{d-1})$. Furthermore
	\[ \norm{S_{\frac{it}{\sqrt{n}}}^{n}1(x,\ol{u})}_{C^{\a}(\S_{T}^{+}\times \P^{d-1})} = O(\beta^{n}) \]
	for some $0<\beta<1$. Note that this $\b$ can be chosen uniformly for all $S_{z}$ with $\abs{z}$ sufficiently small. Thus for $n$ sufficiently large
	\begin{equation}\label{eq:charfuncspectralradius}
	\int \norm{\A^{[n]}(x)u}^{\frac{it}{\sqrt{n}}}d\mu(x) = \rho_{\frac{it}{\sqrt{n}}}^{n}(1+o(1)). 
	\end{equation}
	Using the Taylor expansion of $\rho_{z}$ at $z=0$ we have
	\[ \rho_{\frac{it}{\sqrt{n}}}^{n} = \left( 1-\frac{\rho_{0}''}{2}\left(\frac{t}{\sqrt{n}}\right)^{2}+O\left(\frac{t^{3}}{n^{3/2}}\right)\right)^{n}; \]
here we have used the assumption that $\rho_0'=\lambda_1=0$.
Since $\rho_0'' = \var$, we have
	\begin{align*}
	\int \norm{\A^{[n]}(x)u}^{\frac{it}{\sqrt{n}}}d\mu(x) &= \left( 1-\frac{\rho_{0}''}{2}\left(\frac{t}{\sqrt{n}}\right)^{2}+O\left(\frac{t^{3}}{n^{3/2}}\right)\right)^{n}(1+o(1))\\
	&\xrightarrow{n \to \infty}\exp\left[-\frac{1}{2}\var  \cdot t^{2}\right].
	\end{align*}
	If $\var>0$, then $\displaystyle \exp\left[-\frac{1}{2}\var \cdot t^{2}\right]$ is the characteristic function for $\n(0, \var)$. Hence by Theorem \ref{thm:LevyContinuity} (Levy continuity),
	\[ \frac{\log \norm{\A^{[n]}(x)u}}{\sqrt{n}}\xrightarrow[n \to \infty]{\text{ dist}} \n(0,  \var). \]
	If $\var=0$, then $\displaystyle\exp\left[-\frac{1}{2}\var \cdot t^{2}\right]=1$ is the characteristic function of $0$. So Levy continuity implies that
	\[ \frac{\log \norm{\A^{[n]}(x)u}}{\sqrt{n}}\xrightarrow[n \to \infty]{\text{dist}}0. \]
This completes the proof.
\end{proof}

\subsection{Furstenberg's formula}
For i.i.d. products of matrices, Furstenberg \cite{furstenberg1963noncommuting} gave a formula for the top Lyapunov exponent assuming strong irreducibility:
\[ \lambda_{1}=\iint \log \frac{\norm{Au}}{\norm{u}} \,d\mu(A)d\nu(u) \]
where $\mu$ is the distribution of $A$ and $\nu$ is any stationary measure on the projective space. If, in addition, $A$ is proximal, then the stationary measure $\nu$ is unique.
%\textcolor{red}{explain the setting of this theorem; add reference} 

As a byproduct of the previous subsection, we prove the following analogue of Furstenberg's formula for the top Lyapunov exponent $\lambda_1$.
%Recall that $\nu$ is the eigenmeasure of $\L$.
\begin{proposition}[Furstenberg's Formula]
\label{prop:FurstenbergsFormula}
%We have 
		\[ \lambda_1= \rho_{0}'=\int \sum_{\sigma y =x}g(y) \log \norm{\A(y)^*\frac{u}{\norm{u}}}\,d\nu(x,u) \]
		where $\nu$ is the eigenmeasure of $\L$ from Proposition  \ref{prop: eigendata}. 
\end{proposition}
\begin{proof}
%The first equality $\lambda_1 = \rho'$ is from Proposition ---.
Write $\displaystyle \L_{z}=\rho_{z}\left(\sum_{k=0}^{h-1}e^{\frac{2 \pi i}{h}k}P_{k,z}+S_{z}\right)$ for $\abs{z}$ small as in proposition \ref{prop:perturbation}. Then differentiating the identity $\L_{z}P_{0,z}=\rho_{z}P_{0,z}$ gives
$$\L_{z}'P_{0,z}+\L_{z}P_{0,z}'=\rho_{z}'P_{0,z}+\rho_{z}P_{0,z}.$$
Since $P_{0,z}\L_{z}=\rho_{z}P_{0,z}$ and $P_{0,z}^2 = P_{0,z}$, applying $P_{0,z}$ to both sides gives
$$P_{0,z}\L_{z}'P_{0,z}+\rho_{z}P_{0,z}P_{0,z}'=\rho_{z}'P_{0,z}+\rho_{z}P_{0,z}P_{0,z}'.$$
Canceling $\rho_{z}P_{0,z}P_{0,z}'$ from both sides gives $P_{0,z}\L_{z}'P_{0,z}=\rho_{z}'P_{0,z}$. Evaluating this identity at $z=0$ and applying the operators to the constant function 1 gives
$$\int \L_{0}'1 d\nu= P_{0,0}\L_{0}'P_{0,0}1=\rho_{0}'P_{0,0}1=\rho_{0}'.$$
From Lemma \ref{lem: property L_z} (2), we have $(\L)'_{0} = L_{\log g} M_{\psi_{1}}C_{\A^{[1]}}$. Thus
\begin{align*}
\lambda_1=\rho_{0}' = \int \L_{0}'1 \,d\nu &= \int L_{\log g}M_{\psi_1}C_{\A^{[1]}}1 \,d\nu\\
&=\int L_{\log g}\psi_1 \, d\nu \\
&=  \int \sum_{\sigma y =x}g(y)\log \norm{\A(y)^*\frac{u}{\norm{u}}}\,d\nu(x,u),
\end{align*}
completing the proof.
\end{proof}

\subsection{Large deviation principle}
In this subsection, we prove Theorem \ref{thm:LDP} and Corollary \ref{cor: LDP}.

	For small enough $\e \in \R$, define
	\[ \Lambda(\e) := \log \rho_{\e}-\e  \lambda_1\]
%\textcolor{red}{should it be $$\Lambda(\e) := \log \rho_{\e}-\e  \lambda_1?$$ because above $\Lambda(\ep)$ doesn't satisfy $\Lambda'(0)=0$; I will fix the proof below assuming $\Lambda(\e) := \log \rho_{\e}-\e  \lambda_1$.} \mpcomment{Yes, $\log \rho_{\e}-\e\lambda_{1}$ is correct }
%	\kp{a little confused about the definition of $\Lambda(\e)$ because in the proof of this theorem, the assumption $\lambda_1=0$ will force $\Lambda(\e)$ to be $-\infty$. But I'm sure this is not meant to happen?}
	where $\rho_{\e}$ is the spectral radius of the operator
	\[ \L_{\e}f(x,\ol{u}):=\sum_{\sigma y =x}g(y)\norm{\A(y)^{\ast}\frac{u}{\norm{u}}}^{\e}f(y,\ol{\A(y)^{\ast}u}). \]

%Similar to Theorem ---, it suffices to prove the following Large Deviation Principle concerning $\A^{[n]}(x)u$ with respect to $\mu$
	
\begin{theorem}[Large Deviation Principle]\label{thm:LDPgfunction}
	Suppose ${\h\A}:\S_{T}\to \glr$ is $1$-typical and $\var$ is positive.
%	, $g:\S_{T}^{+}\to \R$ is a H\"older continuous $g$-function, $\mu$ its $g$-measure and that 
%	\[ \lim_{n \to \infty}\frac{1}{n}\int \left(\log \norm{\A^{(n)}(x)u}-n\lambda_{1}(\A, \mu)\right)^{2}d\mu>0.\]
	Then there exists $\eta>0$ such that for any $\e\in \big(0,\frac{\Lambda(\eta)}{\eta}\big)$ and any unit vector $u\in \R^d$,
	\[ \lim_{n \to \infty}\frac{1}{n}\log \mu\set{x:\Big| \log \norm{\A^{[n]}(x)u} - n\lambda_1\Big|>n\e}=-\Lambda^{\ast}(\e)<0. \]
Here $\Lambda^{\ast}$ is the Legendre transform of $\Lambda$ on the interval $[0,\eta]$.
\end{theorem}

\begin{lemma}[Local G\"artner-Ellis Theorem]\label{lem:GartnerEllis}
	Suppose that $X_{n}$ is a sequence of random variables such that there exists $\eta>0$ for which
	\[ \lim_{n \to \infty}\frac{1}{n}\log \E(e^{tX_{n}})=c(t) \]
	exists in an open set of $[-\eta,\eta]$. If $c(t)$ is continuously differentiable and strictly convex on $[-\eta,\eta]$ and $c'(0)=0$, then for any $0<\e < \frac{c(\eta)}{\eta}$
	\[ \lim_{n \to \infty}\frac{1}{n}\log \P(X_{n}>n\e)=-c^{\ast}(\e)<0. \]
	where $c^{\ast}$ is the Legendre transform of $c$ on the interval $[0,\eta]$.
\end{lemma}

The proof of Lemma \ref{lem:GartnerEllis} can be found in \cite[Chapter V Lemma 6.2]{MR886674}.

\begin{proof}[Proof of Theorem \ref{thm:LDPgfunction}]
Without the loss of generality, we may assume that $\lambda_1=0$.
% similar to the proof of --- . 
Then $\Lambda(\ep) = \log \rho_\e$. We will verify the assumptions of Lemma \ref{lem:GartnerEllis} for $X_n =\log \norm{\A^{[n]}(x)u}$.
% and $c(\ep)=\Lambda(\e)$. 

First, $\Lambda'(0) = \rho'_0/\rho_0  = 0$; here we have used the assumption that $\rho_0' = \lambda_1 = 0$.
%Under this assumption, we have $\Lambda(\ep) = \log \rho_\ep$.  
By reasoning as in \eqref{eq:charfuncspectralradius}, for sufficiently small $\e$ we have
	\[ \lim_{n \to \infty}\frac{1}{n}\log\int \norm{\A^{[n]}(x)u}^{\e}d\mu=\lim_{n\to \infty} \frac{1}{n}\log \Big(\rho_\e^n \cdot (1+o(1))\Big)= \Lambda(\e). \]
Since $\var = \rho_0''$ from Proposition \ref{prop: deriv of rho} and $\var$ is positive from the assumption, it follows that 
%	\[ \lim_{n \to \infty}\frac{1}{n}\int \left(\log \norm{\A^{(n)}(x)u}-n\lambda_{1}(\A, \mu)\right)^{2}d\mu>0\]

	\[ \frac{d^{2}\Lambda}{d\e^{2}}\Big|_{\e = 0}=\frac{ \rho_{\e}''\rho_{\e} - (\rho_{\e}')^2}{\rho_{\e}^2}\Big|_{\ep=0} = \rho_0''>0. \]
Thus there is some neighborhood of $0$ on which $\Lambda(\e)$ is strictly convex. 

The analogous argument also applies to $X_n =-\log \norm{\A^{[n]}(x)u}$, and the result then follows by applying Lemma \ref{lem:GartnerEllis}.
\end{proof}
\begin{proof}[Proof of Corollary \ref{cor: LDP}] As in Theorem \ref{thm: clt 2} and \ref{thm:LDPgfunction}, we will prove the result for $\A^{[n]}(x)$ with respect to $\mu$.
Again without loss of generality, we may assume that $\lambda_1 = 0$.

We begin by choosing a basis $\U=\{u_1,\ldots,u_d\}$ of $\R^d$ consisting of unit vectors and defining a new norm 
$$\|A\|_\U:=\max\limits_{1 \leq i \leq d} \|Au_i\|.$$
By relaxing the constant $C>0$ in the statement of the corollary if necessary, we may prove the result with respect to the new norm $\|\cdot \|_\U$. 

Then for any $\ep>0$,
$$\set{x:\log \norm{\A^{[n]}(x)}_\U>n\e} = \bigcup_{1\leq i \leq d} \set{x:\log \norm{\A^{[n]}(x)u_{i}}_\U>n\e},$$
and hence, 
$$\mu\set{x:\log \norm{\A^{[n]}(x)}_\U>n\e}  \leq \sum_{1\leq i \leq d} \mu \set{x:\log \norm{\A^{[n]}(x)u_{i}}_\U>n\e}.$$
Since each term in the summation on the right hand side decreases exponentially from Theorem \ref{thm:LDPgfunction}, so does the sum. Hence, the term on the left also has to decrease exponentially fast, as required.
\end{proof}

As mentioned in the introduction, the large deviation principle stated in Corollary \ref{cor: LDP} is already known. For instance, Gou\"ezel and Stoyanov \cite{gouezel2019} obtain this result which implies exponential returns to Pesin sets. Such property is then used to deduce exponential mixing of Gibbs measures in certain settings. We remark, however, that the large deviation principle stated in Theorem \ref{thm:LDP} is new, and we provide a simpler proof via using the spectral properties of the operator $\L$.

We also note that Duarte, Klein, and Poletti \cite{duartekleinpoletti} prove the uniform version of Corollary \ref{cor: LDP}. They then use it to show \hol continuity of the Lyapunov exponents as a function of the cocycle.

\subsection{Analyticity of the Lyapunov exponent $\lambda_1$} 
%\mpcomment{I made a few changes to this section and I need to add a little more}
By analyticity of a function mapping some domain of $\R$ into a Banach space $\B$ we mean that the function admits an analytic extension to a domain in $\C$ which contains the original domain of $\R$. 

\begin{theorem}\label{thm:Analyticity}
	Suppose $\h\A \colon \S_{T} \to \glr$ is $1$-typical and $t \mapsto g_{t}$ be a function defined on some interval $(-\e,\e)$ such that $g_{t}$ is a H\"older continuous $g$-function for each $t$. If the function $t\mapsto g_{t}$ is analytic,
%	 (in the sense that there is some analytic function defined on a complex neighborhood of $0$ such that the restriction to $(-\e,\e)$ is $g_{t}$),
then $t \mapsto \lambda_{1}(\h\A ,\mu_{g_{t}})$ is real analytic in a neighborhood of $0$.
\end{theorem}
\begin{proof} 
	Let $z \mapsto g_{z}$ be the analytic extension of $t \mapsto g_{t}$. Define
	\[\L_{g_{z}}f(x,\ol{u}):=\sum_{\sigma y =x}g_{z}(y)f(y,\ol{\A(y)^{\ast}u}). \]
	By Theorem \ref{thm:spectralstructure}, for each $t \in (-\e,\e)$ the number $1$ is a simple eigenvalue for the operator $\L_{g_{t}}$. We denote by $\nu_{t}$ the eigenmeasure of $\L_{g_{t}}$ corresponding to $\rho(\L_{g_t})$. By analytic perturbation theory (see \cite{MR1335452} or \cite{MR3920693}), there exists an open neighborhood $U\subseteq \C$ of $(-\e,\e)$ and an analytic function $z \mapsto \nu_{z}$ defined on $U$ which extends $\nu_{t}$; that is, $\nu_{z}$ is a eigenmeasure of $\L_{g_{z}}$ for $z\in U$. Moreover the function $z \mapsto L_{\log g_{z}}$ is analytic and thus $z \mapsto L_{\log g_{z}}^{\ast}\nu_{z}$ is also analytic. Hence the function
	\[ \lambda(z):= \inn{L_{\log g_{z}} \psi_{1},\nu_{z}} \]
	is an analytic function on $U$, and by Proposition \ref{prop:FurstenbergsFormula} we have $\lambda(t) = \lambda_1(\h\A , \mu_{g_{t}})$ for all $t \in (-\e,\e)$. Therefore $t \mapsto \lambda_{1}(\h\A ,\mu_{g_{t}})$ is real analytic in a neighborhood of $0$.
\end{proof}

This result is new even in the case when $\A$ and $g$ depend only on two coordinates. Given an irreducible stochastic matrix $S$, we denote by $\mu_{S}$ the unique Markov measure determined by $S$. In this special case, the result says the following:

\begin{corollary}
	Suppose that $T$ is irreducible and $\set{A_{ij}}_{i,j:T_{ij}=1} \subseteq \glr$. If the function $\A:\S_{T}^{+}\to \glr$ defined by $\A(x)=A_{x_{0}x_{1}}$ is $1$-typical, then there is a relatively open neighborhood $U$ in the hyperplane 
	\[ \set{(z_{ij})\in \C^{d^{2}} : \sum_{i=1}^{d}z_{ij}=1 \text{ and }z_{ij}\neq 0 \text{ if and only if }T_{ij}=1} \]  
	such that
	\[ \set{S\in M_{d}(\R): S \text{ is stochastic and }S_{ij}>0 \text{ if and only if }T_{ij}=1} \subseteq U \] 
	and $S \mapsto \lambda_{1}(\A, \mu_{S})$ extends to an analytic function on $U$.
\end{corollary}

One may compare this result to \cite[Theorem 2]{MR1158741} in which the matrices are assumed to be positive. Finally we use Theorem \ref{thm:Analyticity} to prove Theorem \ref{thm:Analyticitygeneral}.

\begin{proof}[Proof of Theorem \ref{thm:Analyticitygeneral}]
	By assumption, $t \mapsto \hat{\psi}_{t}$ is analytic. Essentially we must check that the reduction to a one-sided potential and then to a $g$-function preserves this analyticity. 
	
	As in \cite[Proposition 1.2]{MR1085356}, there is a bounded linear map $W:C^{\alpha}(\S_{T})\to C^{\alpha/2}(\S_{T}^{+})$ such that $\hat{\psi}$ is cohomologous to $W\hat{\psi}$. Writing $\psi_{t}=W\hat{\psi}_{t}$, we have that $t \mapsto \psi_{t}$ is analytic and $\lambda_{1}(\hat{\A},\mu_{\hat{\psi_{t}}})=\lambda_{1}(\A,\mu_{\psi_{t}})$. Let $h_{t}$ be the eigenfunction for the transfer operator $L_{\psi_{t}}$ corresponding to $\rho(L_{\psi_{t}})$. It is known (see \cite{MR1085356} or \cite{MR1793194}) that $h_{t}$ and $\rho(L_{\psi_{t}})$ are analytic functions. Thus if we define
	\[ g_{t}:=\frac{e^{\psi_{t}(x)}}{\lambda} \cdot \frac{h_{t}(x)}{h_{t}(\s x)} \]
	as in \eqref{eq: g},
	then $t \mapsto g_{t}$ is analytic, and $g_{t}$ is a $g$-function where $\lambda_{1}(\hat{\A},\mu_{\hat{\psi}_{t}})=\lambda_{1}(\h\A, \mu_{g_{t}})$ for each $t \in \R$. The result now follows from Theorem \ref{thm:Analyticity}.
\end{proof}

\section{Appendix}\label{sec: appendix}

\subsection{Complex analysis in Banach Algebras}
%\mpcomment{Modified this section}

Here we collect some standard results about analytic functions which take values in Banach algebras. For proofs of most of these results and general background we refer the reader to \cite[III.14]{MR1009162}. Let $X$ be a Banach algebra.

\begin{definition}
	Suppose that $U \subseteq \C$ is open and $f:U \to X$ is a function. We say that $f$ is \emph{differentiable} at $z_{0}$ if there exists an element $f'(z_{0})\in X$ such that %\kp{typo on $A$?} \mpcomment{Yes}
	\[ \norm{\frac{f(z)-f(z_{0})}{z-z_{0}}-f'(z_{0})}\xrightarrow{z \to z_{0}}0. \]
We say $f$ is \emph{analytic} on $U$ if $f$ is continuously differentiable on $U$.
\end{definition}

\begin{theorem}\label{thm:wealanalyticimpliesanalytic}
	Suppose $U \subseteq \C$ and $f:U \to X$. Then $f$ is analytic if and only if $\varphi(f(z)):U \to \C$ is analytic for all $\varphi \in X^{\ast}$.
\end{theorem}

The following corollary can be deduced from the proof of Theorem \ref{thm:wealanalyticimpliesanalytic}.

\begin{corollary}\label{cor:analyticfunctionboundedlinearops}
	Let $f:U \to L(X_{1},X_{2})$. Then $f(z)$ is analytic if and only if $\inn{f(z)x_{1},x_{2}^{\ast}}$ is analytic for all $x \in X_{1}$ and $x_{2}^{\ast}\in X_{2}^{\ast}$.
\end{corollary}

\begin{lemma}
	If $f(z), g(z):U \to X$ are analytic, then
	\begin{enumerate}
		\item 
		$f+g$ is analytic and $(f+g)'(z) = f'(z)+g'(z)$.
		
		\item 
		$fg$ is analytic and $(fg)'(z)= f'(z)g(z)+f(z)g'(z)$.
	\end{enumerate}
\end{lemma}

\begin{lemma}
	Let $X$ be a Banach space, $z \mapsto L_{z}$ a function from $U\subseteq \C$ to the bounded linear operators on $X$ and $z \mapsto x_{z}$ a function from $U$ to $X$.
	\begin{enumerate}
		\item 
		Suppose that $z \mapsto L_{z}$ is analytic, then $z \mapsto L_{z}^{\ast}$ is analytic.
		
		\item 
		Suppose that $z \mapsto x_{z}$ is analytic and $z \mapsto L_{z}$ is analytic, then 
		$z \mapsto L_{z}x_{z}$ is analytic.
	\end{enumerate}
\end{lemma}
\begin{proof}
%	\begin{enumerate}
(1) follows immediately from Corollary \ref{cor:analyticfunctionboundedlinearops}
and (2) can be deduced in the same way as the usual product rule.
\end{proof}

\begin{proposition}\label{prop:abelbanach}
	Let $X$ be a Banach space and $x_{n}$ be a sequence of vectors in $X$. If $\displaystyle \sum_{n=1}^{\infty}\norm{x_{n}}r^{n}$ converges for some $r>0$, then the function $\displaystyle f(z)=\sum_{n=1}^{\infty}x_{n}z^{n}$ is analytic on the set $\abs{z}<r$.
\end{proposition}

\subsection{The Ruelle-Perron-Frobenius theorem for semi-positive operators}
%\mpcomment{A few small changes in this section}

Here we present the abstract Ruelle-Perron-Frobenius theorem that we make use of in this paper. We begin by recalling some definitions and propositions; for more background and proofs of these facts, we refer the reader to \cite{MR1009162}.

Let $X$ be a complex Banach space and $\L(E)$ the space of bounded linear operators on $X$. If $L \in \L(E)$, we define
\[ \spec(L)=\set{\lambda \in \C : \lambda - L \text{ is not invertible}} \]
and
\[ \res(L)=\C \setminus \spec(L). \]

\begin{theorem}
	The function $\lambda \mapsto (\lambda - L)^{-1}$ is analytic on $\res(L)$.
\end{theorem}

For $\lambda \in \res(L)$, we define
\[ R(\lambda, L)=(\lambda - L)^{-1}. \]

\begin{definition}
	The \emph{geometric multiplicity} of an eigenvalue $\lambda$ is
	\[ \geom (\lambda):=\dim \ker (\lambda - L) \]
	and the \emph{algebraic multiplicity}
	\[ \alg (\lambda):=\dim \bigcup_{k=1}^{\infty} \ker(\lambda - L)^{k}. \]
	The \emph{order} of $\lambda$, provided it exists, is
	\[ \min \set{k:\ker (\lambda - L)^{k+1}=\ker (\lambda - L)^{k}}. \]
\end{definition}

\begin{proposition}\label{prop:index-order}
	If $\lambda_{0}$ is an isolated point of $\spec(L)$ and a pole of $R(\lambda)$, then $\lambda_{0}$ is a eigenvalue for $L$. The order of $\lambda_{0}$ as an eigenvalue is the order of $\lambda_{0}$ as a pole of $R(\lambda)$.
\end{proposition}

The main ideas of results are essentially contained in Sasser \cite{MR0169067}. However, we cannot use the results from those contained in \cite{MR0169067} directly. So we will use the ideas to prove a Ruelle-Perron-Frobenius theorem which applies in our case.

\begin{definition}
	Let $X$ be a real topological vector space a set $\c \subseteq X$ is called a \emph{cone} if
	\begin{enumerate}
		\item 
		$\c$ is convex
		
		\item 
		If $x \in \c$, then $\lambda x \in \c$ for all $\lambda \geq 0$.
	\end{enumerate}
	A cone is called \emph{proper} if $-\c \cap \c = \set{0}$. A cone is called \emph{closed} if it is a closed set in the topology of $X$.
\end{definition}

Throughout this section we will assume that $X$ is a real Banach space ordered by a closed proper cone $\c$.

\begin{definition}\label{defn: sufficient}
	A set $\SS \subseteq X^{\ast}$ is called \emph{sufficient} for $\c$ if 
	\[ \c=\set{x \in X:\inn{x,s}\geq 0 \text{ for all }s \in \SS}. \]
\end{definition}

\begin{definition}\label{defn: semi-positive}
	Let $L:X \to X$ be linear and bounded, $\c \subseteq X$ be a closed proper cone with non-empty interior, and $\SS$ be a sufficient set for $\c$. We say that $L$ is \emph{semi-positive} with respect to $\c$ if
	\begin{enumerate}
		\item 
		$L\c \subseteq \c$.
		
		\item 
		For all $x \in \c$ one of the following is true:
		\begin{enumerate}
			\item 
			There exists $N(x)$ such that $L^{n}x \in \interior(\c)$ for all $n \geq N(x)$.
			
			\item 
			For any $s\in \SS$, the sequence $\set{\inn{L^{n}x,s}}_{n=1}^{\infty}$ converges to $0$.
		\end{enumerate}
	\end{enumerate}
\end{definition} 

The main goal of this subsection is to prove the following theorem.

\begin{theorem}\label{abstractPFThm-Primtive}
	Let $X$ be a real Banach space ordered by a proper closed cone $\c$, with non-empty interior and $L:X \to X$ be a quasi-compact bounded linear operator which is semi-positive with respect to $\c$. Then
	\begin{enumerate}
		\item 
		There exist $u \in \interior(\c)$ and $u^{\ast}\in \c^{\ast}$ such that $Lu=\rho(L)u$, $L^{\ast}u^{\ast}=\rho(L)u^{\ast}$, and $\inn{u,u^{\ast}}=1$.
		
		\item 
		If $\lambda\in \spec(L)$ with $\abs{\lambda}=\rho(L)$, then $\lambda=\rho(L)$.
		
		\item 
		$L$ can be written as
		\[ L=\rho(L)(P+S) \]
		where $Px = \inn{x,u^{\ast}}u$, $S$ is a bounded linear operator with $\rho(S)<1$ and $PS=SP=0$.
		
		\item 
		There exist constants $C>0$ and $0<\gamma < 1$ such that for any $x \in X$
		\[ \norm{\rho(L)^{-n}L^{n}x - \inn{x,u^{\ast}}u}\leq C\norm{x}\gamma^{n} \]
		for all $n \geq 0$.
	\end{enumerate}
\end{theorem}

\begin{lemma}
	Suppose that $\c$ is a closed proper cone and $\SS$ is a sufficient set for $\c$. If $\inn{x,s}=0$ for all $s \in \SS$, then $x=0$.
\end{lemma}
\begin{proof}
	Notice that if $\inn{x,s}=0$ for all $s$, then $x \in \c$. On the other hand $\inn{-x,s}=0$ for all $s \in \SS$, so $-x \in \c$. As $\c$ is proper this implies that $x=0$.
\end{proof}

One may in general be concerned that a cone may not have a sufficient set. Note that every cone has a natural candidate for a sufficient set. Define
\[ \c^{\ast}=\set{x^{\ast}\in X^{\ast}: \inn{x,x^{\ast}}\geq 0 \text{ for all }x \in \c}. \]

\begin{lemma}\label{conepropertieslemma}
	Suppose that $\c$ is a closed proper cone with non-empty interior. Then
	\begin{enumerate}
		\item 
		$x \in \c$ if and only if $\inn{x,x^{\ast}}\geq 0$ for all $x^{\ast}\in \c$; that is $\c^{\ast}$ is a sufficient set for $\c$.
		
		\item 
		We have
		\[ \interior(\c)=\set{x \in X: \inn{x,x^{\ast}}>0 \text{ for all }x^{\ast}\in \c^{\ast} }. \]
	\end{enumerate}
\end{lemma}
\begin{proof} See \cite[Proposition 4.11]{MR2083432} for (1).
 
For (2), we will show inclusions of each set into another.
First, suppose that $x \in \interior(\c)$. Take $\delta>0$ such that $B(x,\delta)\subseteq \c$. If $\norm{y}=\delta/2$, then 
		\[ \norm{(x+y)-x}=\norm{y}<\delta \text{ and }\norm{(x-y)-x}=\norm{y}<\delta ,\]
and this	 implies that $x-y,x+y \in \c$. Thus for any $x^{\ast} \in \c^{\ast}$,
		\[0 \leq \inn{x+y,x^{\ast}}=\inn{x,x^{\ast}}+\inn{y,x^{\ast}} \text{ and }0 \leq \inn{x-y,x^{\ast}}=\inn{x,x^{\ast}}-\inn{y,x^{\ast}} \]
		which implies 
		\[ -\inn{x,x^{\ast}}\leq \inn{y,x^{\ast}}\leq \inn{x,x^{\ast}}. \]
		Then we have
		\begin{equation}\label{eq:interiornorm}
		\norm{x^{\ast}}=2/\delta\sup_{\norm{y}= \delta/2}\abs{\inn{y,x^{\ast}}}\leq 2/\delta \inn{x,x^{\ast}}.
		\end{equation}
		In particular if $x^*\neq 0$, then $\inn{x,x^{\ast}}>0$.
		
		For the other inclusion, suppose that 
		\[x \in \set{x \in X: \inn{x,x^{\ast}}>0 \text{ for all }x^{\ast} \in \c^{\ast}}. \]
		Let $u \in \interior(\c)$ and take $\delta>0$ such that $B(u , \delta)\subseteq \interior(\c)$. Setting
		\[ S=\set{x^{\ast}\in \c^{\ast}: \inn{u , x^{\ast}}=1},\]
		notice that $S$ is closed in the weak*-topology and by \eqref{eq:interiornorm} we have that $S \subseteq \ol{B(0, 2/ \delta)}^{\norm{\cdot}}$. Thus $S$ is weak*-compact by Banach-Alaoglu, and we can take $\eta>0$ such that $\inn{x,x^{\ast}}>\eta$ for all $x^{\ast}\in S$. Then for any $y\in X$ with $\norm{x-y}<\eta\delta/4$ and any $x^{\ast}\in S$, we have
		\[ \abs{\inn{x,x^{\ast}} - \inn{y,x^{\ast}}} \leq \norm{x-y}\norm{x^{\ast}}\leq 2\norm{x-y}/\delta \leq \eta/2. \]
		Hence $\inn{y,x^{\ast}}>0$. As $\c$ is defined as
		\[ \c=\set{x \in X:\inn{x,x^{\ast}} \geq 0 \text{ for all } x^{\ast}\in S}, \]
		we have that $y \in \c$. This shows that $x \in\interior(\c)$. 
\end{proof}

Denote by $\widetilde{X}$ the complexification of $X$. That is,
\[ \widetilde{X}:=\set{x+iy:x,y \in X}. \]
With addition and scalar multiplication defined in the natural way, this becomes a complex vector space. The function
\[ \norm{x+iy}_{\widetilde{X}}:= \sup\set{\norm{x\cos \theta + y\sin \theta }:\theta \in [0,2\pi]} \]
defines a norm on $\widetilde{X}$ making it into a Banach space. 

Any bounded linear operator $T:X \to X$ extends to a bounded linear operator $\widetilde{T}:\widetilde{X}\to \widetilde{X}$ and $\norm{T}_{X}=\norm{\wt{T}}_{\widetilde{X}}$. For a bounded linear operator $T$, we define $\spec(T):=\spec(\widetilde{T})$. Finally the dual of the complexification of $X$ is isomorphic to the complexification of $X^{\ast}$ and elements $x^{\ast}\in X^{\ast}$ act on $\widetilde{X}$ in the natural way by $\inn{x+iy,x^{\ast}}:=\inn{x,x^{\ast}}+i\inn{y,x^{\ast}}$.

Let $\widetilde{L}:\widetilde{X} \to \widetilde{X}$ be the extension of $L$ to $\widetilde{X}$. As $L$ is quasi-compact, any $\lambda_{0}\in \spec(L):=\spec(\widetilde{L})$ with $\abs{\lambda_{0}}=1$ is an isolated point of the spectrum and a pole of the resolvent. Thus $R(z)=(zI-\widetilde{L})^{-1}$ has an expansion
\[ R(z)=\sum_{k=1}^{n}\frac{(\lambda_{0}I-\widetilde{L})^{k-1} P_{\lambda_{0}}}{(z - \lambda_{0})^{k}} + \sum_{k=0}^{\infty}(z-\lambda_{0})^{k}A_{k} \]
which is valid for $0<\abs{z-\lambda_{0}}<\delta$ for any $\delta>0$ with $B(\lambda_{0},\delta)\cap \spec(L)=\set{\lambda_{0}}$. Note that $\displaystyle \sum_{k=0}^{\infty}(z-\lambda_{0})^{k}A_{k}$ converges in the operator norm topology, $n$ is the order of $\lambda$ (by Proposition \ref{prop:index-order}, $n$ is the smallest integer such that $\ker (\lambda_{0}I-\widetilde{L})^{n}=\ker(\lambda_{0}I-\widetilde{L})^{n+1}$), and $P_{\lambda_{0}}$ is the spectral projection onto the subspace $\ker (\lambda_{0}I-\widetilde{L})^{n}$.

Next two results due to Sasser \cite{MR0169067} describes the properties of quasi-compact operator preserving a closed cone.

\begin{proposition}\label{indexprop}
	Suppose that $L:X\to X$ is quasi-compact, $\c$ is a proper closed cone, $L\c\subseteq \c$ and $\rho(L)=1$. Then $1 \in \spec(L)$. 
	Moreover if $\lambda_{0}\in \spec(L)$ and $\abs{\lambda_{0}}=1$, then the order of $\lambda_{0}$ is at most the order of $1$.
\end{proposition}
\begin{proof}
	This is \cite[Theorem 2]{MR0169067}.
\end{proof}

\begin{lemma}\label{eigenvectors}
	Suppose that $L:X\to X$ is quasi-compact, $\c$ is a proper closed cone,  $L\c\subseteq \c$, and $\rho(L)=1$. Then there exist $u \in \c$ ($u \neq 0$) and $u^{\ast}\in \c^{\ast}$ ($u^{\ast}\neq 0$) such that $Lu=u$ and $L^{\ast}u^{\ast}=u^{\ast}$.
\end{lemma}
\begin{proof}
	This is \cite[Theorem 3]{MR0169067}.
\end{proof}

\begin{lemma}\label{lem:pulltoboundary}
		If $z \in \interior(\c)$ and $x \in X$ with $x \neq 0$, then there exists $t \in \R\setminus\set{0}$ such that $z+tx \in \partial \c$.
\end{lemma}
\begin{proof}
By replacing $x$ with $-x$ if necessary, we may assume that $x \notin \c$. Set
	\[ t_{0}=\sup\set{t:z+tx \in \c} \]
	and notice that as $z \in \interior(\c)$ we have that $t_{0}>0$. As $x \notin \c$, we may choose $x^{\ast} \in \c^{\ast}$ such that $\inn{x,x^{\ast}}<0$ by Lemma \ref{conepropertieslemma} (1). Notice that if $t>-\inn{z,x^{\ast}}/\inn{x,x^{\ast}}$, then $z+tx \notin \c$. Thus $t_{0}<\infty$. As $\c$ is closed we have that $z+t_{0}x \in \c$, but it is clear that $z+t_{0}x\notin \interior(\c)$. Hence, $z+t_{0}x \in \partial \c$.
\end{proof}

The following proposition shows that with the added assumption of semi-positivity with respect to a proper closed cone, extra properties (such as some of the listed properties in Theorem \ref{abstractPFThm-Primtive}) of the operator can be established.

\begin{proposition}\label{irreducibleprop}
	Suppose that $L:X\to X$ is quasi-compact, $\c$ is a proper closed cone, $L$ is semi-positive with respect to $\c$, and $\rho(L)=1$. Then the following are true:
	\begin{enumerate}
		
		\item 
		If $w \in \c \setminus\{0\}$ with $Lw=w$, then $w \in \interior(\c)$.
		
		\item 
		The order of the eigenvalue $1$ is one.
		
		\item 
		$\widetilde{L}$ can be written as
		\[ \widetilde{L}=\sum_{j=1}^{m}\lambda_{j}P_{j}+S \]
		where $\set{\lambda_{j}}$ is the collection of eigenvalues of $\L$ with modulus $1$, $P_{j}^{2}=P_{j}$, $P_{j}S=SP_{j}=0$, $P_{i}P_{j}=0$ for $i\neq j$, and $\rho(S)<1$.
		
		\item 
		The dimension of the eigenspace corresponding to $1$ is one-dimensional.
	\end{enumerate}
\end{proposition}
\begin{proof}~
We can modify the proof of \cite[Theorem 4]{MR0169067} to fit our assumptions. 
\begin{enumerate}
	\item 
	It is clear that there is some $s\in \SS$ such that $\set{\inn{L^{n}w,s}}_{n=1}^{\infty}$ does not converge to $0$. So by the assumption that $L$ is semi-positive, it must be that $L^{n}w\in \interior(\c)$ for some $n$, but of course $L^{n}w=w$. Hence, $w \in \interior(\c)$.
	
	\item
	Let $n$ be the order of $1$ and set
	\[ \Gamma = (I-\widetilde{L})^{n-1}P_{1}. \]
	From the Laurent expansion of $R(z)$ about $1$ we can see that
	\begin{align*}
	\Gamma &= \lim_{z \to 1}(z-1)^{n}R(z) \\
	&= \lim_{z \to 1^{+}}(z-1)^{n}R(z)\\
	&= \lim_{z \to 1^{+}}(z-1)^{n}\sum_{k=0}^{\infty}z^{-k-1}\widetilde{L}^{k}.
	\end{align*}
	From this it follows that $\Gamma X \subseteq X$ and $\Gamma \c \subseteq \c$. Moreover 
	\[ (I-\widetilde{L})\Gamma = (I-\widetilde{L})^{n}P_{1}=0. \]
	This combined with the previous equation implies that $\Gamma\widetilde{L}=\widetilde{L}\Gamma = \Gamma$. 
	
	Take $x\in \c$ such that $\Gamma x \neq 0$. Notice that $L\Gamma x = \Gamma x$ and $\Gamma x \in \c$, and thus $\Gamma x \in \interior(\c)$ by part (1). With $u^* \in \c^*$ from Lemma \ref{eigenvectors}, we have
	\begin{align*}
	0<\inn{\Gamma x, u^{\ast}}&=\lim_{z\to 1^{+}}(z-1)^{n}\sum_{k=0}^{\infty}z^{-k-1}\inn{\widetilde{L}^{k}x,u^{\ast}}\\
	&=\inn{x,u^{\ast}}\lim_{z\to 1^{+}}(z-1)^{n}\sum_{k=0}^{\infty}z^{-k-1}\\
	&=\inn{x,u^{\ast}}\lim_{z\to 1^{+}}(z-1)^{n}\frac{z^{-1}}{1-z^{-1}}\\
	&=\inn{x,u^{\ast}}\lim_{z\to 1^{+}}(z-1)^{n-1}.
	\end{align*}
	Therefore $n =1$.
	
	\item 
We enumerate the points in $\spec(L)$ with modulus $1$ as $\set{\lambda_{i}}_{i=1}^{m}$ where $\lambda_{1}=1$. Letting $P_{i}=P_{\lambda_{i}}$ be the spectral projection onto $\ker (\lambda_{i}I-\widetilde{L})$, notice from Proposition \ref{indexprop} that the order of $\lambda_{i}$ is one. Setting $\mathcal{P}=\sum_{i}P_{i}$ and writing
	\[ \widetilde{L}= \sum_{i}\widetilde{L}P_{i} + \widetilde{L}(I-\mathcal{P}), \]
	it is easily seen that $\widetilde{L}P_{i} = \lambda_{i}P_{i}$. 
	
	We then take $S = \widetilde{L}(I-\mathcal{P})$. As $I-\mathcal{P}$ is the spectral projection onto $\spec(L)\setminus\set{\lambda_{i}}_{i=1}^{m}$ we have that $\spec(\widetilde{L}|_{\ran(I-\mathcal{P})})=\spec(L)\setminus \set{\lambda_{i}}_{i=1}^{k}$. In particular, $\rho(S)<1$.
	
	\item 
	Let $u\in \c$ be as in Lemma \ref{eigenvectors}. Supposing that $Lx=x$ for some $x\in X$, we take $t$ as in Lemma \ref{lem:pulltoboundary} such that $u+tx \in \partial \c$. Since $L(u+tx)=u+tx$, it follows from part (1) that $u+tx=0$; that is $x$ is a scalar multiple of $u$.
	
	 Now if $x+iy \in \wt{X}$ is such that $L(x+iy)=x+iy$, then $Lx=x$ and $Ly=y$. Hence, $x+iy$ is a scalar multiple of $u$, as required.
\end{enumerate}
\end{proof}

\begin{lemma}\label{lem:interiorelementclosure}
	Suppose that $L:X\to X$ is quasi-compact, $\c$ is a proper closed cone, $L$ is semi-positive with respect to $\c$, and $\rho(L)=1$. Then for any $z \in \interior(\c)$ we have that $\ol{\set{L^{n}z}}\subseteq \interior(\c)$.
\end{lemma}
\begin{proof}
	Let $u\in \c$ be as in Lemma \ref{eigenvectors}. Notice that for any $x^{\ast}\in \c^{\ast}$, we have that 
	\[ \inn{u,x^{\ast}}= \inn{L^{n}u,x^{\ast}}\leq \norm{u}\norm{(L^{\ast})^{n}x^{\ast}}. \]
	Suppose $z \in \interior(\c)$ and take $\delta>0$ such that $B(z,\delta)\subseteq \c$. Then 
$$
	\inn{L^{n}z, x^{\ast}} = \inn{z, \frac{(L^{\ast})^{n}x^{\ast}}{\norm{(L^{\ast})^{n}x^{\ast}}}}\norm{(L^{\ast})^{n}x^{\ast}}\geq \frac{\delta}{2}\norm{(L^{\ast})^{n}x^{\ast}}\geq \frac{\delta}{2}\cdot \frac{\inn{u,x^{\ast}}}{\norm{u}}
$$
where the first inequality is due to \eqref{eq:interiornorm}.	In particular, if $y \in \ol{\set{L^{n}z}}$, then $\inn{y,x^{\ast}}>0$. Thus by Lemma \ref{conepropertieslemma}, $y \in \interior(\c)$.
\end{proof}

\begin{proposition}\label{primitiveperspec}
	Suppose that $L:X\to X$ is quasi-compact, $\c$ is a proper closed cone, $L$ is semi-positive with respect to $\c$, and $\rho(L)=1$. Then $1$ is the only element of $\spec(L)$ of modulus $1$.
\end{proposition}
\begin{proof}
Let $u$ be as in Lemma \ref{eigenvectors}. Suppose that $e^{i\theta}$ belongs to $\spec(L)$ and let $x+iy \neq 0$ be its corresponding eigenvector; that is, $L(x+iy)=e^{i\theta}(x+iy)$. Take a sequence $n_{k}$ such that $e^{in_{k}\theta}\xrightarrow{k \to \infty}1$, then 
	\[ L^{n_{k}}(x+iy) \xrightarrow{k \to \infty}x+iy .\]
	Thus 
	\[ L^{n_{k}}x \xrightarrow{k \to \infty}x \text{ and }L^{n_{k}}y \xrightarrow{k \to \infty}y. \]
	If $x \neq 0$, then by Lemma \ref{lem:pulltoboundary} we can take $t\neq 0$ such that $u+tx \in \partial \c$. Notice that
	\[ L^{n_{k}}(u+tx) = u + t L^{n_{k}}x \xrightarrow{k \to \infty}u+tx. \]
	By the assumption that $L$ is semi-positive with respect to $\c$, either $u+tx=0$ or there exists $N$ such that $L^{n}(u+tx)\in \interior(\c)$ for all $n\geq N$. 
	
It turns out that we must have $u+tx=0$. This is because if
$u+tx \neq 0$, then we would have $u+tx \in \ol{\set{L^{n}(L^{N}(u+tx))}_{n=1}^{\infty}} \subseteq \interior(\c)$ by Lemma \ref{lem:interiorelementclosure}. Then $u+tx \in \partial \c$, and at the same time, $u+tx \in \interior(\c)$, which is absurd. 
	
	The same argument shows that if $y \neq 0$, then $y$ is a scalar multiple of $u$. Therefore, $x+iy$ is a scalar multiple of $u$ and $e^{i \theta}=1$.
\end{proof}

\begin{proof}[Proof of Theorem \ref{abstractPFThm-Primtive}]~
	\begin{enumerate}
		\item 
		This follows by applying Lemma \ref{eigenvectors} as well as Proposition \ref{irreducibleprop} (1) and (2) to $\rho(L)^{-1}L$.
		
		\item 
		This follows by an application of Proposition \ref{primitiveperspec} to $\rho(L)^{-1}L$.
		
		\item 
		This follows from Proposition \ref{irreducibleprop} (4) and part (2) of this theorem.
		
		\item 
		Notice that by part (3) we have that $\rho(L)^{-n}L^{n}= P + S^{n}$. Thus
		\begin{align*}
		\norm{\rho(L)^{-n}L^{n}x - \inn{x,u^{\ast}}u} = \norm{Px+S^{n}x - \inn{x,u^{\ast}}u}=\norm{S^{n}x}.
		\end{align*}
		Take $\rho(S)<\gamma < 1$, then 
		\[ \norm{S^{n}x}\leq \frac{\norm{S^{n}}}{\gamma^{n}}\norm{x}\gamma^{n} .\]
		As $\norm{S^{n}}$ has an exponential growth rate of $\log \rho(S)$ by Gelfand's formula, we have that $\displaystyle \frac{\norm{S^{n}}}{\gamma^{n}}$ tends to $0$, and in particular, is bounded by some constant $C>0$.
	\end{enumerate}
\end{proof}

\subsection{Direct Sums of Banach Spaces and spectral properties of non-negative matrices}

\begin{proposition}\label{prop:directsumproperties}
~
	\begin{enumerate}
		\item 
		Let $\set{X_{i}}_{i=1}^{n}$ be a finite collection of Banach spaces. Define the vector space
		\[ \bigoplus_{i=1}^{n}X_{i}=\set{(x_{1}, x_{2}, \ldots , x_{n}) :x_{i}\in X_{i} } \]
		with addition and scalar multiplication being coordinate wise. Then the function
		\[ \norm{(x_{i})_{i=1}^{n}} = \max_{i}\norm{x_{i}}_{X_{i}} \]
		defines a norm on $\bigoplus_{i=1}^{n}X_{i}$ which makes it a Banach space.
		
		\item 
		Suppose that $\set{X_{i}}_{i=1}^{n}$ is a collection of Banach spaces and that $\set{A_{ij}}_{i=1,j=1}^{n,n}$ is a collection of bounded linear operators $A_{ij}:X_{j}\to X_{i}$. The function
		\[ \A(x_{i})_{i=1}^{n}:= \left(\sum_{j=1}^{n}A_{ij}x_{j}\right)_{i=1}^{n} \]
		defines a bounded linear operator $\A:\bigoplus_{i=1}^{n}X_{i}\to \bigoplus_{i=1}^{n}X_{i}$. We will denote this operator by $[A_{ij}]$.
		
		\item 
		Suppose that $\A: \bigoplus_{i}X_{i} \to \bigoplus_{i}X_{i}$ is a bounded linear operator. Then there exists a collection of bounded linear operators $A_{ij}:X_{j}\to X_{i}$ such that $\A = [A_{ij}]$.
		
		\item 
		Let $\A = [A_{ij}]$ and $\B=[B_{ij}]$, then $\A\B = [\sum_{k}A_{ik}B_{kj}]$.
	\end{enumerate}
\end{proposition}
\begin{proof} ~
	\begin{enumerate}
		\item 
		This can be found in \cite[Chapter 3 proposition 4.4]{MR1070713}.
		
		\item 
		That the function $\A$ is linear can be verified by computation. To see that $\A$ is bounded, we note that 
		\begin{align*}
		\norm{\sum_{j=1}^{n}A_{ij}x_{j}}_{X_{i}}\leq \sum_{j=1}^{n}\norm{A_{ij}}_{\op}\norm{x_{j}}_{X_{j}}\leq \norm{(x_{i})}\sum_{j=1}^{n}\norm{A_{ij}}_{\op}.
		\end{align*}
		Therefore $\norm{\A}_{\op}\leq \max\limits_{i}\displaystyle\sum_{j=1}^{n}\norm{A_{ij}}_{\op}$.
		
		\item 
		
	For each $1\leq i\leq n$, define $P_{i}:\bigoplus_{k}X_{k} \to X_{i}$ by
\[ P_{i}(x_{1}, x_{2}, \ldots , x_{n})=x_{i} \]
and $:\ol{P}_{i}:X_{i} \to \bigoplus_{k}X_{k}$ by
\[ \ol{P}_{i}x = (0,\ldots, 0,\underset{\text{$i$th place}}{x},0,\ldots ,0). \]
Setting $A_{ij}=P_{i}\A \ol{P}_{j}$, it can be verified that $\A=[A_{ij}]$.
	
%		For each $1\leq i\leq n$ define $P_{i}:\bigoplus_{i}X_{i} \to X_{i}$ by
%		\[ P_{i}(x_{1}, x_{2}, \ldots , x_{n})=x_{i}. \]
%		Setting $A_{ij}=P_{i}\A P_{j}$, it is a simple computation to verify that $\A=[A_{ij}]$.
		
		\item 
		Notice that
		\begin{align*}
		\A \B(x_{i})_{i=1}^{n}&= \A\left(\sum_{j=1}^{n}B_{ij}x_{j}\right)_{i=1}^{n}\\
		&=\left(\sum_{k=1}^{n}A_{ik}\sum_{j=1}^{n}B_{kj}x_{j}\right)_{i=1}^{n}\\
		&=\left(\sum_{j=1}^{n}\sum_{k=1}^{n}A_{ik}B_{kj}x_{j}\right)_{i=1}^{n}.
		\end{align*} 
		Therefore $\A \B= [\sum_{k}A_{ik}B_{kj}]$.
	\end{enumerate}
\end{proof}

When working with an operator $\A$ acting on a direct sum, we will sometimes write
\[ \A = \bmat{A_{11}& \cdots & A_{1n} \\ \vdots & & \vdots \\ A_{n1} & \cdots & A_{nn}} \]
to mean that $\A=[A_{ij}]$. By Proposition \ref{prop:directsumproperties}, the operation of multiplication of operators is compatible with the typical formulas for matrix multiplication, justifying such notation.

\begin{lemma}
	Let $\set{X_{i}}_{i=1}^{n}$ be a collection of Banach space and $\A:\bigoplus_{i}X_{i}\to \bigoplus_{i}X_{i}$. Suppose that $A$ is block upper triangular 
	\[ \A=\bmat{A_{11}&A_{12}&\cdots & A_{1n} \\ 0& A_{22} & & \vdots \\ \vdots &  & \ddots & \\ 0 & \cdots & 0 &A_{nn}} \]
	\begin{enumerate}
		\item 
		$A$ is invertible if and only if $A_{ii}$ is invertible for all $1 \leq i \leq n$.
		
		\item We have
		\[ \spec(A) = \bigcup_{i=1}^{n}\spec(A_{ii}). \]
	\end{enumerate}
\end{lemma}
\begin{proof}~
	\begin{enumerate}
		\item 
		$(\implies)$ This direction is clear; use the invertibility of $A$ and apply it to vectors of the form $(0,0, \ldots, x_{i}, \ldots , 0)$ to define the inverse for $A_{ii}$.
		
\noindent		$(\impliedby)$ The proof is by induction on $n$ (the number of Banach spaces). It is clear that the result holds for $n=1$. Supposing that the result holds for $n-1$, we can think of $\bigoplus_{i=1}^{n}X_{i}$ as $\left(\bigoplus_{i=1}^{n-1}X_{i}\right) \oplus X_{n}$ and write $\A$ as a matrix with respect to this direct sum		
\[ \A = \bmat{\wt{A}_{11}& \wt{A}_{12} \\ 0 & A_{nn}} .\]
		By assumption $A_{nn}$ is invertible, and by the induction hypothesis $\wt{A}_{11}$ is invertible. Now one can check using Proposition \ref{prop:directsumproperties} (4) that
		\[ \bmat{\wt{A}^{-1}_{11}&-\wt{A}_{11}^{-1}\wt{A}_{12}A_{nn}^{-1} \\ 0 & A_{nn}^{-1}} \]
		defines an inverse for $A$.
		
		\item 
		This follows from part $(1)$.
	\end{enumerate}
\end{proof}

\begin{theorem}[Perron-Frobenius]\label{thm:irreduciblePF}
	Suppose that $T$ is an irreducible non-negative matrix. There exist a number $h\geq 1$, called the \emph{period} of $A$, and a permutation matrix $P$ such that $PAP^{\ast}$ can be written as a block matrix of the following form
	\[ PTP^{\ast}=\bmat{0& T_{12} & 0 & \cdots & 0 \\ \vdots &0 & T_{23}& \cdots &0\\ \vdots &  &   & \ddots & \vdots \\  0 & 0& \cdots & & T_{h-1h}\\ T_{h1} & 0& \cdots & & 0} \]
	where the diagonal blocks are square. Moreover,
	\[ PT^{h}P^{\ast}=\bmat{T_{12}T_{23}T_{34}\cdots T_{h1} & 0 &0 &\cdots& 0 \\ 0 & T_{23}T_{34}\cdots T_{h1}T_{12} &0 &\cdots& 0 \\ \vdots  & &\; \; \ddots & & \vdots \\0 & \cdots & & & T_{h1}T_{12}T_{23}\cdots T_{h-1h}} \]
	and each diagonal block is primitive.
\end{theorem}
\begin{proof}
For the proof we refer the reader to \cite[Section 8.4]{MR2978290} and the references therein.
\end{proof}

\begin{proposition}\label{prop:rotationsymmetry}
	Suppose that 
	\[ \A = \bmat{0&A_{12} & 0 & \cdots & 0 \\ \vdots &0 & A_{23}& \cdots &0\\ \vdots &  &   & \ddots & \vdots \\  0 & 0& \cdots & & A_{h-1h}\\ A_{h1} & 0& \cdots & & 0}. \]
	That is $\A=[A_{ij}]$ where $A_{ij}=0$ unless $j\equiv i+1 \mod h$. Then for each $0\leq k\leq h-1$, there exists an isometric isomorphism $\D:\bigoplus_{i=1}^{h}X_{i}\to \bigoplus_{i=1}^{h}X_{i}$ such that $\D^{-1}\A\D = e^{\frac{2k \pi i}{h}}A$.
\end{proposition}
\begin{proof}
	Define
	\[ \D=\bmat{e^{\frac{2k \pi i}{h}}& 0 &  & \cdots  &0 \\ 0&e^{\frac{2k \pi i}{h}2}&0&  &0 \\ 0&0& e^{\frac{2k \pi i}{h}3}&\cdots & 0\\ \vdots &  & \ddots &  & \vdots \\ 0& &  &e^{\frac{2k \pi i}{h}(h-1)}&0 \\0&\cdots & & 0 &1}. \]
	Then
	\[ \D^{-1}\A\D = \bmat{0&e^{\frac{-2k \pi i}{h}}A_{12}e^{\frac{2k \pi i}{h}2} & 0 & \cdots & 0 \\ \vdots &0 & e^{\frac{-2k \pi i}{h}2}A_{23}e^{\frac{2k \pi i}{h}3}& \cdots &0\\ \vdots &  &   & \ddots & \vdots \\  0 & 0& \cdots & & e^{\frac{-2k \pi i}{h}(h-1)}A_{h-1h}  \\ A_{h1}e^{\frac{2k \pi i}{h}} & 0& \cdots & & 0}=e^{\frac{2k \pi i}{h}}\A \]
\end{proof}

\bibliographystyle{amsalpha}
\bibliography{operator_final}
\end{document}

%% file: general-preamble.tex
\numberwithin{equation}{section}
\numberwithin{figure}{section}
\usepackage{enumitem}		% customizable list environments
 \let\footnote=\endnote
\@ifundefined{lettrine}{\usepackage{lettrine}}{}
\theoremstyle{definition}
\newtheorem{thmx}{Theorem}
\def\ep{\varepsilon}
\newcommand{\e}{\varepsilon}
\newcommand{\vp}{\varphi}

\newcommand{\hol}{H\"older }

\newcommand{\hx}{\hat{x}}
\newcommand{\hy}{\hat{y}}

\newcommand{\hm}{\hat{m}}
\newcommand{\hnu}{\hat{\nu}}
\newcommand{\heta}{\hat{\eta}}
\newcommand{\hmu}{\hat{\mu}}

\newcommand{\Wloc}{\mathcal{W}_{\text{loc}}}
\newcommand{\Sig}{\Sigma_T}

\newcommand{\glr}{\text{GL}_d(\R)}

\def\id{\text{id}}

\theoremstyle{definition}
\newtheorem{theorem}{Theorem}[section]
\newtheorem{lemma}[theorem]{Lemma}

\newtheorem{corollary}[theorem]{Corollary}
\newtheorem{proposition}[theorem]{Proposition}
\newtheorem{definition}[theorem]{Definition}

\newcommand{\bmat}[1]{\begin{bmatrix} #1 \end{bmatrix}}

\newcommand{\inn}[1]{\left\langle #1 \right\rangle}
\newcommand{\set}[1]{\left\{ #1 \right\}}
\newcommand{\abs}[1]{\left| #1 \right|}

\newcommand{\norm}[1]{\left \| #1 \right \|}
\newcommand{\wt}[1]{\widetilde{ #1}}

\newcommand{\ol}[1]{\overline{#1}}

\def\[#1\]{\begin{align*}#1\end{align*}}

\DeclareMathOperator{\spn}{span}

\DeclareMathOperator{\var}{Var}

\DeclareMathOperator{\op}{op}

\DeclareMathOperator{\spec}{spec}
\DeclareMathOperator{\res}{res}
\DeclareMathOperator{\ran}{Ran}
\DeclareMathOperator{\interior}{int}
\DeclareMathOperator{\geom}{geom}
\DeclareMathOperator{\alg}{alg}

\renewcommand{\mod}{\; \mathrm{mod} \;}

\newcommand{\inv}{\text{inv}}
\newcommand{\Lip}{\text{Lip}}
\def\a{\alpha}
\def\A{\mathcal{A}}

\def\b{\beta}
\def\B{\mathcal{B}}

\def\c{\mathcal{C}}  
\def\C{\mathbb{C}}
\def\CC{\mathcal{C}}

\def\d{\delta}
\def\D{\mathcal{D}}     

\def\E{\mathbb{E}}

\newcommand{\h}{\hat}

\def\I{\mathrm{I}}

\def\L{\mathcal{L}}

\def\M{\mathcal{M}}

\def\n{\mathcal{N}}
\def\N{\mathbb{N}}

\def\P{\mathbb{P}}

\def\R{\mathbb{R}}

\def\s{\sigma}
\def\S{\Sigma} 
\def\SS{\mathcal{S}}

\def\U{\mathcal{U}}

\def\w{\mathcal{W}}

\def\Z{\mathbb{Z}}

%% file: operator_final.bbl
\providecommand{\bysame}{\leavevmode\hbox to3em{\hrulefill}\thinspace}
\providecommand{\MR}{\relax\ifhmode\unskip\space\fi MR }
\providecommand{\MRhref}[2]{%
  \href{http://www.ams.org/mathscinet-getitem?mr=#1}{#2}
}
\providecommand{\href}[2]{#2}
\begin{thebibliography}{DKP20}

\bibitem[AV07]{avila2007simplicity}
Artur Avila and Marcelo Viana, \emph{Simplicity of lyapunov spectra: proof of
  the zorich-kontsevich conjecture}, Acta mathematica \textbf{198} (2007),
  no.~1, 1--56.

\bibitem[Bal00]{MR1793194}
Viviane Baladi, \emph{Positive transfer operators and decay of correlations.},
  Advanced Series in Nonlinear Dynamics, 16., World Scientific Publishing Co.,
  Inc., River Edge, NJ, 2000.

\bibitem[BBB18]{backes2018continuity}
Lucas Backes, Aaron Brown, and Clark Butler, \emph{Continuity of lyapunov
  exponents for cocycles with invariant holonomies}, Journal of Modern Dynamics
  \textbf{12} (2018), no.~1, 223--260.

\bibitem[BL85]{MR886674}
Philippe Bougerol and Jean Lacroix, \emph{Products of random matrices with
  applications to {S}chr\"odinger operators}, Progress in Probability and
  Statistics, vol.~8, Birkh\"auser Boston, Inc., Boston, MA, 1985.

\bibitem[BNV17]{MR3667994}
Carlos Bocker-Neto and Marcelo Viana, \emph{Continuity of {L}yapunov exponents
  for random two-dimensional matrices}, Ergodic Theory and Dynamical Systems
  \textbf{37} (2017), no.~5, 1413--1442.

\bibitem[Bou88]{bougerol1988theoremes}
Philippe Bougerol, \emph{Th{\'e}or{\`e}mes limite pour les syst{\`e}mes
  lin{\'e}aires {\`a} coefficients markoviens}, Probability Theory and related
  fields \textbf{78} (1988), no.~2, 193--221.

\bibitem[Bow75]{bowen1975ergodic}
Rufus Bowen, \emph{Equilibrium states and the ergodic theory of anosov
  diffeomorphisms}, Lecture Notes in Mathematics, vol. 470, Springer-Verlag,
  1975.

\bibitem[BQ16]{benoist2016random}
Yves Benoist and Jean-Francois Quint, \emph{Random walks on reductive groups},
  Random Walks on Reductive Groups, Springer, 2016, pp.~153--167.

\bibitem[BV04]{bonatti2004lyapunov}
Christian Bonatti and Marcelo Viana, \emph{Lyapunov exponents with multiplicity
  1 for deterministic products of matrices}, Ergodic Theory and Dynamical
  Systems \textbf{24} (2004), no.~5, 1295--1330.

\bibitem[BV05]{bochi2005lyapunov}
Jairo Bochi and Marcelo Viana, \emph{The lyapunov exponents of generic
  volume-preserving and symplectic maps}, Annals of mathematics (2005),
  1423--1485.

\bibitem[Con90]{MR1070713}
John~B. Conway, \emph{A course in functional analysis}, second ed., Graduate
  Texts in Mathematics, vol.~96, Springer-Verlag, New York, 1990.

\bibitem[DK16]{duarte2016lyapunov}
Pedro Duarte and Silvius Klein, \emph{Lyapunov exponents of linear cocycles},
  Continuity via large deviations. Atlantis Studies in Dynamical Systems
  \textbf{3} (2016).

\bibitem[DKP20]{duartekleinpoletti}
Pedro Duarte, Silvius Klein, and Maurico Poletti, \emph{personal
  communication}.

\bibitem[DS88]{MR1009162}
Nelson Dunford and Jacob~T. Schwartz, \emph{Linear operators. {P}art {I}},
  Wiley Classics Library, John Wiley \& Sons, Inc., New York, 1988.

\bibitem[FK60]{MR121828}
H.~Furstenberg and H.~Kesten, \emph{Products of random matrices}, Ann. Math.
  Statist. \textbf{31} (1960), 457--469.

\bibitem[FK83]{MR727020}
H.~Furstenberg and Y.~Kifer, \emph{Random matrix products and measures on
  projective spaces}, Israel J. Math. \textbf{46} (1983), no.~1-2, 12--32.

\bibitem[Fur63]{furstenberg1963noncommuting}
Harry Furstenberg, \emph{Noncommuting random products}, Transactions of the
  American Mathematical Society \textbf{108} (1963), no.~3, 377--428.

\bibitem[GLP04]{MR2087783}
Yves Guivarc'h and \'{E}mile Le~Page, \emph{Simplicit\'{e} de spectres de
  {L}yapounov et propri\'{e}t\'{e} d'isolation spectrale pour une famille
  d'op\'{e}rateurs de transfert sur l'espace projectif}, Random walks and
  geometry, Walter de Gruyter, Berlin, 2004, pp.~181--259.

\bibitem[GS19]{gouezel2019}
S\'ebastien Gou\"ezel and Luchezar Stoyanov, \emph{Quantitative pesin theory
  for anosov diffeomorphisms and flows}, Ergodic Theory and Dynamical Systems
  \textbf{39} (2019), no.~1, 159–200.

\bibitem[Gui15]{guivarc2015spectral}
Yves Guivarc’h, \emph{Spectral gap properties and limit theorems for some
  random walks and dynamical systems}, Hyperbolic dynamics, fluctuations and
  large deviations, Proc. Sympos. Pure Math \textbf{89} (2015), 279--310.

\bibitem[Hen93]{MR1129880}
Hubert Hennion, \emph{Sur un th\'{e}or\`eme spectral et son application aux
  noyaux lipchitziens}, Proc. Amer. Math. Soc. \textbf{118} (1993), no.~2,
  627--634.

\bibitem[HJ13]{MR2978290}
Roger~A. Horn and Charles~R. Johnson, \emph{Matrix analysis}, second ed.,
  Cambridge University Press, Cambridge, 2013.

\bibitem[Kat95]{MR1335452}
Tosio Kato, \emph{Perturbation theory for linear operators}, Classics in
  Mathematics, Springer-Verlag, Berlin, 1995.

\bibitem[Kin73]{MR356192}
J.~F.~C. Kingman, \emph{Subadditive ergodic theory}, Ann. Probability
  \textbf{1} (1973), 883--909.

\bibitem[Klo19]{MR3920693}
Beno\^{\i}t~R. Kloeckner, \emph{Effective perturbation theory for simple
  isolated eigenvalues of linear operators}, J. Operator Theory \textbf{81}
  (2019), no.~1, 175--194.

\bibitem[KS13]{kalinin2013cocycles}
Boris Kalinin and Victoria Sadovskaya, \emph{Cocycles with one exponent over
  partially hyperbolic systems}, Geometriae Dedicata \textbf{167} (2013),
  no.~1, 167--188.

\bibitem[Lep00]{leplaideur2000local}
Renaud Leplaideur, \emph{Local product structure for equilibrium states},
  Transactions of the American Mathematical Society \textbf{352} (2000), no.~4,
  1889--1912.

\bibitem[LP82]{MR669072}
\'{E}mile Le~Page, \emph{Th\'{e}or\`emes limites pour les produits de matrices
  al\'{e}atoires}, Probability measures on groups ({O}berwolfach, 1981),
  Lecture Notes in Math., vol. 928, Springer, Berlin-New York, 1982,
  pp.~258--303.

\bibitem[LP89]{page1989regularite}
{\'E}mile Le~Page, \emph{R{\'e}gularit{\'e} du plus grand exposant
  caract{\'e}ristique des produits de matrices al{\'e}atoires ind{\'e}pendantes
  et applications}, Annales de l'IHP Probabilit{\'e}s et statistiques, vol.~25,
  1989, pp.~109--142.

\bibitem[Nau04]{MR2083432}
Fr\'ed\'eric Naud, \emph{Birkhoff cones, symbolic dynamics and spectrum of
  transfer operators}, Discrete Contin. Dyn. Syst. \textbf{11} (2004), no.~2-3,
  581--598.

\bibitem[Ose68]{oseledets1968multiplicative}
Valery~Iustinovich Oseledets, \emph{A multiplicative ergodic theorem.
  characteristic ljapunov, exponents of dynamical systems}, Trudy Moskovskogo
  Matematicheskogo Obshchestva \textbf{19} (1968), 179--210.

\bibitem[Par20]{park2020quasi}
Kiho Park, \emph{Quasi-multiplicativity of typical cocycles}, Communications in
  Mathematical Physics \textbf{376} (2020), no.~3, 1957--2004.

\bibitem[Per92]{MR1158741}
Yuval Peres, \emph{Domains of analytic continuation for the top {L}yapunov
  exponent}, Ann. Inst. H. Poincar\'{e} Probab. Statist. \textbf{28} (1992),
  no.~1, 131--148.

\bibitem[Pir18]{piraino2018weak}
Mark Piraino, \emph{The weak bernoulli property for matrix gibbs states},
  Ergodic Theory and Dynamical Systems (2018), 1--20.

\bibitem[PP90]{MR1085356}
William Parry and Mark Pollicott, \emph{Zeta functions and the periodic orbit
  structure of hyperbolic dynamics}, Ast\'erisque (1990), no.~187-188.

\bibitem[Rat73]{ratner1973central}
Marina Ratner, \emph{The central limit theorem for geodesic flows on
  n-dimensional manifolds of negative curvature}, Israel Journal of Mathematics
  \textbf{16} (1973), no.~2, 181--197.

\bibitem[Sas64]{MR0169067}
D.~W. Sasser, \emph{Quasi-positive operators}, Pacific J. Math. \textbf{14}
  (1964), 1029--1037. \MR{0169067}

\bibitem[Ser19]{sert2019}
Cagri Sert, \emph{Large deviation principle for random matrix products}, Ann.
  Probab. \textbf{47} (2019), no.~3, 1335--1377.

\bibitem[VY19]{viana2019continuity}
Marcelo Viana and Jiagang Yang, \emph{Continuity of lyapunov exponents in the
  c0 topology}, Israel Journal of Mathematics \textbf{229} (2019), no.~1,
  461--485.

\bibitem[You90]{young1990large}
Lai-Sang Young, \emph{Large deviations in dynamical systems}, Transactions of
  the American Mathematical Society \textbf{318} (1990), no.~2, 525--543.

\end{thebibliography}
